\documentclass[a4paper,11pt,reqno]{amsart}
\usepackage{color}
\usepackage{enumerate}
\usepackage{amssymb, amsmath,graphicx}
\usepackage{todonotes}
\usepackage{verbatim}
\usepackage{float}

\theoremstyle{plain}
\newtheorem{theorem}{Theorem}[section]
\newtheorem{corollary}[theorem]{Corollary}
\newtheorem{lemma}[theorem]{Lemma}
\newtheorem{proposition}[theorem]{Proposition}
\newtheorem{definition}[theorem]{Definition}

\newtheorem*{definition*}{Definition}

\theoremstyle{remark}
\newtheorem{remark}[theorem]{Remark}
\newtheorem{example}[theorem]{Example}

\newtheorem*{claim*}{Claim}
\newtheorem*{remark*}{Remark}
\newtheorem*{example*}{Example}
\newtheorem*{notation*}{Notation}

\def\N{{\mathbb N}}

\def\R{{\mathbb R}}

\def\D{{\mathcal D}}

\newcommand{\Dom}{\mathit{Dom}}

\newcommand{\Hil}{\mathcal{H}}
\newcommand{\F}{\mathcal{F}}
\newcommand{\E}{\mathcal{E}}

\newcommand{\Ric}{\mathrm{Ric}}
\newcommand{\scal}{\mathrm{scal}}
\newcommand{\Sec}{\mathrm{Sec}}

\newcommand{\Cpl}{\mathrm{Cpl}}
\newcommand{\Pz}{\mathcal{P}}
\newcommand{\Lip}{\mathrm{Lip}}
\newcommand{\lip}{\mathrm{lip}}

\newcommand{\Cz}{\mathcal{C}}

\begin{document}

\title{Heat Flow on Time-dependent Metric Measure Spaces and super-Ricci Flows}

\author{Eva Kopfer, \ Karl-Theodor Sturm}

\thanks{
The authors gratefully acknowledges  support by the German Research Foundation through the Hausdorff Center for Mathematics and the Collaborative Research Center 1060
as well as support by the European Union through the ERC-AdG ``RicciBounds''.
They also thank the MSRI for hospitality in spring 2016 and related support
by the National Science Foundation under Grant No. DMS-1440140.
}

\maketitle

\begin{abstract}
We study the heat equation on time-dependent metric measure spaces (as well as the dual and the adjoint heat equation) and  prove existence, uniqueness and regularity. Of particular interest are properties which 
characterize the underlying space as a super-Ricci flow 
as previously introduced by the second author \cite{sturm2015}.
Our main result yields the equivalence of
\begin{itemize}
\item[$\triangleright$] dynamic convexity of the Boltzmann entropy
on the (time-dependent) $L^2$-Wasserstein space
\item[$\triangleright$] monotonicity of $L^2$-Kantorovich-Wasserstein distances under the dual heat flow acting on probability measures (backward in time)

\item[$\triangleright$]
gradient estimates for the  heat flow acting on functions (forward in time)
\item[$\triangleright$] a  Bochner inequality involving the time-derivative of the metric.
\end{itemize}
Moreover, we characterize the heat flow on functions as the unique forward EVI-flow for the (time-dependent) energy in $L^2$-Hilbert space
and the dual heat flow on probability measures as the unique backward EVI-flow for the (time-dependent)  Boltzmann entropy in $L^2$-Wasserstein space.
\end{abstract}

\tableofcontents

\section{Introduction and Statement of Main Results}

\subsection{Introduction}

The present paper has two main objectives
\begin{itemize}
\item[(i)]
to define and study the heat flow on time-dependent metric measure spaces
\item[(ii)]
to characterize super-Ricci flows of metric measure spaces 
by properties of optimal transports  and  heat flows.
\end{itemize}
The former is regarded as the `parabolic' analogue to the analysis of heat flow, optimal transport, and functional inequalities on `static' metric measure spaces. The latter should be considered as a first contribution to a theory of Ricci flows of metric measure spaces.
Our approach will combine and extend two previous -- hitherto unrelated -- lines of developments:
 the analysis on (`static') metric measure spaces and  the analysis on
(`smooth')  time-dependent Riemannian manifolds.
 
\subsubsection*{Heat flow on (`static') metric measure spaces}
The heat equation is one of the most fundamental and well studied PDEs on Riemannian manifolds. 
It is intimately linked to other important objects like Dirichlet energy, Boltzmann entropy, optimal transport, and Brownian motion.
On one hand, it is a very robust object and  admits an integral representation in terms of the heat kernel. Without any extra assumptions, its existence and basic properties are always guaranteed. On the other hand, its more subtle properties reveal deep informations on the underlying space, like curvature, genus, index etc.

Within the last decades, the heat flow was also successfully studied on more general spaces, in particular, on metric measure spaces \cite{Cheeger, Hajlasz, Shanmugalingam, sturm1998diffusion}.
The foundational work of Ambrosio, Gigli and Savar\'e \cite{agscalc, agsmet, agsbe}  clarified the picture, allowed to unify various of the previous approaches, and made clear that  for each metric measure space $(X,d,m)$ with $\int\exp\big(-C d^2(x,z)\big)dm(x)<\infty$ (for some $C,z$) 
there exists a unique solution to the heat equation, most conveniently defined as gradient flow in $L^2(X,m)$ for the Dirichlet energy (`Cheeger energy')  
$\E(u)=\int_X |\nabla u|^2\,dm$.

\subsubsection*{Synthetic lower Ricci bounds}
The heat flow  on Riemannian manifolds -- and more generally on metric measure spaces --  turned out to be a powerful tool for characterizing (synthetic) lower bounds on the Ricci curvature. Such curvature bounds are indeed necessary and sufficient for various important properties of the heat flow $t\mapsto P_tu$. Moreover, they imply that $t\mapsto (P_tu)m$ is the gradient flow for  the Boltzmann entropy $S(um)=\int u\log u\,dm$  in the space $\Pz(X)$ of probability measures equipped with the $L^2$-Kantorovich-Wasserstein distance $W$.
For instance, nonnegative Ricci curvature is equivalent to  
\begin{itemize}
\item[$\triangleright$] the gradient estimate $|\nabla P_tu|^2\le P_t|\nabla u|^2$
\item[$\triangleright$] the existence of coupled pairs of Brownian motions with $d(X_t,Y_t)\le d(X_0,Y_0)$
\item[$\triangleright$] the transport estimate $W\big((P_tu)m,(P_tv)m\big)\le W\big(u m,vm\big)$
\item[$\triangleright$] the convexity of the Boltzmann entropy $S$
on the geodesic space $(\Pz(X), W)$.
\end{itemize}
Indeed,  in the Lott-Stum-Villani approach to synthetic lower Ricci bounds \cite{sturm2006, lott2009ricci} the latter property was used to \emph{define} nonnegative Ricci curvature for metric measure spaces.
Furthermore, the previous properties -- gradient estimate, coupling property of Brownian motions, and transport estimate --  illustrate the effect of nonnegative Ricci curvature in a very graphical way, well suited for applications and modeling, and also perfectly make sense in  discrete settings, cf. Ollivier \cite{Ollivier}, 
Tannenbaum et al. \cite{Tannenbaum}, Sandhu et al.  \cite{Sandhu}.

\subsubsection*{Heat flow on time-dependent metric measure spaces}
New phenomena emerge and novel challenges arise for the heat flow if the underlying geometric objects (Riemannian manifolds, metric measure spaces) will vary in time, e.g. if they will change their `shape' or `material properties'. This might result from exterior forces or from an interior dynamic, like mean curvature flow or Ricci flow.
To model such time-dependent geometric objects, one typically considers families $(M,g_t)_{t\in I}$ consisting of a manifold $M$ and a one-parameter family of metric tensors $g_t, t\in I\subset\R$.
We will consider more generally   \emph{time-dependent metric measure spaces} $(X,d_t,m_t)_{t\in I}$ consisting of a Polish space $X$ equipped with one-parameter families of metrics (= distance functions) $d_t$ and measures $m_t, t\in I$.
The main question to be addressed are:
\begin{itemize}
\item[(a)] 
In which generality does existence and uniqueness hold  for solutions to  the heat equation on time-dependent metric measure spaces?
\item[(b)] 
Is the heat flow the gradient flow for the energy? Does it  coincide with the gradient flow for the entropy? 
\\
More generally: is there a meaningful concept of gradient flows for time-dependent functionals on  time-dependent geodesic spaces?
\item[(c)] 
What is the time-dependent counterpart to nonnegative Ricci curvature or, more generally,  to the CD$(0,\infty)$-condition?
\\
More precisely: which kind of curvature bound is necessary and/or sufficient for (the time-dependent counterpart to) the gradient estimate? Which for the corresponding transport estimate?
\\
Is there a synthetic version of such a curvature bound?
\end{itemize}
In contrast to the static case, until now nothing seemed to be known  for the heat flow on general time-dependent metric measure spaces.

For time-dependent Riemannian manifolds $(M,g_t)_{t\in I}$ -- with smoothly varying, non-degenerate $g_t$ -- question (a) allows for an easy, affirmative answer. Surprisingly enough, Brownian motion  was constructed only recently  
\cite{arnaudon2008brownian, Cou11}.
Question (b) was unsolved so far.
McCann/Topping 2010 \cite{mccanntopping},  
Arnaudon/Coulibaly/Thalmaier \cite{acthorizontal}, and
Haslhofer/Naber \cite{HN2015}  proved that the first three questions in (c) have one common answer:
\begin{equation}\label{surifl}\Ric_{g_t}+\frac12\partial_t g_t\ge0.\end{equation}
Finally, in   \cite{sturm2015}  the second author 
presented a synthetic definition for the latter, formulated as `dynamic convexity' of the Boltzmann entropy $S_t$ in the Wasserstein space $(\Pz(X),W_t)$.

The current paper, regarded as accompanying paper to \cite{sturm2015}, will provide complete answers to the previous questions in the setting of  time-dependent metric measure spaces.
We will prove existence, uniqueness, and regularity results for the heat equation and its dual. The former will be identified as the forward gradient flow for the Dirichlet energy $\E_t$  in $L^2(X,m_t)$, the latter as the backward gradient flow for the Boltzmann entropy $S_t$  in $(\Pz(X),W_t)$. A general  discussion on gradient flows for  time-dependent functionals on  time-dependent geodesic spaces will be included.
Our main result provides a comprehensive characterization of  super Ricci flows  $(X,d_t,m_t)_{t\in I}$ by the equivalence of dynamic convexity of the Boltzmann entropy, monotonicity of transport estimates under the dual heat flow, monotonicity of gradient estimates under the primal heat flow, and the time-dependent Bochner inequality.

In the static case, synthetic lower Ricci bounds will   play its role to the full only
 in combination with an upper bound on the dimension which led to the formulation of the so-called curvature-dimension condition CD$(K,N)$. The time-dependent counterpart to the CD$(K,N)$-condition will be so-called \emph{super-$(K,N)$-Ricci flows}. 
 Taking into account the role of the parameter $N\in \R_+$ requires quite some effort.
However, we expect this to be worth for future applications. 
 The case $K\not=0$, however, can be reduced to the case $K=0$ by means of a simple scaling of space and time, see Theorem \ref{K-trafo}. To  simplify the presentation, throughout this paper we thus will restrict ourselves to the curvature bound $K=0$.

\subsubsection*{Ricci flows, Super-Ricci flows, and Super-$N$-Ricci flows}

Given a manifold $M$ and a smooth 1-parameter family $(g_t)_{t\in I}$ of Riemannian tensors on $M$, we say that the `time-dependent Riemannian manifold' $(M,g_t)_{t\in I}$ evolves as a \emph{Ricci flow} if 
$\Ric_{g_t}=-\frac12\partial_t g_t$
for all $t\in I$. It is called \emph{super-Ricci flow} if instead only $\Ric_{g_t}\ge-\frac12\partial_t g_t$ holds true
on $M\times I$ (regarded as inequalities between quadratic forms on the tangent bundle of $(M,g^x_{t})$ for each $(x,t)\in M\times I$). In other words, super-Ricci flows
 are `super-solutions' to the Ricci flow equation and Ricci flows are `minimal' super-Ricci flows.

Thanks to the groundbreaking work of Hamilton \cite{Ham1, Ham2} and Perelman \cite{perelman2002entropy, perelman2003ricci, perelman2003finite}, see also \cite{CZ,KL,MoTian}, Ricci flow has attracted lot of attention and has proved itself as a powerful tool and inspiring source for many 
new developments. Currently, one of the major challenges is to extend the theory of Ricci flows and the scope of its applications beyond the setting of smooth Riemannian manifolds. In particular, one aims to define and analyze (`Ricci') flows through singularities and to study evolutions of spaces with changing dimension and/or topological type.
Kleiner/Lott \cite{kleiner2014singular} and  Haslhofer/Naber \cite{HN2015} presented notions of singular and weak solutions
 for Ricci flows. 
In \cite{HN2015}, Ricci flows of  `regular' (i.e. smooth with uniform bounds on curvature and derivatives of it)  time-depending Riemannian manifolds $(M,g_t)_{t\in I}$ of arbitrary dimension are characterized by means of functional inequalities on the path space (spectral gap or logarithmic Sobolev inequalities for the Ornstein Uhlenbeck operator). In \cite{kleiner2014singular}, Ricci flow of `singular' 3-dimensional Riemannian manifolds $(M,g_t)_{t\in I}$ (regarded as 4-dimensional Ricci flow spacetimes) is defined and analyzed in detail, allowing also for Ricci flows through singularities.

Compared to Ricci flows, super-Ricci flows allow for a much larger classes of examples. This is an advantage if one is interested in  analysis (e.g. functional inequalities, heat kernel estimates, etc.) on huge classes of singular spaces or if one tries to extend tools and insights from the study of `classical' Ricci flows to more general  time evolutions of geometric objects. It is a disadvantage if one aims for uniqueness results or for properties close to those of Ricci flows.
The defining property of super-Ricci flows for mm-spaces $(X,d_t,m_t)_t$ contains no constraint on the evolution of the measures $m_t$ but only a lower bound on the evolution of the distances $d_t$.
Moreover, super-Ricci flows can increase the dimension in order to match the constraint imposed by the lower bound on the Ricci curvature.
These distracting effects can be ruled out by considering the more restrictive class of `super-$N$-Ricci flows'.
A time-dependent weighted $n$-dimensional Riemannian manifold $(M,g_t,e^{-f_t}dvol_{g_t})_t$, for instance, is a super-$n$-Ricci flow if and only if $g_t$ satisfies \eqref{surifl}
and if $f_t$ is constant for each $t$, see Theorem 2.9 in \cite{sturm2015}.

In   \cite{sturm2015},  the second author of this paper 
presented a synthetic definition for super-$N$-Ricci flows in the general setting of time-dependent metric measure spaces. Work in progress deals with synthetic upper Ricci bounds \cite{sturm2017} which -- in combination with the former --  then also will allow for characterizations of `Ricci flows' of mm-spaces.
For most of our results, we request a controlled $t$-dependence for $d_t$ and $m_t$. Of course, this is  a severe limitation and rules out various challenging applications. Even more, one might wish to replace $X$ by varying $X_t$, e.g. allowing for changing topological type.
However, in contrast to the static case, so far
there are no existence and uniqueness results for the heat flow on time-dependent mm-spaces which hold in  `full generality'.
The current paper will lay the foundations for further work 
devoted to enlarge the scope and to include  singularities and degenerations.

\subsection{Some Examples}

Let us give some motivating examples of super-Ricci flows as defined in \cite[Definition 2.4]{sturm2015}. 
We also discuss whether they are super-$N$-Ricci flows or Ricci flows.

\begin{example}[`Vertebral column']
Consider a surface of revolution with piecewise constant negative curvature $\Ric=-Kg$ for some $K>0$ depicted in
Figure \ref{fig:awesome_image}. Under the evolution of a Ricci flow the curvature of the surface where $\Ric=-Kg$ will increase, while the curvature of 
the ``rims'' ($\Ric=+\infty$) will decrease.  In this sense the region of negative curvature will inflate, while the edges will smooth out. Under the 
evolution of a super-Ricci flow the surface inflates as well but it may keep the edges -- or it may start to smoothen them at any later time  or with smaller speed.
\begin{figure}[h]
    \centering
   \includegraphics[scale=0.75]{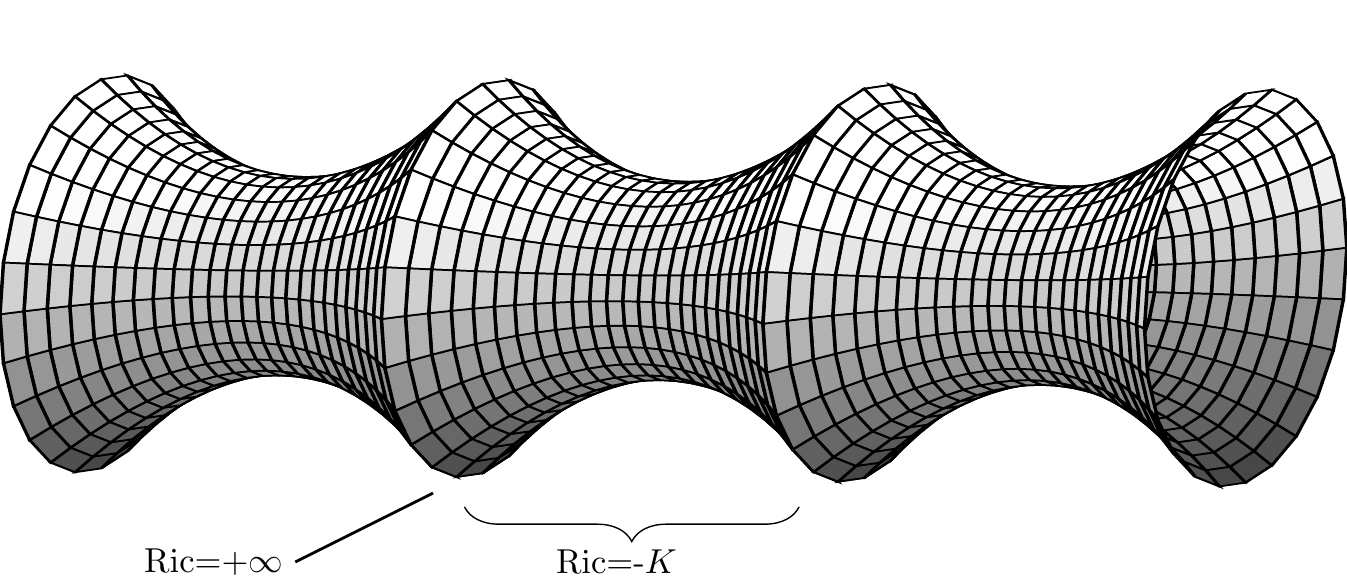}
    \caption{Surface of revolution of a piecewise hyperbolic space}
    \label{fig:awesome_image}
\end{figure}
\end{example}

\begin{example}[`Wandering Gaussian']
Let $X=\R^n$, $d_t(x,y)=\|x-y\|
$ 
and $m_t=e^{-V_t} {\mathcal Leb}^n$ with
$$V_t(x)=\langle x, \alpha_t\rangle^2+\langle x, \beta_t\rangle +\gamma_t$$
where 
$\alpha,\beta:I\to \R^n$  and $\gamma:I\to \R$ are arbitrary functions. Then $(X,d_t,m_t)_{t\in I}$ is a super-Ricci flow.
For each $N\in[n,\infty)$ it will be a super-$N$-Ricci flow if and only if $\alpha\equiv\beta\equiv 0$.
\end{example}

\begin{example}[`Exploding point'] Let $(M,g_0)$ be a compact, $n$-dimensional Riemannian manifold of constant Ricci curvature $-Kg_0<0$ (e.g. a compact quotient of a hyperbolic space) and put
$$g_t=\left\{\begin{array}{ll}
(1+2Kt)g_0,& t> t_*\\
0,&t\le t_*
\end{array}\right.$$
for $t_*=-\frac1{2K}$. Let $(X,d_t,m_t)_{t\in\R}$ be the induced time-dependent mm-space with normalized volume $m_t$  where $(X,d_t)$ for $t\le t_*$ will be identified with a 1-point space (and $m_t$ with the Dirac mass in this point), see also Firgure \ref{awesome3}. Then this is a super-Ricci flow -- provided we slightly enlarge the scope of \cite{sturm2015} to also admit degenerate distances $d_t$ (or varying spaces $X_t$). It will be no super-$N$-Ricci flow for $N<n$.
  \begin{figure}[h]\centering
   \includegraphics[scale=.8]{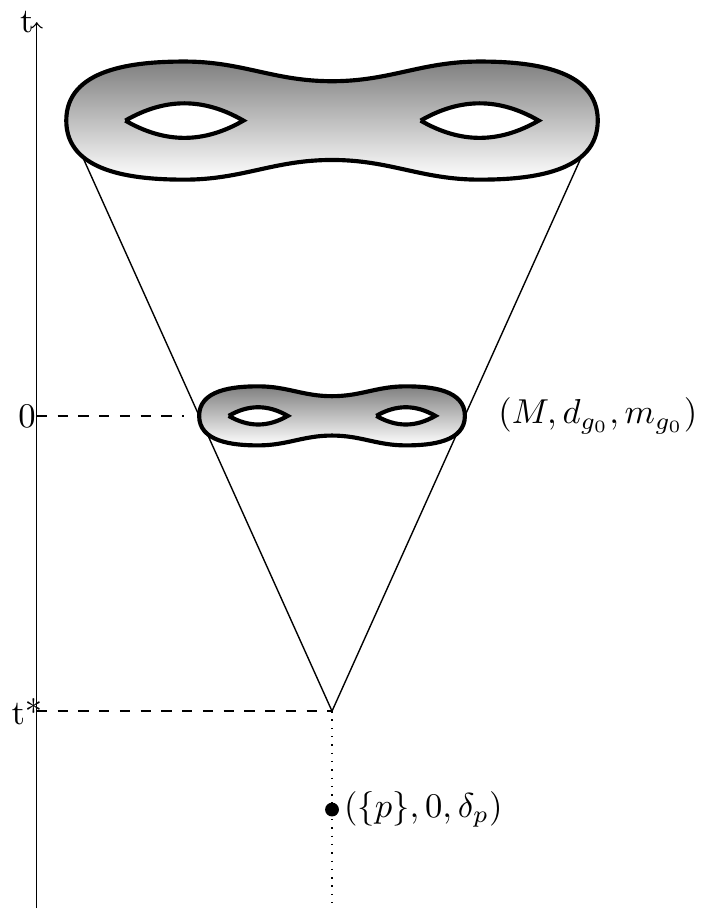}
    \caption{Point exploding to a hyperbolic quotient}
    \label{awesome3}
    \end{figure}
    
More generally, consider $(\overline M, \overline g_t)=(M'\times M,g'\otimes g_t)$  with $(M',g')$ being a compact $n'$-dimensional Ricci-flat Riemannian manifold. Then the induced time-dependent mm-space is a super-Ricci flow but no super-$N$-Ricci flow for $N<n'+n$. For any $N\in [n',n'+n)$, up to isometry the only super-$N$-Ricci flow which coincides with the given mm-space for $t\le t_*$ is the static mm-space  induced by $(M',g')$.
\end{example}

\begin{example}[`Singular suspension']
Consider the product $M\times[0,\pi]$, where $M=S^2(1/\sqrt3)\times S^2(1/\sqrt3)$ and $S^2(r)$ denotes the 2-dimensional sphere with radius $r$. 
We contract each
of the fibers $\mathcal S:=M\times\{0\}$ and $\mathcal N:=M\times\{\pi\}$ to a point, the `south' and the `north pole', respectively.
The resulting space is called \emph{spherical suspension} and is denoted by $\Sigma(M)$. We endow $\Sigma(M)$ with 
the
 measure $d\hat m(x,s):=\, dm(x)\otimes(\sin^4 s\, ds)$ and 
the  metric $d_{\Sigma(M)}$ defined by
\begin{align*}
 \cos(d_{\Sigma(M)}((x,s),(x',s'))):=\cos s\cos s'+\sin s\sin s'\cos(d(x,x')\wedge \pi),
\end{align*}
where
 $m$ and $d$ are the volume and metric  of $M$ and
where $(x,s),(x',s')\in M\times[0,\pi]$.
Since $M$ is a RCD$^*(3,4)$ space, the cone of it is a RCD$^*(4,5)$ space \cite{Ketterer}. 

The punctured cone $\Sigma_0:=\Sigma(M)\setminus\{\mathcal S,\mathcal N\}$
is an incomplete $5$-dimensional Riemannian manifold.
Let $g_0$ denote the metric tensor of $\Sigma_0$. The curvature of the punctured cone can be calculated explicitly and is given by $\Ric(g_0)=4g_0$.
Then 
$ g(t):=(1-8t)g_0$
defines a solution to the Ricci flow $\Ric(g_t)=-\frac12\partial_tg_t$ with $g(0)=g_0$, which collapses to a point at time $T=\frac18$.

 We claim that the associated metric measure space 
$(\Sigma(M),d_{\Sigma(M)}(t),\hat m_t)_{t\in I}$ for $I=(0,T)$  is a super-Ricci flow. 
Fix $t\in I$ and let 
$\mu_0,\mu_1\in\ Dom(S_t)$ on $\Sigma(M)$ be given. Let $(\mu_a)_{a\in[0,1]}$ be a $W_t$-geodesic connecting $\mu_0,\mu_1$.
Then, $\mu_a=(e_a)_*\nu$, where $\nu$ is an optimal path measure, i.e. a
probability measure on the $d_t$-geodesics $\Gamma(\Sigma(M))$ of $\Sigma(M)$ such that 
$(e_0,e_1)_*\nu$ is an optimal 
coupling of $(e_0)_*\nu=\mu_0,(e_1)_*\nu=\mu_1$, where $e_a\colon\Gamma(\Sigma(M))\to\Sigma(M)$ denotes the evaluation map. According to Theorem 3.3 in 
\cite{bacher} every optimal path measure $\nu$ will give no mass to $d_t$-geodesics through 
the poles. Hence we can omit the $d_t$-geodesics through the poles without changing the $W_t$-geodesics. Since the punctured cone $(\Sigma_0,g_t)_{t\in I}$ is a 
Ricci flow, and in particular a super-Ricci flow in the sense of Definition 2.4 in \cite{sturm2015}, 
the metric measure space $(\Sigma(M),d_{\Sigma(M)}(t),\hat m_t)_{t\in I}$ is a super-Ricci flow as well.

Let us emphasize that for each $t\in[0,1/8)$ the sectional curvature of the punctured spherical cone $\Sigma_0$ is neither 
bounded from below nor from above. Indeed, for $x,y\in S^2(1/\sqrt3)$ and $0<r<\pi$ an orthonormal basis of the tangent space $T_{(x,y,r)}\Sigma_0$ is given 
by $\{\hat u_1,\hat u_2,\hat v_1,\hat v_2,\hat w\}$ where
$\hat u_i=\frac1{\sin r}(u_i,0,0)$, $\hat v_i=\frac1{\sin r}(0,v_i,0)$, $\hat w=(0,0,1)$
and $u_1,u_2$ is an orthonormal basis of $T_x(S^2(1/\sqrt 3))$ and $v_1,v_2$ is an orthonormal basis of $T_y(S^2(1/\sqrt 3))$. Then for the sectional curvature we 
find
\begin{align*}
 \Sec_{(x,y,r)}(\hat u_1,\hat u_2)&=\frac{3-\cos^2r}{\sin^2r},\quad \Sec_{(x,y,r)}(\hat u_1,\hat v_1)=-\frac{\cos^2r}{\sin^2r}\\
 \Sec_{(x,y,r)}(\hat u_1,\hat v_2)&=-\frac{\cos^2r}{\sin^2r},\quad \Sec_{(x,y,r)}(\hat u_1,\hat w)=1,
\end{align*}
and analogously if we replace $\hat u_1$ by the vectors $\hat u_2,\hat v_1,\hat v_2$. This implies in particular that $\Ric_{(x,y,r)}(\xi,\xi)=4$, but
for $r\to0$ and $r\to\pi$, $\Sec_{(x,y,r)}(\hat u_1,\hat u_2)\to+\infty$ and $\Sec_{(x,y,r)}(\hat u_1,\hat v_i)\to-\infty$.

Let us also point out ongoing work \cite{erbar-sturm2017} indicating that $(\Sigma(M),d_{\Sigma(M)}(t),\hat m_t)_{t\in I}$
will not be a Ricci flow in the sense of \cite{sturm2017}.
\end{example}

\subsection{Main Results}

\subsubsection*{The setting}
Throughout this  introductory chapter, we fix a \emph{time-dependent metric measure space}  $\big(X,d_t,m_t\big)_{t\in I}$  where $I=(0,T)$ and
$X$ is a compact space equipped with  one-parameter families of geodesic metrics 
$d_t$ and Borel measures
$m_t$.
We always assume the measures $m_t$ are mutually absolutely continuous with bounded, Lipschitz continuous logarithmic  densities and that 
the metrics $d_t$  are uniformly bounded and equivalent to each other with 
\begin{equation}\label{d-lipintro}
\left| \log\frac{d_t(x,y)}{d_s(x,y)}\right|\le L\cdot |t-s|
\end{equation}
(`log Lipschitz continuity'). Moreover, 
we assume that
 for each $t$ the static space $(X,d_t,m_t)$ satisfies a Riemannian curvature-dimension condition 
 in the sense of \cite{agmr}, \cite{eks2014}. (In various respects, the latter is not really a restriction, see Remark \ref{apriori}.)
 
 Thus for each $t$ under consideration, there is a well-defined \emph{Laplacian}
  $\Delta_t$ 
  on $L^2(X,m_t)$ 
  characterized by
$-\int_X \Delta_tu \, v\,dm_t=\E_t(u,v)$
where the Dirichlet energy
\begin{equation*}
\E_t(u,u)=\int_X |\nabla_t u|^2\-dm_t=\liminf_{\stackrel{v\to u\; \mbox{\tiny in}\; L^2(X,m_t)}{v\in \Lip(X,d_t)}}\int_X (\mathrm{lip}_tv)^2\,dm_t
\end{equation*} is defined either in terms of the  minimal weak upper gradient  $|\nabla_t u|$ of $u\in L^2(X,m_t)$ or alternatively in terms of the  pointwise Lipschitz constant  $\mathrm{lip}_tv(.)$.

\subsubsection*{Heat equation}
Our first important result concerns existence and uniqueness for solutions to the heat equation -- as well as for the adjoint heat equation -- on the time-dependent metric measure space  $(X,d_t,m_t)_{t\in I}$. Moreover, it yields regularity of solutions  and representation  as integrals w.r.t.\ a heat kernel. See Theorems 3.3 and 3.5 for the precise formulations in slightly more general context.

\begin{theorem}
There exists a
\emph{heat kernel}  $p$ on $\{(t,s,x,y)\in I^2\times X^2: t>s\}$, H\"older continuous in all variables and satisfying the propagator property
$p_{t,r}(x,z)=\int p_{t,s}(x,y)p_{s,r}(y,z)\,dm_s(y)$,
such that
\begin{itemize}
\item[(i)]
for each $s\in I$ and $h\in L^2(X,m_s)$
$$(t,x)\mapsto P_{t,s}h(x):=\int p_{t,s}(x,y)h(y)\, dm_s(y)$$
is the unique solution to the \emph{heat equation} 
$$\partial_tu_t=\Delta_tu_t\qquad\mbox{on }(s,T)\times X$$
with $u_s=h$;
\item[(ii)]
for each $t\in I$ and $g\in L^2(X,m_t)$
$$(s,y)\mapsto P^*_{t,s}g(y):=\int p_{t,s}(x,y)g(x)\, dm_t(x)$$
is the unique solution to the \emph{adjoint heat equation} 
$$\partial_sv_s=-\Delta_sv_s+\dot f_s\cdot v_s\qquad\mbox{on }(0,t)\times X$$
with $v_t=g$. Here $\dot f_s=-\partial_t\big(\frac{dm_t}{dm_s}\big)\big|_{t=s}.$
\end{itemize}
\end{theorem}

Many properties which are self-evident for the heat semigroup on static mm-spaces (e.g. ``operator and semigroup commute'' or ``the semigroup maps $L^2$ into the domain of the operator'') no longer hold true for the heat propagator on time-dependent mm-spaces -- or require detailed, sophisticated proofs.
Let us emphasize here that in general $\Dom(\Delta_t)$ will depend on $t$.

We derive various important $L^2$-properties and estimates -- partly in the more general setting of heat flows for time-dependent Dirichlet forms -- 
the most prominent of them being the EVI-characterization, the energy estimate and the commutator lemma.

\begin{theorem}
\begin{itemize}
\item[(i)] 
The heat  flow is uniquely characterized as the \emph{dynamic forward EVI${(-L/2,\infty)}$-flow for $\frac12\times$ the Dirichlet energy} on $L^2(X,m_t)_{t\in I}$ in the following sense: for all solutions $(u_t)_{t\in(s,\tau)}$ to the heat equation,
for all $\tau\le T$ and all $w\in \Dom(\E)$ 
  \begin{align*}
     -\frac12\partial_s^+ \big\|u_s-w\big\|^2_{s,t}\Big|_{s=t}+\frac L4\cdot \big\|u_s-w\big\|^2_{s,t}
       ~\ge \frac12\E_t(u_t)-\frac12\E_t(w).
  \end{align*}

\item[(ii)] For all $s\in (0,T)$ and $u\in\Dom(\E_s)$ 
$$P_{t,s}u\in\Dom(\Delta_t)\qquad\mbox{for a.e.\ $t>s$}$$
and $\int_s^\tau e^{-3L(t-s)}\int|\Delta_tP_{t,s}u|^2dm_t\,dt\le\frac12\E_s(u)$ for all $\tau>s$..
\item[(iii)] For all $\sigma<\tau$, all $u,v\in L^2$ and a.e.\ $s,t\in(\sigma,\tau)$ with $s<t$
$$\int\Big[ \Delta_tP_{t,s}u_s-P_{t,s}\Delta_su_s\Big]v_t\,dm_t\le C\cdot\sqrt{t-s}$$
where $u_s=P_{s,\sigma}u, v_t=P^*_{\tau,t}v$.
\end{itemize}
\end{theorem}

\bigskip

We define the \emph{dual heat flow}
$\hat P_{t,s}: \Pz(X)\to\Pz(X)$ by
$$(\hat P_{t,s}\mu)(dy)=
\left[\int p_{t,s}(x,y)\,d\mu(x)\right]m_s(dy).$$ 
In particular, $(\hat P_{t,s}\delta_x)(dy)=p_{t,s}(x,dy)$
and
$
\hat P_{t,s}\big(g\cdot m_t\big)=\big(P^*_{t,s}g\big)\cdot m_s$.

\bigskip

\subsubsection*{Characterization of super-Ricci flows}
In \cite{sturm2015}, the second author has introduced and analyzed the notion of \emph{super-Ricci flows} for time-dependent metric measure 
 $(X,d_t,m_t)_{t\in I}$. The defining property of the latter is the so-called dynamic convexity of the Boltzmann entropy
 $S:I\times\Pz\to(-\infty,\infty]$ with
 $$S_t(\mu)=\int u \log u\,dm_t\qquad\mbox{if $\mu=u\,m_t$ }$$
 and $S_t(\mu)=\infty$ if $\mu\not\ll m_t$. Here $\Pz=\Pz(X)$ will denote the space of probability measures on $X$, equipped with time-dependent Kantorovich-Wasserstein distances $W_t$ induced by $d_t$, $t\in I$.
 This property was proven to be stable under an appropriate space-time version of measured Gromov-Hausdorff convergence and suitably bounded families of super-Ricci flows were shown to be compact -- a far reaching analogue to the stability and compactness results in 
 the Lott-Sturm-Villani theory of metric measure spaces with synthetic lower Ricci bounds.
 Furthermore,  in the case of time-dependent Riemannian manifolds this novel, synthetic definition  of super-Ricci flows
 was proven to be equivalent to the classical one: $\Ric_{g_t}+\frac12\partial_t g_t\ge0$.
 
 \medskip
 
 The main goal of the current paper  is to characterize super-Ricci flows in terms of the heat flow (acting on functions, forward in time) and of the dual 
 heat flow (acting on probability measures, backward in time).
  Our first result in this direction
  is a complete analogue to the characterization of synthetic lower Ricci bounds in the sense of Lott-Sturm-Villani
   for `static' metric measure spaces derived by Ambrosio, Gigli, Savar\'e \cite{agsbe}.

\begin{theorem}\label{Main} The following assertions are equivalent:
\begin{itemize}
\item[\bf(I)] 
For a.e.\ $t\in (0,T)$ and every  $W_t$-geodesic  $(\mu^a)_{a\in[0,1]}$ in $\Pz$ with
$\mu^0,\mu^1\in \Dom(S)$
\begin{equation}\label{est-I}
\partial^+_a S_t(\mu^{a})\big|_{a=1-}-\partial^-_a S_t(\mu^{a})\big|_{a=0+}
\ge- \frac 12\partial_t^- W_{t-}^2(\mu^0,\mu^1)
\end{equation}
(`dynamic convexity').

\item[\bf(II)] 
For all $0\le s<t\le T$ and $\mu,\nu\in\Pz$
\begin{equation}\label{est-II}
W_s (\hat P_{t,s}\mu,\hat P_{t,s}\nu)\le W_t (\mu,\nu)
\end{equation}
(`transport estimate').
\item[\bf(III)] 
For all $u\in\Dom(\E)$ and all $0<s<t< T$
\begin{equation}\label{est-III}
\big|\nabla_t(P_{t,s}u)\big|^2\le P_{t,s}\big(|\nabla_s u|^2\big)
\end{equation}
(`gradient estimate').
\item[\bf(IV)] 
For all $0<s<t<T$ and for all $u_s,g_t\in\F$ with $g_t\ge0$,  $g_t\in L^\infty$, $u_s\in \Lip(X)$
and for a.e.\ $r\in(s,t)$ 
 \begin{equation}\label{est-IV}
 {\bf \Gamma}_{2,r}(u_r)(g_r)
 \geq\frac12\int\stackrel{\bullet}{\Gamma}_r(u_r)g_rdm_r
\end{equation}
(`dynamic Bochner inequality' or `dynamic Bakry-Emery condition')
where $u_r=P_{r,s}u_s$ and $g_r=P^*_{t,r}g_t$.
Moreover, the following regularity assumption is satisfied:

\begin{equation}\label{reg-boch}
\mbox{ $u_r\in\Lip(X)$ for all $r\in  (s,t)$ with }
\sup_{r,x}\lip_ru_r(x)<\infty.
\end{equation}
\end{itemize}
\end{theorem}
Here and in the sequel
 $${\bf \Gamma}_{2,r}(u_r)(g_r):= \int\Big[\frac12\Gamma_{r}(u_r)\Delta_r g_r
 +(\Delta_r u_r)^2g_r
 +\Gamma_r(u_r,g_r)\Delta_ru_r\Big]dm_r$$ denotes the distribution valued $\Gamma_2$-operator (at time $r$) applied to $u_r$ and tested against $g_r$
 and
$$\stackrel{\bullet}{\Gamma}_r(u_r):=\mbox{w-}\lim_{\delta\to0}\
\frac1\delta\Big(\Gamma_{r+\delta}(u_r)-\Gamma_r(u_r)\Big)$$ denotes any subsequential weak limit  of $\frac1{2\delta}\big(\Gamma_{r+\delta}-\Gamma_{r-\delta}\big)(u_r)$ in $L^2((s,t)\times X)$.

\bigskip

\subsubsection*{EVI characterization of the dual heat flow}
Recall that we started with the heat equation (acting on functions, forward in time) as a forward gradient flow for the time-dependent Dirichlet energy.
By duality, we defined the dual heat flow (acting on probability measures, backward in time). This turns out to be the backward gradient flow for the Boltzmann entropy -- in a very precise,  strong sense  -- and it is the only one with this property.

\begin{theorem} Each of the assertions of the previous Theorem 
implies that
the dual heat flow $t\mapsto\mu_t=\hat P_{\tau,t}\mu$ is the unique \emph{dynamical (backward) EVI$^-$-gradient flow} for the Boltzmann entropy $S$ in the following sense: \\
For every $\mu\in\Dom(S)$ and every $\tau< T$ the absolutely continuous curve $t\mapsto \mu_t$ satisfies
\begin{equation*}
\frac12 \partial_s^- W_{s,t}^2(\mu_s,\sigma)\big|_{s=t-}\geq S_t(\mu_t)-S_t(\sigma)
\end{equation*}
for all $\sigma\in\Dom(S)$ and all $t\le\tau$.
\end{theorem}

\bigskip

\subsubsection*{Characterization of super-$N$-Ricci flows}

For static metric measure spaces, it turned out that many powerful applications of synthetic lower bounds on the Ricci curvature are available only in  combination with some synthetic upper bound on the dimension. This led to the
so-called curvature-dimension condition CD$(K,N)$.
In a similar spirit, in \cite{sturm2015} the notion of super-Ricci flows for time-dependent metric measure spaces was tightened up  towards super-$N$-Ricci flows.

 We aim to characterize super-$N$-Ricci flows in terms of the heat flow, the dual heat flow, and the time-dependent Bochner inequality.
 Our main result provides a complete characterization, analogous to the 
 proof of the equivalence of the curvature-dimension condition of Lott-Stum-Villani and the Bochner inequality of Bakry-\'Emery  
   for `static' metric measure spaces derived by Erbar, Kuwada, and the second author \cite{eks2014}.

\begin{theorem}\label{Main-N} For each $N\in (0,\infty)$ the following are equivalent:
\begin{itemize}
\item[\bf(I$_N$)] 
For a.e.\ $t\in (0,T)$ and every  $W_t$-geodesic  $(\mu^a)_{a\in[0,1]}$ in $\Pz$ with
$\mu^0,\mu^1\in \Dom(S)$
\begin{equation}\label{est-I-N}
\partial^+_a S_t(\mu^{a})\big|_{a=1-}-\partial^-_a S_t(\mu^{a})\big|_{a=0+}
\ge- \frac 12\partial_t^- W_{t-}^2(\mu^0,\mu^1)
+\frac1N \big| S_t(\mu^0)-S_t(\mu^1)\big|^2.
\end{equation}

\item[\bf(II$_N$)] 
For all $0\le s<t\le T$ and $\mu,\nu\in\Pz$
\begin{equation}\label{est-II-N}
W^2_s (\hat P_{t,s}\mu,\hat P_{t,s}\nu)\le W^2_t (\mu,\nu)-\frac2N\int_s^t \left[ S_r(\hat P_{t,r}\mu)-S_r(\hat P_{t,r}\nu)\right]^2dr.
\end{equation}
\item[\bf(III$_N$)] 
For all $u\in\Dom(\E)$ and all $0<s<t< T$
\begin{equation}\label{est-III-N}
\big|\nabla_t(P_{t,s}u)\big|^2\le P_{t,s}\big(|\nabla_s(u)|^2\big)-\frac2N\int_s^t \Big( P_{t,r}\Delta_r P_{r,s} u\Big)^2dr.
\end{equation}
\item[\bf(IV$_N$)] 
For all $0<s<t<T$ and for all $u_s,g_t\in\F$ with $g_t\ge0$,  $g_t\in L^\infty$, $u_s\in \Lip(X)$
the regularity assumption \eqref{reg-boch} is satisfied and for a.e.\ $r\in(s,t)$ 
 \begin{equation}\label{est-IV-N}
 {\bf \Gamma}_{2,r}(u_r)(g_r)
 \geq\frac12\int\stackrel{\bullet}{\Gamma}_r(u_r)g_rdm_r+\frac1N\Big(\int\Delta_r u_rg_rdm_r\Big)^2
\end{equation}
(`dynamic Bochner inequality' or `dynamic Bakry-Emery condition')
where $u_r=P_{r,s}u_s$ and $g_r=P^*_{t,r}g_t$.
\end{itemize}
\end{theorem}

\begin{remark}

\begin{itemize}
\item[a.]
In {\bf(I$_N$)}, the requested property for a.e.\ $t$ will imply that it holds true for all $t\in(0,T)$.
\item[b.]
The transport estimate {\bf(II$_N$)} implies the `stronger' property
\begin{equation*}
W^2_s (\hat P_{t,s}\mu,\hat P_{t,s}\nu)\le W^2_t (\mu,\nu)-\frac2N
\int_s^t \int_0^1\Big(\partial_a
S_r(\rho^a_r)\Big)^2\,da\,dr
\end{equation*}
 where $(\rho^a_r)_a$ denotes the $W_r$-geodesic connecting $\hat P_{r,t}\mu$ and $\hat P_{r,t}\nu$. 
 \item[c.] Under slightly more restrictive assumptions on $(X,d_t,m_t)$ -- namely, $C^1$-dependence of $t\mapsto \log d_t$ instead of Lipschitz continuity --  in subsequent work of the first author \cite{Ko2} a refined version of the dynamic Bochner inequality {\bf(IV$_N$)} will be deduced with estimate \eqref{est-IV-N}  for every $r$ and all $u_r,g_r$ in respective domains --   without  requiring that they are solutions to heat and adjoint heat equations, resp.
 \item[d.]
Note that the regularity assumption \eqref{reg-boch} in our formulation of the dynamic Bochner inequality is not really a restriction.
Indeed, such an estimate with $C=2(K+L)$ will always follow from  the log-Lipschitz bound \eqref{d-lipintro} and the RCD$(-K,\infty)$-condition for the static mm-spaces $(X,d_t,m_t)$.
 \end{itemize}
\end{remark}

\bigskip

\subsubsection*{Super-$(K,N)$-Ricci flows}

A more general version of the previous Theorem will deal with the equivalences to \emph{dynamic $(K,N)$-convexity} of the Boltzmann entropy as introduced in \cite{sturm2015}. To simplify the presentation, however, we will restrict ourselves here to the case $K=0$. Indeed, we would not expect new challenges or novel insights from the more general case $(K,N)$ since this can be easily transformed into the case $(0,N)$ by means of a simple rescaling time and space.

\begin{theorem}\label{K-trafo} Assume that the time-dependent mm-space $(X,d_t,m_t)_{t\in I}$ is  super-$(K,N)$-Ricci flow in the sense that
for a.e.\ $t\in I$ and every  $W_t$-geodesic  $(\mu^a)_{a\in[0,1]}$ in $\Pz$ with
$\mu^0,\mu^1\in \Dom(S)$
\begin{eqnarray}\label{KN-conv}\nonumber
\partial^+_a S_t(\mu^{a})\big|_{a=1-}-\partial^-_a S_t(\mu^{a})\big|_{a=0+}
&\ge&- \frac 12\partial_t^- W_{t-}^2(\mu^0,\mu^1)
+\frac1N \big| S_t(\mu^0)-S_t(\mu^1)\big|^2\\
&&\qquad+
KW_{t}^2(\mu^0,\mu^1).
\end{eqnarray}
Then for each $C\in\R$ the  time-dependent mm-space $(X,\tilde d_t,\tilde m_t)_{t\in \tilde I}$ is a super-$N$-Ricci flow
if we put $$\tilde d_t=e^{-K\tau(t)}d_{\tau(t)}, \qquad \tilde m_t=m_{\tau(t)}, \qquad \tau(t)=\frac{-1}{2K}\log(C-2Kt)$$
and $\tilde I=\{\tau(t): t\in I, 2Kt<C\}$.
\end{theorem}

\begin{proof} Put $\tilde d=e^{-K\tau(t)}d_{\tau(t)}$. Then every $\tilde W_t$-geodesic will be a $W_{\tau(t)}$-geodesic. Therefore, the transformation $d\mapsto \tilde d$ will not change the term $\frac1N \big| S_t(\mu^0)-S_t(\mu^1)\big|^2$ nor the term $\partial^+_a S_t(\mu^{a})\big|_{a=1-}-\partial^-_a S_t(\mu^{a})\big|_{a=0+}$ in \eqref{KN-conv}. Moreover,
\begin{eqnarray*}
 \frac 12\partial_t^- \tilde W_{t-}^2(\mu^0,\mu^1)&=&
 e^{-2K\tau(t)}\Big[-K\partial_t\tau(t)\cdot W_{\tau(t)}+\big(\partial_t^-  W_{.}\big)\big(\tau(t)-\big)\cdot \partial_t\tau(t)
 \Big]\cdot W_{\tau(t)}\\
 &=&
 e^{-2K\tau(t)}\cdot \partial_t\tau(t)\cdot\Big[-K\cdot W^2_{.}+\frac12\partial_t^- W^2_{.}\Big]\big(\tau(t)-\big)
 \\
 &=&\Big[-K\cdot W^2_{.}+\frac12\partial_t^- W^2_{.}\Big]\big(\tau(t)-\big).
 \end{eqnarray*}
Thus \eqref{KN-conv} implies
\begin{eqnarray*}
 \frac 12\partial_t^- \tilde W_{t-}^2(\mu^0,\mu^1)&=&\Big[-K\cdot W^2_{.}+\frac12\partial_t^- W^2_{.}\Big]\big(\tau(t)-\big)\\
 &\ge&
- \partial^+_a S_{\tau(t)}(\mu^{a})\big|_{a=1-}+\partial^-_a S_{\tau(t)}(\mu^{a})\big|_{a=0+}
 +\frac1N \big| S_{\tau(t)}(\mu^0)-S_{\tau(t)}(\mu^1)\big|^2
 \end{eqnarray*}
 which proves the dynamic $N$-convexity of $\tilde S$ and thus the super-$N$-Ricci flow property of $(X,\tilde d_t,\tilde m_t)_{t\in \tilde I}$.
\end{proof}

\begin{corollary}\label{cor-KN} For each $N\in(0,\infty)$ and $K\in\R$ the following are equivalent:
\begin{itemize}
\item[\bf(I$_{K,N}$)] 
For a.e.\ $t\in (0,T)$ and every  $W_t$-geodesic  $(\mu^a)_{a\in[0,1]}$ in $\Pz$ with
$\mu^0,\mu^1\in \Dom(S)$
\begin{eqnarray}
\partial^+_a S_t(\mu^{a})\big|_{a=1-}-\partial^-_a S_t(\mu^{a})\big|_{a=0+}\nonumber
&\ge&- \frac 12\partial_t^- W_{t-}^2(\mu^0,\mu^1)+K\cdot W_{t}^2(\mu^0,\mu^1)\\
&&\quad +\frac1N \big| S_t(\mu^0)-S_t(\mu^1)\big|^2.\label{est-I-KN}
\end{eqnarray}

\item[\bf(II$_{K,N}$)] 
For all $0\le s<t\le T$ and $\mu,\nu\in\Pz$
\begin{equation}\label{est-II-KN}
e^{-2Ks}W^2_s (\hat P_{t,s}\mu,\hat P_{t,s}\nu)\le e^{-2Kt}W^2_t (\mu,\nu)-\frac2N\int_s^t e^{-2Kr}\left[ S_r(\hat P_{t,r}\mu)-S_r(\hat P_{t,r}\nu)\right]^2dr.
\end{equation}
\item[\bf(III$_{K,N}$)] 
For all $u\in\Dom(\E)$ and all $0<s<t< T$
\begin{equation}\label{est-III-KN}
e^{2Kt}\big|\nabla_t(P_{t,s}u)\big|^2\le e^{2Ks} P_{t,s}\big(|\nabla_s(u)|^2\big)-\frac2N\int_s^t e^{2Kr}\Big( P_{t,r}\Delta_r P_{r,s} u\Big)^2dr.
\end{equation}
\item[\bf(IV$_{K,N}$)] 
For all $0<s<t<T$ and for all $u_s,g_t\in\F$ with $g_t\ge0$,  $g_t\in L^\infty$, $u_s\in \Lip(X)$
the regularity assumption \eqref{reg-boch} is satisfied and for a.e.\ $r\in(s,t)$ 
 \begin{equation}\label{est-IV-KN}
 {\bf \Gamma}_{2,r}(u_r)(g_r)
 \geq\frac12\int\stackrel{\bullet}{\Gamma}_r(u_r)g_rdm_r+K\int{\Gamma}_r(u_r)g_rdm_r+\frac1N\Big(\int\Delta_r u_rg_rdm_r\Big)^2
\end{equation}
where $u_r=P_{r,s}u_s$ and $g_r=P^*_{t,r}g_t$.
\end{itemize}
\end{corollary}

\begin{proof}
As in the proof of the previous Theorem, consider  the  time-dependent mm-space $(X,\tilde d_t,\tilde m_t)_{t\in \tilde I}$ with $\tilde d_t=e^{-K\tau(t)}d_{\tau(t)}$, $\tilde m_t=m_{\tau(t)}$
and $\tilde I=\{\tau(t): t\in I, 2Kt<C\}$ where  $\tau(t)=\frac{-1}{2K}\log(C-2Kt)$.
Then 
$$\tilde  W^2_t=e^{-2K\tau}W^2_\tau,\quad
\tilde \Gamma_t =e^{2K\tau}\Gamma_\tau,\quad \tilde\Delta_t=e^{2K\tau}\Delta_\tau,\quad \Gamma_{2,t} =e^{2K\tau}\Gamma_{2,\tau},\quad \dot\tau_t=e^{2K\tau}.$$
Moreover, $\tilde P_{t,s}=P_{\tau(t),\tau(s)}$.
Thus  each of the statements (I$_N$) -- (IV$_N$) for $(X,\tilde d_t,\tilde m_t)_{t\in \tilde I}$ obviously is equivalent to the corresponding statement (I$_{K,N}$) -- (IV$_{K,N}$) for $(X,d_t,m_t)_{t\in I}$. For instance, the equivalence 
``(II$_{N}$) for $(X,\tilde d_t,\tilde m_t)$ $\Leftrightarrow$ (II$_{K,N}$) for $(X,d_t,m_t)$'' follows from the fact that
$$e^{-2K\tau}W_\tau^2-e^{-2K\sigma}W_\sigma^2=\tilde W_t^2-\tilde W_s^2$$ for $\tau=\tau(t)$ and $\sigma=\tau(s)$ and
\begin{eqnarray*}
\frac2N\int_s^t \left[ \tilde S_r(\hat{\tilde P}_{t,r}\mu)-S_r(\hat{\tilde P}_{t,r}\nu)\right]^2dr&=&
\frac2N\int_\sigma^\tau e^{-2Kr}\left[ S_r(\hat P_{t,r}\mu)-S_r(\hat P_{t,r}\nu)\right]^2dr.
\end{eqnarray*}

\end{proof}

\bigskip

\subsubsection*{Discussion of standing assumptions.}

Let us briefly comment on the
assumptions which we imposed throughout this introduction and for major parts of this paper.

Let us start with the discussion on the a priori assumption that each of the static spaces satisfies 
a Riemannian curvature-dimension condition.

\begin{remark}\label{apriori}
Given a time-dependent mm-space $(X,d_t,m_t)_{t\in I}$ which satisfies all the assumptions mentioned in the beginning of this chapter but no Riemannian curvature-dimension condition is requested.
Instead of that, each static mm-space $(X,d_t,m_t)$ is merely assumed to be infinitesimally Hilbertian and $S_t$ is requested to be  absolutely continuous along $W_t$-geodesics.

Then assertion {\bf(I$_N$)} of the Main Theorem \ref{Main-N} implies that for a.e.\  $t\in I$ the static space
$$(X,d_t,m_t) \quad\mbox{satisfies a RCD}^*(-L,N) \mbox{ condition}.$$ 
\end{remark}

\begin{proof} 
 {\bf(I$_N$)} together with the log-Lipschitz bound \eqref{d-lipintro} implies that
 along all $W_t$-geodesics
 \begin{eqnarray*}
\partial^+_a S_t(\mu^{a})\big|_{a=1-}-\partial^-_a S_t(\mu^{a})\big|_{a=0+}
&\ge&- L\cdot W_{t}^2(\mu^0,\mu^1)
+\frac1N \big| S_t(\mu^0)-S_t(\mu^1)\big|^2.
\end{eqnarray*}
In combination with the absolute continuity of $a\mapsto S_t(\mu^{a})$ this yields the RCD$^*(-L,N)$-condition,
cf. \cite{sturm2015}.
\end{proof}

Next, we will discuss the assumption \eqref{d-lipintro} concerning log-Lipschtiz continuity of $t\mapsto d_t$.

\begin{remark} Let $(M,g_t)_t$ be a time-dependent  Riemannian manifold and let $(X,d_t,m_t)_t$ be the induced time-dependent mm-space.
\begin{itemize}
\item[(i)]
Then for any $L_1,L_2\in [-\infty,\infty]$
\begin{equation*}
L_1\le\frac1{t-s}\log\frac{d_t}{d_s}\le L_2\quad\Longleftrightarrow\quad
L_1g_t\le \frac12\partial_t g_t\le L_2g_t.\end{equation*}
Moreover, if $(M,g_t)_t$ evolves as Ricci flow then  the previous assertions are equivalent to
\begin{equation}-L_2 g_t\le \Ric_{g_t}\le -L_1g_t.\label{bothRic}\end{equation}
If $(M,g_t)_t$ is a super-Ricci flow then instead we merely have the implications
$$
\frac1{t-s}\log\frac{d_t}{d_s}\le L_2\quad\Longrightarrow\quad -L_2 g_t\le \Ric_{g_t}$$
and
$$L_1\le
\frac1{t-s}\log\frac{d_t}{d_s}\quad\Longleftarrow\quad
 \Ric_{g_t}\le -L_1g_t.$$
 The \emph{proof} is obvious. Similar assertions
holds for the  log-Lipschitz continuity of $t\mapsto m_t$. 

\item[(ii)]
For Ricci flows of  Riemannian manifolds,  we can write
$m_t=e^{-(f_t-f_s)}m_s$ for all $s<t$  with 
$f_t-f_s=\int_s^t \scal_{g_r}dr$.
Thus
\begin{equation*}L_1\le\frac1{t-s}\log\frac{dm_t}{dm_s}\le L_2\quad\Longleftrightarrow\quad
-L_2\le \scal_{g_t}\le -L_1.
\end{equation*}
Super-Ricci flows allow for arbitrary time-dependence of the exponential weight functions $f_t$. Their regularity in time  does not impose any a priori restriction on the metric tensors of the underlying space.

\item[(iii)]
The condition \eqref{bothRic} with finite $L_1,L_2$ rules out Ricci flows running through singularities.
In particular, it will not allow collapsing or changing topological type.
\end{itemize}
\end{remark}
\bigskip

\subsubsection*{Related works.}
Our main results, Theorem \ref{Main} and Theorem \ref{Main-N}, combine and extend two previous -- hitherto unrelated -- lines of developments:
\begin{itemize}
\item results in the setting of
 `smooth' families of time-dependent Riemannian manifolds
which  characterize  solutions to $\Ric+\frac12\partial_t g_t\ge0$ on $I\times M$ (`super-Ricci flows') e.g.\
by means of the monotonicity property {\bf (II)}  in terms of the $L^2$-Wasserstein metric for the dual heat flow,
initiated by work by McCann and Topping \cite{mccanntopping}; for subsequent work in this direction which also includes equivalences with gradient estimates
{\bf (III)} and coupling properties of backward Brownian motions, see e.g. 
Topping \cite{topping2}, Philipowski/Kuwada \cite{kuwada2011coupling,kuwada2011non}, Arnaudon/Coulibaly/Thalmaier \cite{arnaudon2008brownian}, 
Lakzian/Munn \cite{lakzian2012super}, Li/Li \cite{li2015w}.
\item results for (`static') metric measure spaces by Ambrosio/Gigli/Savare \cite{agsbe}  as well as by Erbar/Kuwada/Sturm \cite{eks2014}.
\end{itemize}
Indeed, Theorem \ref{Main} and Theorem \ref{Main-N} extend the main results from 
\cite{agsbe} and from \cite{eks2014} (cf. also \cite{ams}) 
to the time-dependent setting.
Partly, our proofs also provide new and simpler arguments in the static setting, for instance, for the implication {\bf (III$_N$) $\Rightarrow$ (II$_N$)}.
Even though we benefited very much from the powerful, detailed calculus on mm-spaces developed in \cite{agsmet,agsbe,agscalc} and 
pushed forward in \cite{aes,agmr,ams,gigli2014nonsmooth}, 
in many cases we had to develop entirely new strategies and to derive numerous auxiliary estimates and regularity assertions.
 For the proof of implication {\bf (II$_N$) $\Rightarrow$ (III$_N$)}, we followed the argumentation of \cite{bggk2015} and carried over
 their arguments from the static to the dynamic setting.

\smallskip

The analysis of the heat flow on time-dependent spaces (either Dirichlet spaces or metric measure spaces) seems to be completely new.

Even in the smooth case, the characterization  {\bf (I)}  of super-Ricci flows in terms of the so-called dynamic convexity (as introduced in the accompanying paper \cite{sturm2015} by the second author) was not known before.

\subsubsection*{Work in progress} 
The current paper, together with the previous paper by the second author \cite{sturm2015}, will lay the foundations for a broad systematic study of (super-)Ricci flows in the context of mm-spaces with various  subsequent publications in preparation which among others will address the following challenges:
 \begin{itemize}
 \item time-discrete gradient flow scheme  \`a Jordan-Kinderlehrer-Otto for the heat equation and its dual as gradient flows of energy and entropy, resp. \cite{Ko1};
 \item improved dynamic Bochner inequality; $L^p$-gradient and $L^q$-transport estimates;  construction and optimal coupling of Brownian motions on time-dependent mm-spaces  \cite{Ko2}
  \item
 geometric functional inequalities on time-dependent mm-spaces -- in particular, local Poincar\'e, logarithmic Sobolev and dimension-free Harnack inequalities --
  and characterization of super-Ricci flows in terms of them \cite{KoS2};
  \item synthetic approaches to upper Ricci bounds \cite{sturm2017} and rigidity results for Ricci flat metric cones \cite{erbar-sturm2017}.
\end{itemize}

\bigskip

\subsubsection*{Preliminary remarks.}
We use $\partial_t$ as a short hand notation for $\frac{d}{dt}$.
Moreover, we put
$\partial_t^+u(t)=\limsup_{s\to t}\frac1{t-s}(u(t)-u(s))$ and
$\partial_t^-u(t)=\liminf_{s\to t}\frac1{t-s}(u(t)-u(s))$.

In the sequel, $r,s,t$ always denote `time' parameters whereas $a,b$ denote `curve' parameters.

\bigskip

\bigskip

\subsection{Sketch of the Argumentation for the Main Result}

\begin{proof}[The structure of the proof of Theorem \ref{Main-N} is as follows]
In Chapter 4, we present the implications {\bf(I$_N$)} $\Longrightarrow$ {\bf(II$_N$)} and
{\bf(III$_N$)} $\Longrightarrow$ {\bf(II$_N$)} as well as the converse of the latter in the case $N=\infty$.
 Chapter 5 is devoted to the proof of the equivalence {\bf(III$_N$)} $\Longleftrightarrow$ {\bf(IV$_N$)}
 as well as to the proof of the implication {\bf(II$_N$)}$\Longrightarrow$ {\bf(IV$_N$)}.
 
 In Chapter 6 we prove that {\bf(III)} implies the dynamic {\bf EVI} (`evolution variation inequality'). More precisely, we derive two versions,  the dynamic EVI$^-$ and a relaxed form of the dynamic EVI$^+$.
 The combination of these two versions implies that the dual heat flow is the {\bf unique EVI} flow for the Boltzmann entropy.
 
The latter will be proven in a more abstract context in the Appendix (Chapter 7) which is devoted to the study of dynamical EVI-flows in a general framework. Here in particular, it will also be shown that 
 {\bf(III$_N$) \& EVI$^-$} $\Longrightarrow$ {\bf(I$_N$)}.
\end{proof}

Let us now briefly sketch the arguments for each of the implications.

\subsubsection*{{\bf(I$_N$)} $\Longrightarrow$ {\bf(II$_N$)}}
Given two solutions to the dual heat flow $(\mu_r)_{r}$ and $(\nu_{r})_r$, for fixed $t$ we connect the measures $\mu_t= u m_t$ and  $\nu_t= v m_t$ by a $W_t$-geodesic $(\eta^a)_{a\in[0,1]}$ and we choose a  pair of functions  $\phi,\psi$ in duality w.r.t.\ $\frac12W_t^2$ and optimal for the pair $\mu_t,\nu_t$ (`Kantorovich potentials'), see Figure \ref{awesome2}. (Note that in the smooth Riemannian setting the $W_t$-geodesic and the Kantorovich potentials are linked through the relation
$\eta^a=\big(\exp(-a\, \nabla\phi)\big)_*\mu_t=\big(\exp(-(1-a)\, \nabla\psi)\big)_*\nu_t$.)

In the general setting, we deduce with  $u=\frac{d\mu_t}{dm_t}, v=\frac{d\nu_t}{dm_t}$
\begin{itemize}
\item 
$\frac12 \partial_{r}^- W_t^2 (\mu_r,\nu_r)|_{r=t+}\ge {\mathcal E}_t(\phi,u)+ {\mathcal E}_t(\psi,v)$
from Kantorovich duality

\item 
${\mathcal E}_t(\phi,u)+{\mathcal E}_t(\psi,v)\ge
-\partial_aS_t(\eta^{a})\big|_{a=0+}+
\partial_aS_t(\eta^{a})\big|_{a=1-}$
from semiconvexity of $S_t$

\item 
$
\frac12\partial_{r}^- W_r^2 (\mu_t,\nu_t)\big|_{r=t-}\ge
-\partial_aS_t(\eta^{1-})+
\partial_aS_t(\eta^{0+})
+\frac1N\big[S_t(\mu_t)-S_t(\nu_t)\big]^2
$
from the defining property of a super-$N$-Ricci flows.
\end{itemize}
Additing up these estimates yields $\frac12 \partial_{r}^- W_t^2 (\mu_r,\nu_r)|_{r=t+}+\frac12\partial_{r}^- W_r^2 (\mu_t,\nu_t)\big|_{r=t-}\ge\frac1N\big[S_t(\mu_t)-S_t(\nu_t)\big]^2$.
A careful time shift argument allows to replace the left hand side by
$\frac12\partial_{t+}^- W_t^2 (\mu_t,\nu_t)$ which then proves the claim.
 \begin{figure}[h]\centering
   \includegraphics[scale=1]{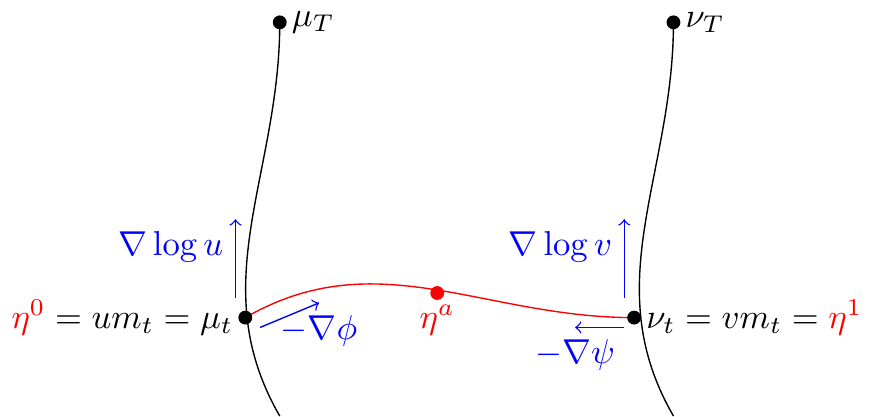}
    \caption{}
    \label{awesome2}
    \end{figure}

\subsubsection*{{\bf(II$_N$)} $\Longrightarrow$ {\bf(IV$_N$)}} 

Given a Lipschitz function  $u$ and a probability density $g$  (w.r.t.\ $m_\tau$) put
$g_r=P^*_{\tau,r}g$, $u_r=P_{r,\sigma}u$
and   $h_r:=\int g_r\Gamma_r(u_r)dm_r$
 for $0<\sigma<r<\tau<T$.

By duality we already know that the transport estimate {\bf (II$_N$)} implies the infinite-dimensional gradient estimate {\bf (III)} which helps us to deduce that 
 \begin{align*}
h_\tau-h_\sigma\geq\int_\sigma^\tau\Big[-2{\bf\Gamma}_{2,r}(u_r)(g_r)+\int \stackrel{\bullet}\Gamma_r(u_r)\,g_r\,m_r
  \Big] \,dr.
 \end{align*}
%
To improve this inequality, we follow the approach initiated by \cite{bggk2015} and consider the perturbation of  $g_\tau$  given by
\begin{equation*}
g_\tau^{\sigma,a}:=g_\tau\Big(1-a[\Delta_\tau u_\sigma+\Gamma_\tau(\log g_\tau,u_\sigma)]\Big)
\end{equation*}
for small $a>0$.
It can be interpreted as the Taylor expansion of the $W_\tau$-geodesic starting in $g_\tau$ with initial velocity $u_\sigma$.
The transport estimate {\bf (II$_N$)} applied to the probability measures
$g_\tau m_\tau$ and $g^{\sigma,a}_\tau m_\tau$ gives us for all 
$a>0$
\begin{eqnarray*}
 W_\sigma^2(\hat P_{\tau,s}(g_\tau m_\tau),\hat P_{\tau,\sigma}(g^{\sigma,a}_\tau m_\tau))&-& W_\tau^2(g_\tau  m_\tau,g^{\sigma,a}_\tau m_\tau)\\
 &\le&-\frac{2}N\int_\sigma^\tau[S_r(\hat P_{\tau,r}(g_\tau m_\tau))-S_r(\hat P_{\tau,r}(g^{\sigma,a}_\tau m_\tau))]^2dr.
\end{eqnarray*}
In the limit $a\searrow0$ we eventually end up with
    \begin{align*}
  h_\tau-h_\sigma\leq  -\frac2N\int_\sigma^\tau &\Big( \int \Delta_r u_r\, g_rdm_r\Big)^2dr.
 \end{align*}  
Together with the previous lower estimate for $h_\tau-h_\sigma$ this proves the claim.

\subsubsection*{{\bf(IV$_N$)} $\Longleftrightarrow$ {\bf(III$_N$)}} This is -- modulo regularity issues -- a simple, well-known (cf. \cite{sturm2015}, Theorem 5.5) differentiation-integration argument
for the function $$r\mapsto \int P^*_{t,r}g\cdot \Gamma_r\big(P_{r,s}u\big)\,dm_r.$$%

\subsubsection*{{\bf(III$_N$)} $\Longrightarrow$ {\bf(II$_N$)}}

Given any `regular' curve $(\mu^a_\tau)_{a\in [0,1]}$ and  $\tau\in I$ we will study the evolution of this curve under the dual heat flow. More precisely, we analyze the growth of the 
  action  
  $${\mathcal A}_t\big(\mu_t^{\cdot}\big):=\int_0^1\big|\dot\mu_t^a\big|_{t}\,da=\int_0^1\int_X
\big|\nabla_t\Phi^a_t\big|^2\,d\mu_t^a\,da$$
   of  the curve $(\mu^a_t)_{a\in [0,1]}$  for  $t<\tau$  where $\mu^a_t=\hat P_{\tau,t}\mu_\tau^a=u_t^am_t$
and $(\Phi^a_t)_{a\in [0,1]}$ denotes the 
velocity potentials in the static space $(X,d_t,m_t)$.
For $s<t$ we 
approximate the action ${\mathcal A}_s\big(\mu_s^{\cdot}\big)$ by 
$$\sum_{i=k}\frac1{a_i-a_{i-1}}W^2_s\big(\mu_s^{a_{i-1}},\mu_s^{a_i}\big),$$
the latter in terms of 
$W_s$-Kantorovich potentials, and finally by means of the interpolating Hopf-Lax semigroup. Applying the 
 Bakry-Ledoux gradient estimate \textbf{(III$_N$)} then allows to estimate
\begin{eqnarray*}
2\varepsilon+ \frac1{t-s}\Big[{\mathcal A}_t(\mu_t^\cdot)-{\mathcal A}_s(\mu_s^\cdot)
\Big]
&\ge& \frac 2{N+\varepsilon}\Big|\int_0^1\int_X\nabla_t\Phi^a_t\cdot \nabla_t\log u^a_t\, d\mu^a_tda\Big|^2\\
&=& \frac 2{N+\varepsilon}\Big|S_t(\mu^1_t)-S_t(\mu^0_t)\Big|^2
 \end{eqnarray*} for each $\varepsilon>0$
 provided that $s$ is sufficiently close to $t$. 
 Passing to the limit $s\uparrow t$ and integrating the result from $s$ to $\tau$ yields
\begin{equation*}
{\mathcal A}_s(\mu_s^\cdot)
\le
{\mathcal A}_\tau(\mu_\tau^\cdot) -\frac{2}N\int_s^\tau \left[ S_t(\mu^0_t)-S_t(\mu^1_t)\right]^2dt.
\end{equation*}
This indeed proves the claim since 
$$W^2_\tau(\mu^0,\mu^1)=\inf\Big\{{\mathcal A}_\tau(\mu_\tau^\cdot) : \ (\mu^a_\tau)_{a\in [0,1]} \mbox{ regular curve connecting }\mu^0,\mu^1\Big\}$$ for any $\mu^0,\mu^1$ and $\tau$
whereas $W^2_s(\mu_s^0,\mu_s^1)\le{\mathcal A}_s(\mu_s^\cdot)$ for all $s<\tau$.

\subsubsection*{{\bf(III$_N$)} $\Longrightarrow$ {\bf(I$_N$)}} 

To deduce the dynamic convexity of the Boltzmann entropy $S_t$, let a $W_t$-geodesic $(\mu^a_t)_{a\in[0,1]}$ be given and consider its evolution $\mu^a_s:=\hat P_{t,s}\mu^a_t$, $s<t$, under the dual heat flow. Then on one hand
\begin{align}\label{geodesic1}
W_t^2(\mu^0_t,\mu^1_t)=\frac1a W_t^2(\mu^0_t,\mu^a_t) +\frac1{1-2a} W_t^2(\mu^a_t,\mu^{1-a}_t)+\frac1a W_t^2(\mu^{1-a}_t,\mu^1_t)
\end{align}
for all $a\in(0,1/2)$, whereas on the other
\begin{align}\label{geodesic2}
W_s^2(\mu^0_t,\mu^1_t)\le\frac1a W_s^2(\mu^0_t,\mu^a_s) +\frac1{1-2a} W_s^2(\mu^a_s,\mu^{1-a}_s)+\frac1a W_s^2(\mu^{1-a}_s,\mu^1_t).
\end{align}
We already know that the  gradient estimate {\bf(III$_N$)} implies the transport estimate  {\bf(II$_N$)} and the latter yields
$$\liminf_{s\nearrow t}\frac1{t-s}\frac1{1-2a}\Big[W_t^2(\mu^a_t,\mu^{1-a}_t)-W_s^2(\mu^a_s,\mu^{1-a}_s)\Big]\ge
\frac2{N}\frac1{1-2a}\Big[S_t(\mu_t^a)-S_t(\mu_t^{1-a})\Big]^2.$$
The EVI-property to be discussed below will allow to estimate 
$$\liminf_{s\nearrow t}\frac1{t-s}\frac1{a}\Big[W_t^2(\mu^0_t,\mu^{a}_t)-W_s^2(\mu^0_t,\mu^{a}_s)\Big]\ge
\frac2{a}\Big[S_t(\mu_t^a)-S_t(\mu_t^{0})\Big]-La W^2_t(\mu_t^0,\mu_t^1),$$
as well as
$$\liminf_{s\nearrow t}\frac1{t-s}\frac1{a}\Big[W_t^2(\mu^{1-a}_t,\mu^{1}_t)-W_s^2(\mu^{1-a}_s,\mu^{1}_t)\Big]\ge
\frac2{a}\Big[S_t(\mu_t^{1-a})-S_t(\mu_t^{1})\Big]-La W^2_t(\mu_t^0,\mu_t^1).$$
Using \eqref{geodesic1} together with \eqref{geodesic2} and adding up the last three inequalities we obtain after letting $a\searrow0$ (see also Figure \ref{trapez}):
\begin{eqnarray*}
\liminf_{s\nearrow t}\frac1{t-s}\Big[W_t^2(\mu^0_t,\mu^{1}_t)-W_s^2(\mu^0_s,\mu^{1}_s)\Big]&\ge&
\frac2{N}\Big[S_t(\mu_t^0)-S_t(\mu_t^{1})\Big]^2\\
&&+2\partial_a^-S_t(\mu_t^a)\big|_{a=0+}-2\partial^+_aS_t(\mu_t^a)\big|_{a=1-}.
\end{eqnarray*}

 \begin{figure}[h]\centering
   \includegraphics[scale=1]{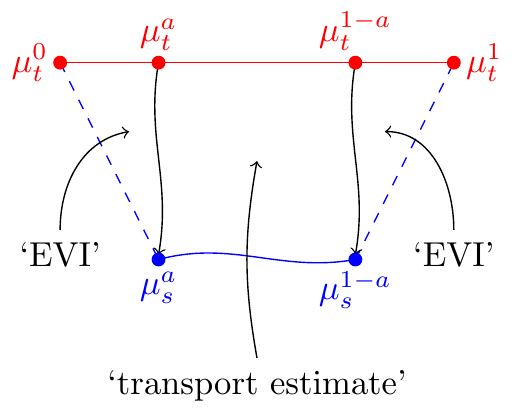}
 \caption{}
      \label{trapez}
    \end{figure}

In order to prove the EVI-property, we follow the approach by \cite{agsbe} and \cite{eks2014} respectively and extend their arguments to the time-dependent setting. We show that the gradient estimate implies that the dual heat flow is a dynamic EVI$^{-}$-gradient flow. For this we introduce in Section \ref{secdynamickant} a dual formulation $\tilde W_{s,t}$ of our time-dependent distance $W_{s,t}$. 

For each fixed $s<t$ we take a regular curve $(\rho^a)_{a\in[0,1]}$ approximating the $W_t$-geodesic joining $\sigma$ and $\mu_t:=\hat P_{\tau,t}\mu$ where $\mu,\sigma\in\mathcal P(X)$ are fixed. We then apply the dual heat flow $\rho_{a,\vartheta}:=\hat P_{t,s+a(t-s)}\rho^a$ to the regular curve, cf. Figure \ref{awesome1}, and eventually show using \textbf{(III)} that
\begin{align*}
\frac12\tilde W_{s,t}^2(\rho_{1,\vartheta},\rho_{0,\vartheta})-(t-s)(S_t(\rho_{1,\vartheta})-S_s(\rho_{0,\vartheta}))
\leq \int_0^1\Big[\frac12|\dot \rho^a|^2_t
+(t-s)^2\int \dot f_{\vartheta(a)}d\rho_{a,\vartheta}\Big]da.
\end{align*}
 
Then, by approximation, we obtain
\begin{align*}
\frac12\tilde W^2_{s,t}(\mu_s,\sigma)-(t-s)(S_t(\sigma)-S_s(\mu_s))
\leq \frac12 W_t^2(\mu_t,\sigma)-(t-s)^2\int_0^1\int\dot f_{\vartheta(a)}d\rho_{a,\vartheta}da.
\end{align*}
In contrast to the static case we obtain the additional error term $(t-s)^2\int_0^1\int\dot f_{\vartheta(a)}d\rho_{a,\vartheta}da$ which however vanishes after dividing by $t-s$ and letting $s\nearrow t$. Thus
\begin{equation*}
\begin{aligned}
S_t(\mu_t)-S_t(\sigma)
\leq \liminf_{s\nearrow t}\frac1{2(t-s)}\left(W_t^2(\mu_t,\sigma)- \tilde W_{s,t}^2(\mu_s,\sigma)\right)
=\frac12 \partial_s^- W_{s,t}^2(\mu_s,\sigma)_{|s=t-}.
\end{aligned}
\end{equation*}
Note that the log-Lipschitz continuity of the distance allows to estimate the last term from above by
$$\frac12\partial_s^- W_{t}^2(\mu_s,\sigma)_{|s=t-}+\frac L2W^2(\mu_t,\sigma).$$

\begin{figure}[h]\centering
   \includegraphics[scale=1]{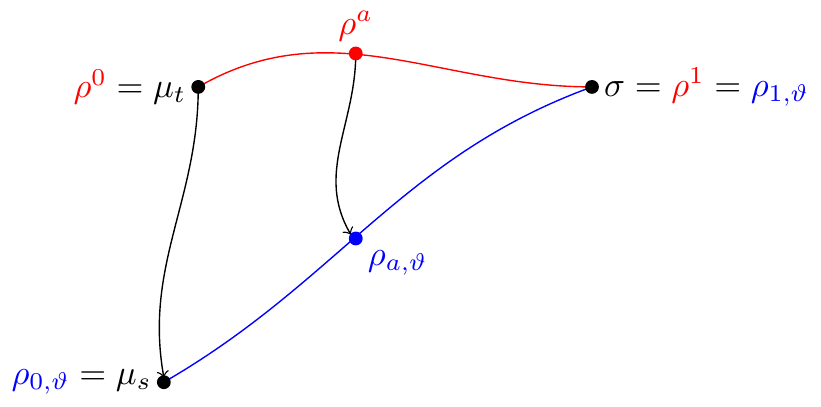}
 \caption{}
      \label{awesome1}
    \end{figure}

\newpage

\section{The Heat Equation  for Time-dependent Dirichlet Forms }

\subsection{The Heat Equation}

Let us choose here a setting which is slightly more general  than for the rest of the paper.
We assume that we are given a Polish space $X$ and 
a $\sigma$-finite reference measure $m_\diamond$ on it which is assumed to have full topological support.
Moreover, we assume that we are given a strongly local Dirichlet form $\E_\diamond$ with  domain ${\F}={\Dom}(\E_\diamond)$ on $\Hil =L^2(X,m_\diamond)$
and
with square field operator $\Gamma_\diamond$ such that
$\E_\diamond(u,v)=\int_X\Gamma_\diamond(u,v)\,dm_\diamond$ for all functions $u,v\in{\F}$.
These objects will be regarded as \emph{reference measure} and \emph{reference Dirichlet form}, resp., in the subsequent definitions and discussions.
The spaces $\Hil$ and $\F$ will be regarded as a Hilbert space equipped with the scalar products $\int uv\, dm_\diamond$ and $\E_\diamond(u,v)+\int uv\, dm_\diamond$, resp.
We identify $\Hil$ with its own dual; the dual of $\F$ is denoted by $\F^*$. Thus we have
$\F\subset\Hil\subset \F^*$ with continuous and dense embeddings.

Recall that  a \emph{Dirichlet form} $\E_\diamond$ on $L^2(X,m_\diamond)$  is a densely defined, nonnegative symmetric form on  $L^2(X,m_\diamond)$ which is closed (which is equivalent to say that the quadratic form is lower semicontinous  on $L^2(X,m_\diamond)$) and
which  satisfies the Markov property
\begin{align*}
\E_\diamond(\xi\circ u)\leq \E_\diamond(u)\quad \text{ for all }\xi\colon\mathbb R\to\mathbb R\text{ 1-Lipschitz such that }\xi(0)=0.
\end{align*} 
Here and in the sequel, the same symbol will be used for a bilinear form and the quadratic form associated with it, i.e.\ 
$\E_\diamond(u)=\E_\diamond(u,u)$.
The Dirichlet form $\E_\diamond$ is called strongly local if $\E_\diamond(u,v)=0$ whenever $(u+c)v=0\, m_\diamond$-a.e.\ for some $c\in\mathbb R$.
We refer to \cite{fukushima} for a comprehensive study of Dirichlet forms and to \cite{Bakry} for the important role of the square field operator.

\medskip

Let $I\subset \R$ be a bounded open interval, say $I=(0,T)$ for simplicity.
In order to deal with time-dependent evolutions, following \cite{sturm1995analysis}
  we consider for $0\le s<\tau\le T$ the Hilbert spaces
$$\F_{(s,\tau)}=L^2\big((s,\tau)\to \F\big)\cap H^1\big( (s,\tau)\to \F^*\big)$$
 equipped with the respective norms
 $\left(\int_s^\tau \|u_t\|^2_\F+\|\partial_t u_t\|^2_{\F^*}\,dt\right)^{1/2}$. According to \cite{renardy2006introduction}, Lemma 10.3, the embeddings
 $\F_{(s,\tau)}\subset {\mathcal C}\big( [s,\tau]\to \Hil\big)$
 hold true
which guarantee that values at $t=s$ and $t=\tau$ are well defined.

\medskip

Moreover,  assume that we are given a 1-parameter family $(m_t)_{t\in(0,T)}$ of measures on $X$ such  that $m_t=e^{-f_t}m_\diamond$ for some bounded measurable function $f$ on $I\times X$ with $f_t\in\F$ and $\exists C$ s.t. $\forall t,x$
\begin{equation}\Gamma_\diamond(f_t)(x)\le C.
\label{f-x}
\end{equation}
The basic ingredient will be a 1-parameter family $(\Gamma_t)_{t\in(0,T)}$ of 
\begin{itemize}
\item
symmetric, positive semidefinite bilinear forms $\Gamma_t$ on $\F$,
each of which has the diffusion property
$$\Gamma_t(\Psi(u_1,\ldots,u_k),v)=\sum_{i=1}^k \Psi_i(u_1,\ldots,u_k)\Gamma_t(u_i,v)$$
$\forall k\in\N, \forall v,u_1,\ldots,u_k\in\F\cap L^\infty(X,m_\diamond), \forall \Psi\in \mathcal C^1(\R^k)$ with $\Psi(0)=0$, \cite{Bakry},
\item and all of them being uniformly comparable ({\lq uniformly elliptic\rq}) w.r.t. the reference form $\Gamma_\diamond$ on $\F$, i.e.  $\exists C$ s.t. $\forall t\in(0,T), \forall u\in\F, \forall x\in X$
\begin{equation}\frac 1C \, \Gamma_\diamond(u)(x)\le \Gamma_t(u)(x)\le C \, \Gamma_\diamond(u)(x). 
\label{G-unif}
\end{equation}
\end{itemize}
For each $t\in (0,T)$ we define a strongly local, densely defined, symmetric Dirichlet form $\E_t$  on $L^2(X,m_t)$ with domain ${\Dom}(\E_t)=\F$ and a self-adjoint, non-positive operator $A_t$ on $L^2(X,m_t)$  with domain ${\Dom}(A_t)\subset\F$  uniquely determined  by the relations
$$\int_X \Gamma_t(u,v)\,dm_t=
\E_t(u,v)=-\int_X A_tu \,v \,dm_t$$
for $u,v\in\F$. Recall that $u\in\Dom(A_t)$ if and only if $u\in\F$ and $\exists C'$ such that $\E_t(u,v)\le C'\cdot \|v\|_{L^2(m_t)}$ for all $v\in \F$. 

\begin{definition} 
A function $u$ is called solution to the heat equation
$$A_t u=\partial_t u\qquad\mbox{on }(s,\tau)\times X$$
if $u\in \F_{(s,\tau)}$ and if for all $w\in \F_{(s,\tau)}$
\begin{equation}\label{def-heat}
-\int_s^\tau \E_t(u_t,w_t)dt
=\int_s^\tau \langle\partial_tu_t, w_te^{-f_t}\rangle_{\mathcal F^*,\mathcal F}\,dt
\end{equation}
where $\langle\cdot,\cdot\rangle_{\mathcal F^*,\mathcal F}=\langle\cdot,\cdot\rangle$ denotes the dual pairing. Note that thanks to \eqref{f-x}, $w\in L^2\big((s,\tau)\to \F\big)$ if and only if $we^{-f}\in L^2\big((s,\tau)\to \F\big)$.
\end{definition}

 Since $u_t\in\Dom(A_t)$ (and thus $\partial_tu_t\in L^2$) for almost every $t$ by virtue of Theorem \ref{energy-est} we may equivalently rewrite the right hand side of the above equation 
 as
 $$\int_s^\tau \langle\partial_tu_t, w_te^{-f_t}\rangle_{\mathcal F^*,\mathcal F}\,dt=
   \int_s^\tau \int_X \partial_tu_t\cdot (w_te^{-f_t})\,dm_\diamond\,dt
   =
   \int_s^\tau \int_X \partial_tu_t\cdot w_t\,dm_t\,dt$$ 
   which allows for a more intuitive, alternative formulation of \eqref{def-heat} as follows:
$$-\int_s^\tau \E_t(u_t,w_t)dt
=\int_s^\tau \int_X \partial_tu_t\cdot w_t\,dm_t\,dt.$$

\begin{theorem} \label{heat}
For all $0\le s<\tau\le T$ and  each $h\in \Hil$ there exists a unique solution  $u\in \F_{(s,\tau)}$ of the heat equation on $(s,\tau)\times X$ with $u_s=h$
(or equivalently with $\lim_{t\searrow s}u_t=h$).
\end{theorem}

\begin{proof}
For each $t$ the bilinear form $\E_t^\diamond$ on $\F$ is defined by
\begin{eqnarray*}
\E_t^\diamond(u,v)&=&-\int_X A_tu\, v\,dm_\diamond\\
&=& \int_X \Gamma_t(u, ve^{f_t})e^{-f_t}\,dm_\diamond\\
&=& \int_X \left[\Gamma_t(u, v)+v\Gamma_t(u,f_t)\right]\,dm_\diamond
\end{eqnarray*}
for $u,v\in\F$.
It immediately follows that $u\in\F_{(s,\tau)}$ is a solution to the heat equation if and only if for all $w\in\F_{(s,\tau)}$
$$-\int_s^\tau \E_t^\diamond(u_t,w_t)dt=\int_s^\tau \int_X \partial_tu_t\cdot w_t\,dm_\diamond \,dt.$$
(Indeed, we simply have to replace the test function $w_t$ by $w_te^{f_t}$.)

Our assumptions on $\Gamma_t$ and $f_t$ guarantee that $\E_t^\diamond$ for each $t$ is a closed coercive form with domain $\F={\Dom}(\E_\diamond)$ on $\Hil=L^2(X,m_\diamond)$, uniformly comparable to $\E_\diamond$.  For each $t$ , the operator $A_t$ is a bounded linear operator from $\F$ to $\F^*$.
Indeed, 
\begin{eqnarray*}
\left\| A_t\right\|_{\F,\F^*}&=&
\sup_{u,v\in\F}  \frac{\big|\E_t^\diamond(u,v)\big|}{\|u\|_\F^{1/2}\cdot\|v\|_\F^{1/2}} \\
&\le&
\sup_{u,v\in\F}  \frac1{\|u\|_\F^{1/2}\cdot\|v\|_\F^{1/2}}
 \int_X \left| \Gamma_t(u, v)\right|\,dm_\diamond
+\sup_{u,v\in\F}  \frac1{\|u\|_\F^{1/2}\cdot\|v\|_\F^{1/2}}
 \int_X \left|v\Gamma_t(u,f_t)\right|\,dm_\diamond\\
 &\le&
 C\left(1+\left\| \Gamma(f_t)\right\|_\infty^{1/2}\right)
\end{eqnarray*}
if $C$ is chosen such that 
$ \left| \Gamma_t(u, v)\right|\le C\cdot  \Gamma_\diamond(u)^{1/2}\cdot \Gamma_\diamond(v)^{1/2}$
for all $u,v$ and $t$.
Thus we may apply the general existence result for solutions to time-dependent operator equations $\partial_tu=A_tu$ on a fixed Hilbert space $\Hil$. For this, we refer to 
\cite{lions2012non}, Chapter III, Theorem 4.1 and Remark 4.3, see also 
\cite{renardy2006introduction}, Theorem 10.3. (Note, however, that the latter assumes a continuity of $t\mapsto A_t$ in operator norm which is not 
really necessary.)
\end{proof}

\begin{remark}
We denote this solution  by $u_t(x)=P_{t,s}h(x)$. Then $(P_{t,s})_{0< s\le t< T}$ is a family of bounded linear operators on $\Hil$ which has the \emph{propagator property}
$$P_{t,r}=P_{t,s}\circ P_{s,r}$$
for all $r\le s\le t$.
For fixed $s$ and $h$ the function $t\mapsto P_{t,s}h$ is continuous in $\Hil$ (due to the embedding 
$\F_{(s,T)}\subset {\mathcal C}\big( [s,T]\to \Hil\big)$). And by construction  the function $(t,x)\mapsto P_{t,s}h(x)$ is a solution to the (forward) heat equation
$\partial_t u=A_tu$
on $(s,T)\times X$. That is,  for all $h\in\Hil$
\begin{equation}\partial_tP_{t,s}h=A_t P_{t,s}h.\end{equation}
Note that the operator $P_{t,s}:\Hil\to\Hil$ in the general time-dependent case is not symmetric -- neither with respect to $m_\diamond$ nor with respect to $m_t$ nor with respect to $m_s$.
\end{remark}
\bigskip

\subsection{The Adjoint Heat Equation}

\begin{definition}  Given $0\le\sigma<t\le T$, a function $v$ is called solution to the adjoint heat equation
$$-A_s v +\partial_s f \cdot v=\partial_s v\qquad\mbox{on }(\sigma,t)\times X$$
if $v\in \F_{(\sigma,t)}$ and if for all $w\in \F_{(\sigma,t)}$
$$\int_\sigma^t \E_s(v_s,w_s)ds+\int_\sigma^t \int_X v_s\cdot w_s\cdot \partial_sf_s\,dm_s\,ds=\int_\sigma^t \int_X \partial_sv_s\cdot w_s\,dm_s\,ds.$$
\end{definition}

\begin{theorem} \label{adj-heat}
Assume \eqref{f-x} and
\begin{equation}|f_t(x)-f_s(x)|\le L\,|t-s|\label{f-t+}.
\end{equation}
\begin{itemize}
\item[(i)]
Given $0\le \sigma<t\le T$,  for each $g\in \Hil$ there exists a unique solution  $v\in \F_{(\sigma,t)}$ of the adjoint heat equation on $(\sigma,t)\times X$ with $v_t=g$. 
\item[(ii)]
This solution can be represented as 
$$v_s=P^*_{t,s}g$$ in terms of a family $(P^*_{t,s})_{s\le t}$ of linear operators on $\Hil$ satisfying the `adjoint propagator property'
$$P^*_{t,r}=P^*_{s,r}\circ P^*_{t,s}\qquad(\forall r\le s\le t).$$
\item[(iii)]
The operators $P_{t,s}$ and $P^*_{t,s}$ are in duality w.r.t.\ each other:
$$\int P_{t,s}h\cdot g\,dm_t=\int h\cdot P^*_{t,s}g\,dm_s\qquad (\forall g,h\in\Hil).$$
\end{itemize}
\end{theorem}

\begin{proof} 
(i), (ii)
The assumption implies that the same arguments used before to prove existence and uniqueness of solutions to the heat equation $\partial_t u=A_tu$ can now be 
applied to prove existence and uniqueness of solutions to the adjoint heat equation
$-\partial_s v=A_s v- (\partial_s f_s)v$. 

(iii) Put $u_t=P_{t,s}h$ and $v_s=P^*_{t,s}g$. Then
\begin{eqnarray*}
\lefteqn{\int u_tv_t\,dm_t-\int u_sv_s\,dm_s}\\&=&
\int_s^t \int \partial_ru_r\, v_r\,dm_r\,dr+\int_s^t \int u_r\, \partial_rv_r\,dm_r\,dr-\int_s^t \int u_r\, v_r\,\partial_rf_r\,dm_r\,dr\\
&=&\int_s^t \E_r(u_r,v_r)\,dr-\int_s^t \E_r(u_r,v_r)\,dr=0.
\end{eqnarray*}
\end{proof}

Note, however, that -- even under the assumption $m_\diamond(X)<\infty$ --
 in general constants will not be solutions to the adjoint heat equation.
Instead of preserving constants, the adjoint heat flow preserves integrals of nonnegative densities.

\begin{lemma} For each fixed $t$, the operators $A_t$ and $A_t^*: u\mapsto A_tu-\partial_t f_t\cdot u$ on $L^2(X,m_t)$ have the same domains: $\Dom(A_t)=\Dom(A_t^*)$
\end{lemma}

\begin{proof}
Recall that $v\in \Dom(A_t^*)$ if and only if $v\in \Dom(\E_t)$ and if there exists a constant $C$ such that for all $u\in \Dom(\E_t)$
$$\E_t(u,v)+\int u\,v\, \partial_t f\,dm_t\le C\cdot \|u\|_{L^2(m_t)}.$$
Boundedness of  $\partial_t f$ implies that this is equivalent to $v\in \Dom(A_t)$.
\end{proof}

In contrast to the form domains, the operator domains $\Dom(A_t)$ in general will depend on $t$.

\begin{example}\label{dom-dep}
Consider $\Hil=L^2(\R,dx)$ with $m_t(dx)=dx$ and 
$$\Gamma_t(u)(x)=\big[1+t\cdot 1_{\R_+}(x)\big]\cdot |u'(x)|^2$$
for $t\in I=(0,1)$.
Then 
$$\Dom(A_t)=\Big\{u\in W^{1,2}(\R)\cap W^{2,2}(\R_-)\cap W^{2,2}(\R_+): \ u'(0-)=(1+t)\cdot u'(0+)\Big\}.$$
Thus $\Dom(A_s)\ne \Dom(A_t)$
for all $s\ne t$. 
\end{example}

\begin{proof}
Obviously, $u\in \Dom(A_t)$ if and only if $u\in W^{1,2}(\R)$ and $[1+t\cdot 1_{\R_+}] u'\in W^{1,2}(\R)$.
\end{proof}

\bigskip

A basic quantity for the subsequent considerations will be the time-dependent Boltzmann entropy.
Here  we put
$S_t(v):=\int_X v\cdot\log v\,dm_t$ and consider it as a time-dependent functional on the space of (not necessarily normalized)  measurable 
functions $v:X\to [0,\infty]$.

\begin{proposition} \label{entropy-est}
\begin{itemize}
\item[(i)]
For all solutions $u\ge0$ to the  heat equation and all $s<t$
$$S_t(u_t)\le e^{L(t-s)}\cdot S_s(u_s).$$

\item[(ii)] For all solutions $v\ge0$ to the adjoint heat equation and all $s<t$
$$S_s(v_s)\le S_t(v_t)+L\int_s^t \int_X v_r\,dm_r\,dr.$$
Note that $\int_Xv_r\,dm_r$ is independent of $r$ if $m_\diamond(X)<\infty$.
\end{itemize}
\end{proposition}

\begin{proof}
In both cases, straightforward calculations yield
\begin{eqnarray*}
e^{Lt}\partial_t\left[e^{-Lt}\int u_t\log u_t\,dm_t\right]&\le&
\int(\log u_t+1)\partial_tu_t\,dm_t
=-\int \Gamma_t(\log u_t)\,u_t\,dm_t\le0
\end{eqnarray*}
and
\begin{eqnarray*}
\partial_s\int v_s\log v_s\,dm_s&=&
\int(\log v_s+1)\partial_sv_s\,dm_s-\int  v_s\log v_s\cdot \partial_sf_s\,dm_s\\
&=&\int \Gamma_s(\log v_s)\,v_s\,dm_s+\int v_s\cdot \partial_sf_s\,dm_s\ge-L\int v_s\,dm_s.
\end{eqnarray*}
\end{proof}

\subsection{Energy Estimates}
Throughout this section,
assume  \eqref{f-x} as well as \eqref{f-t+}
and in addition
\begin{equation}\label{G-lip}
\left| \Gamma_t(u)-\Gamma_s(u)\right|\le 2L\cdot \int_s^t\Gamma_r(u)dr
\end{equation}
for all $u\in\F$ and all $s<t$.

Recall that by definition each solution $u$ to the heat equation on $(s,\tau)\times X$ satisfies 
$u\in L^2\big((s,\tau)\to \F\big)\cap H^1\big( (s,\tau)\to \F^*\big)\subset {\mathcal C}\big((s,\tau)\to \Hil\big)$ and
\begin{equation}\label{uno}
\int_s^\tau  \E_t(u_t)\,dt\le \frac12 \|u_s\|_{L^2(m_s)}^2.\end{equation}
We are now going to prove that these assertions can be improved by one order of (spatial) differentiation.
To do so, we first define 
a self-adjoint, non-positive  operator $\tilde A_t$ on $L^2(X,m_\diamond)$ by
$$-\int_X \tilde A_tu\, v\,dm_\diamond=\tilde\E_t(u,v):=\int_X\Gamma_t(u,v)\,dm_\diamond$$
for all $u,v\in\F$. Then  $\Dom(\tilde A_t)=\Dom(A_t)$ 
and
$$\tilde A_tu= A_tu+\Gamma_t(u,f_t).$$ Indeed,
$-\int A_tu\,v\,dm_\diamond=\int\Gamma_t(u,ve^{f_t})e^{-f_t}\,dm_\diamond=-\int \tilde A_t u\,v\,dm_\diamond+
\int\Gamma_t(u,f_t)v\,dm_\diamond$.
Next, consider the Hille-Yosida approximation 
$\tilde A_t^\delta:=(I-\delta \tilde A_t)^{-1}\tilde A_t$ of $\tilde A_t$ on $L^2(X,m_\diamond)$, put $\tilde\E_t^\delta(u,v):=-\int \tilde A_t^\delta u\, v\,dm_\diamond$ and recall the well-known fact that
 $\tilde\E_t^\delta(u,u)\nearrow \tilde\E_t(u,u)$
 for each $u\in\F$
as $\delta\searrow0$. More generally,

\begin{lemma}
For all $\alpha,\beta>0$ with $\beta-\alpha\le\frac12$: \
 $\F\subset\Dom((I-\delta \tilde A_t)^{-\alpha}\tilde A_t^\beta)$ and for all $u\in\F$:
$$u\in\Dom(\tilde A_t^\beta)\quad\Longleftrightarrow\quad \sup_{\delta>0}\,
\left\| (I-\delta \tilde A_t)^{-\alpha}\tilde A_t^\beta u\right\|_{L^2}<\infty$$
with
$
\left\| (I-\delta \tilde A_t)^{-\alpha}\tilde A_t^\beta u\right\|_{L^2}\nearrow
\left\| \tilde A_t^\beta u\right\|_{L^2}$ for $\delta\searrow0$.
\end{lemma}

\begin{proof}
For fixed $t$ we apply the spectral theorem to the non-negative self-adjoint operator $-\tilde A_t$ on $\Hil$
which yields the representation 
$-\tilde A_t=\int_0^\infty \lambda\, E_\lambda$ in terms of projection operators. For each continuous semi-bounded $\Phi:\R_+\to\R$
$$\Dom\left(\Phi(-\tilde A_t)\right)=\left\{u\in\Hil: \ \int_0^\infty |\Phi(\lambda)|^2dE_\lambda(u,u)\right\}$$
and $(\Phi(-\tilde A_t)u,v)_\Hil= \int_0^\infty \Phi(\lambda) dE_\lambda(u,v)$. Thus, in particular,
$\F=\left\{u\in\Hil: \ \int_0^\infty \lambda dE_\lambda(u,u)\right\}$ and
$$\Dom\left((I-\delta \tilde A_t)^{-\alpha}\tilde A_t^\beta\right)
=\left\{u\in\Hil: \ \int_0^\infty \left|\frac{\lambda^\beta}{(1+\delta\lambda)^\alpha}\right|^2dE_\lambda(u,u)\right\}.$$
Moreover, by monotone convergence as $\delta\searrow0$
$$\left\| (I-\delta \tilde A_t)^{-\alpha}\tilde A_t^\beta u\right\|^2_{L^2}
=\int_0^\infty \left|\frac{\lambda^\beta}{(1+\delta\lambda)^\alpha}\right|^2dE_\lambda(u,u)
\ \nearrow \
\int_0^\infty \lambda^{2\beta}dE_\lambda(u,u)
=
\left\| \tilde A_t^\beta u\right\|^2_{L^2}.$$
\end{proof}

\begin{lemma}
For all $\delta>0$ and all $u,v\in\F$ the map $t\mapsto \tilde\E_t^\delta(u,v)$ is absolutely continuous with
$$\left|\partial_t \tilde\E_t^\delta(u,v)\right|\le \frac L2\left[ \tilde\E_t(u,u)+\tilde\E_t(v,v)\right].$$
\end{lemma}

\begin{proof}
For all $\delta,u,v$ as above, put $u_t^\delta=(I-\delta \tilde A_t)^{-1}u$ and $v_t^\delta=(I-\delta \tilde A_t)^{-1}v$. Then
\begin{eqnarray*}
\partial_t \tilde\E_t^\delta(u,v)&=&
\lim_{\epsilon\to0}\frac1\epsilon
\int \left[
(I-\delta \tilde A_{t+\epsilon})^{-1}\tilde A_{t+\epsilon}u-(I-\delta \tilde A_t)^{-1}\tilde A_{t}u
\right]\cdot v\,dm_\diamond\\
&=&
\lim_{\epsilon\to0}\frac1\epsilon
\int \left[
(I-\delta \tilde A_{t+\epsilon})^{-1}(\tilde A_{t+\epsilon}-\tilde A_t)(1-\delta\tilde A_t)^{-1}u
\right]\cdot v\,dm_\diamond\\
&=&
\lim_{\epsilon\to0}\frac1\epsilon
\left[\tilde\E_t(u_t^\delta, v_{t+\epsilon}^\delta)-\tilde\E_{t+\epsilon}(u_t^\delta, v_{t+\epsilon}^\delta)\right]\\
&\le&\frac L2 \lim_{\epsilon\to0}
 \left[\tilde\E_t(u_t^\delta, u_{t}^\delta)+\tilde\E_{t+\epsilon}(v_{t+\epsilon}^\delta, v_{t+\epsilon}^\delta)\right]\\
&\le&\frac L2 \lim_{\epsilon\to0}
 \left[\tilde\E_t(u, u)+\tilde\E_{t+\epsilon}(v, v)\right]
=\frac L2 
 \left[\tilde\E_t(u, u)+\tilde\E_{t}(v, v)\right].
\end{eqnarray*}
Here we also used the fact that  $\tilde\E_t(u_t^\delta, u_{t}^\delta)\nearrow\tilde\E_t(u_t, u_{t})$ as $\delta\to0$.
\end{proof}

\begin{lemma} There exists a constant $C$ such that for all $0<s<\tau<T$,
for all solutions $u\in\F_{(s,\tau)}$ to the heat equation on $(s,\tau)\times X$ and for all $\delta>0$
\begin{equation}
\int_s^\tau
\int_X \left|(I-\delta\tilde A_t)^{-1/2}\tilde A_tu_t\right|^2
dm_\diamond\,dt
\le C\cdot \left[\E_s(u_s)+ \|u_s\|_{L^2(m_s)}^2\right].
\end{equation}
Thus, in particular, if $u_s\in\F$ then $u_t\in\Dom(\tilde A_t)$ for a.e.\ $t\in(s,\tau)$ and
\begin{equation}
\int_s^\tau
\int_X \left| \tilde A_tu_t\right|^2
dm_\diamond\,dt
\le C\cdot \left[\E_s(u_s)+ \|u_s\|_{L^2(m_s)}^2\right].
\end{equation}
\end{lemma}

\begin{proof} For any $\delta>0$ and $u\in\F$
\begin{eqnarray*}
\tilde\E_s(u_s)&\ge&\tilde\E^\delta_s(u_s)\ge-\int_s^\tau\partial_t \tilde\E^\delta_t(u_t)\,dt
\ge-2\int_s^\tau \E^\delta_t(u_t,\partial_t u_t)\,dt-o_1\\
&=&2
\int_s^\tau
\int_X (I-\delta\tilde A_t)^{-1}\tilde A_tu\cdot
 A_t u_t\,
dm_\diamond\,dt-o_1\\
&=&2
\int_s^\tau
\int_X (I-\delta\tilde A_t)^{-1}\tilde A_t u\cdot
 \tilde A_t u_t\,
dm_\diamond\,dt\\
&&\qquad
-2
\int_s^\tau
\int_X (I-\delta\tilde A_t)^{-1}\tilde A_t u\cdot
 \Gamma_t(u_t,f_t)\,
dm_\diamond\,dt
-o_1\\
&\ge&
\int_s^\tau
\int_X \left|(I-\delta\tilde A_t)^{-1/2}\tilde A_tu\right|^2
dm_\diamond\,dt
-o_1-o_2.
\end{eqnarray*}
Here
\begin{eqnarray*}
o_1&:=&
\int_s^\tau\partial_r \E^\delta_r(u_t)\Big|_{r=t}dt
\le L\int_s^\tau \E_t(u_t)dt\le \frac L2 \|u_s\|_{L^2(m_s)}^2
\end{eqnarray*}
according to the previous Lemma and
\begin{eqnarray*}
o_2&:=&
\int_s^\tau
\int_X \left|(I-\delta\tilde A_t)^{-1/2}\Gamma_t(u_t,f_t)\right|^2
dm_\diamond\,dt\\
&\le& C'\int_s^\tau
\int_X\Gamma_t(u_t)\,e^{-f_t}\,dm_\diamond\,dt\le
\frac {C'}2\|u_s\|_{L^2(m_s)}^2
\end{eqnarray*}
for $C'=\sup_{t}\|\Gamma_t(f_t)e^{f_t}\|_{L^\infty(m_t)}$.
Moreover, $\tilde\E_s(u_s)\le C''
\E_s(u_s)$ for
$C''=\sup_{t}\|e^{f_t}\|_{L^\infty(m_t)}$. Thus the claim follows with
$C=\max\{C'', \frac{L+C'}2\}$.
\end{proof}

\begin{theorem} \label{energy-est}
 For all $0<s<\tau<T$ and for all solutions $u\in \F_{(s,T)}$ to the heat equation
\begin{itemize}
\item[(i)] 
$u_t\in\Dom(A_t)$ for a.e.\ $t\in (s,\tau)$.

\item[(ii)]
If the initial condition $u_s\in\F$ then
$$u\in L^2\big((s,\tau)\to \Dom(A_\cdot\big)\cap H^1\big( (s,\tau)\to \Hil\big).$$
More precisely,
\begin{equation}\label{tre}
e^{-3L\tau}\E_\tau(u_\tau)+2\int_s^\tau e^{-3Lt}\int_X \big| A_tu_t\big|^2\,dm_t\,dt\le  e^{-3Ls}\cdot \E_s(u_s).
\end{equation}

\item[(iii)]For all solutions $v$ to the adjoint heat equation on $(\sigma,t)\times X$ and all $s\in (\sigma,t)$
$$\E_s(v_s)+\|v_s\|^2_{L^2(m_s)}\le e^{3L(t-s)}\cdot\Big[  \E_t(v_t)+\|v_t\|^2_{L^2(m_t)}\Big].$$
Moreover, 
$v_s\in\Dom(A_s)$
for a.e.\ $s\in(\sigma,t)$.
\end{itemize}

\end{theorem}

\begin{proof}
(i): In the case $u_s\in\F$, this follows from the previous Lemma and the fact that $\Dom(A_t)=\Dom(\tilde A_t)$. In the general case $u_s\in\Hil$, by the 
very definition of the heat equation it follows that 
$u_\sigma\in\F$ for a.e.\ $\sigma\in(s,\tau)$.
Applying the previous argument now with $\sigma$ in the place of $s$ yields that $u_t\in\Dom(A_t)$ for a.e.\
$t\in(\sigma,\tau)$ and thus the latter finally holds for  a.e.\
$t\in(s,\tau)$.

(ii):
The log-Lipschitz bound \eqref{G-lip} states $|\partial_t\Gamma_t(.)|\le 2L\cdot \Gamma_t(.)$. 
Together with \eqref{f-t+} this implies
$\partial_s\E_s(u_t)\big|_{s=t}\le 3L\cdot \E_t(u_t)$.
Therefore,
\begin{eqnarray*}
e^{3Lt}\partial_t\left[e^{-3Lt}\E_t(u_t)\right]&\le&\partial_s\E_t(u_s)\big|_{s=t}
=
-2\int |A_t u_t|^2dm_t
\end{eqnarray*}
where the last equality is justified according to (i).

\bigskip

(iii)
Similarly as we did in the previous Lemmas, we can construct a regularization for the adjoint heat equation which will allow to prove that
$v_s\in\Dom(A_s)$
for a.e.\ $s\in(\sigma,t)$. Therefore, we may conclude
\begin{eqnarray*}
\partial_s\E_s(v_s)&\ge&
2\int |A_s v_s|^2dm_s- 3L\cdot \E_s(v_s)-2\int A_sv_s\cdot v_s\cdot \partial_sf_s\,dm_s\\
&\ge&- 3L\cdot \E_s( v_s)-\frac L2\int v_s^2\,dm_s
\end{eqnarray*}
and thus
\begin{eqnarray*}
\partial_s\Big[\E_s(v_s)+\|v_s\|^2_{L^2(m_s)}\Big]&\ge&
- 3L\cdot \E_s (v_s)-\frac L2\int v_s^2\,dm_s\\
&&\qquad+2\int\big[\Gamma_s(v_s)+ v_s^2\cdot \partial_sf_s\big] dm_s-\int v_s^2\cdot \partial_sf_s\,dm_s\\
&\ge&- 3L\cdot \Big[\E_s(v_s)+\|v_s\|^2_{L^2(m_s)}\Big].
\end{eqnarray*}
\end{proof}

\begin{remark}
For fixed $s$ and a.e.\ $\sigma>s$  the operator $P_{\sigma,s}$ maps $\Hil$ into $\Dom(\E)$ and then for a.e.\ $t>\sigma$ the operator $P_{t,\sigma}$ 
maps $\Dom(\E)$ into $\Dom(A_t)$. Thus by  composition, for a.e.\ $t>s$ the operator $P_{t,s}$ maps $\Hil$ into $\Dom(A_t)$.
\end{remark}

A simple restatement of the assertions of the subsequent Proposition \ref{pos-preserv} will yield that
for all $s\le t$ and all $h\in\Hil$
\begin{itemize}
\item $0\le h\le 1\quad \Rightarrow \quad 0\le P_{t,s}h\le 1$
\item $P_{t,s}1=1$ provided $m_\diamond(X)<\infty$
\item $\big(P_{t,s}h\big)^2\le P_{t,s}\big( h^2\big)$.
\end{itemize}

\begin{proposition}\label{pos-preserv}
The following holds true.
\begin{itemize}
\item[(i)] For all solutions $u$ to the heat equation on $(s,\tau)\times X$ and all $t>s$
$$u_s\ge0 \ a.e.\mbox{ on }X\qquad\Longrightarrow\qquad u_t\ge0\ a.e. \mbox{ on $X$.}$$
More generally, for any $M\ge0$
$$u_s\le M \ a.e. \mbox{ on }X\qquad\Longrightarrow\qquad u_t\le M\ a.e. \mbox{ on $X$.}$$
If $m_\diamond(X)<\infty$ then this implication holds for all $M\in\R$.
\item[(ii)] For all solutions $v$ to the adjoint heat equation on $(\sigma,t)\times X$ and all $s<t$
$$v_t\ge0 \ a.e.\mbox{ on }X\qquad\Longrightarrow\qquad v_s\ge0\ a.e. \mbox{ on $X$.}$$
More generally, for any $M\ge0$
$$v_t\le M \ a.e. \mbox{ on }X\qquad\Longrightarrow\qquad v_s\le e^{L(t-s)}M\ a.e. \mbox{ on $X$.}$$
If $m_\diamond(X)<\infty$ then this implication holds for all $M\in\R$.

\item[(iii)] For all solutions $u$ to the heat equation on $(s,\tau)\times X$, all $t>s$ and all $p\in[1,\infty]$
$$\|u_t\|_{L^p(m_t)}\le e^{L/p\cdot(t-s)}\cdot\|u_s\|_{L^p(m_s)}.$$
In particular, $\int u_t\,dm_t\le e^{L(t-s)} \int u_s\,dm_s$ for nonnegative solutions.

\item[(iv)] For all solutions $u,g$ to the heat equation on $(s,\tau)\times X$ and all $t>s$
$$u^2_s\le g_s\ a.e. \mbox{ on }X\qquad\Longrightarrow\qquad u^2_t\le g_t\ a.e. \mbox{on $X$.}$$

\end{itemize}
\end{proposition}

\begin{proof}
(i) Assume that $u$ solves the heat equation. Put $w=(u-M)_+$.  Then for each $t$, strong locality of the Dirichlet form $\E_t$ implies
$$\E_t\big(u_t, (u_t-M)_+\big)=\E_t\big((u_t-M)_+, (u_t-M)_+\big).$$
The chain rule applied to $\Phi(x)=(x)_+$
implies that a.e\ on $(s,T)\times X$
$$\partial_tu_t\cdot(u_t-M)_+=\partial_t(u_t-M)_+\cdot(u_t-M)_+.$$
Therefore, for a.e.\ $t$
\begin{eqnarray*}
0&\le& \E_t\big((u_t-M)_+, (u_t-M)_+\big)=\E_t\big(u_t, (u_t-M)_+\big) \\
&=&-\int \partial_tu_t,(u_t-M)_+e^{-f_t}\, dm_\diamond
=-\int\partial_t(u_t-M)_+(u_t-M)_+e^{-f_t}\, dm_\diamond\\
&\le&-\frac12e^{Lt}\cdot \partial_t\left[e^{-Lt}\int_X (u_t-M)_+^2dm_t\right],
\end{eqnarray*}
where we used \eqref{f-t+} in the last inequality.
Thus $u_s\le M$ will imply $u_t\le M$ for all $t>s$.

In the case, $m_\diamond(X)<\infty$, the constants will be in $\Hil$ and solve the heat equation. Thus the previous argument can also be applied to $u\pm M$
which yields the claim.

(ii) Assume that $v$ solves the adjoint heat equation. Then with a similar calculation as before we obtain for a.e. $s$
\begin{align*}
 &\frac12\partial_s\int(v_s-e^{L(t-s)}M)^2_+\, dm_s\\
 =&\int(v_s-e^{L(t-s)}M)_+\partial_s(v_s-e^{L(t-s)}M)_+\, dm_s-\frac12\int(v_s-e^{L(t-s)}M)^2_+\partial_sf_s\, dm_s\\
 =&\int(v_s-e^{L(t-s)}M)_+(\partial_sv_s+Le^{L(t-s)}M)_+\, dm_s-\frac12\int(v_s-e^{L(t-s)}M)^2_+\partial_sf_s\, dm_s\\
 =&\mathcal E_s(v_s,(v_s-e^{L(t-s)}M)_+)+\int v_s(v_s-e^{L(t-s)}M)_+\partial_sf_s\, dm_s\\
 &+\int(v_s-e^{L(t-s)}M)_+(Le^{L(t-s)}M)_+\, dm_s-\frac12\int(v_s-e^{L(t-s)}M)^2_+\partial_sf_s\, dm_s\\
 \geq& -\frac32L\int(v_s-e^{L(t-s)}M)_+^2\, dm_s.
\end{align*}
Applying Gronwall's inequality yields
\begin{align*}
 \int(v_s-e^{L(t-s)}M)^2_+\, dm_s\leq e^{3L(t-s)}\int(v_t-M)^2_+\, dm_t,
\end{align*}
which proves the claim.

(iii) Assume $p\in (1,\infty)$. (The case $p=\infty$ follows from (i), and the case $p=1$ follows from (ii) by duality.)
Then, by the previous arguments the linear operator
\begin{align*}
 P_{t,s}\colon L^1(m_s)+L^\infty(m_s)\to L^1(m_t)+L^\infty(m_t)
\end{align*}
maps $L^1(m_s)$ boundedly into $L^1(m_t)$ and $L^\infty(m_s)$ boundedly into $L^\infty(m_t)$. Then, by the Riesz-Thorin interpolation theorem $P_{t,s}$ maps
$L^p(m_s)$ boundedly into $L^p(m_t)$ with quantitative estimate
\begin{align*}
 ||P_{t,s}u||_{L^p(m_t)}\leq e^{L(t-s)/p}||u||_{L^p(m_s)}.
\end{align*}

(iv) Choose $w=(u^2-g)_+$. Then, again by the chain rule and since $u$ and $g$ are solutions to the heat equation, we find for a.e. $t$
\begin{eqnarray*}
\frac12e^{Lt}\cdot\partial_t \left[e^{-Lt}\int_X w_t^2dm_t\right]&\le&\int\partial_t(u_t^2-g_t) w_t\, dm_t\\
&=&\int\partial_tu_t(2u_tw_t)\, dm_t-\int\partial_tg_tw_t\, dm_t\\
&=&-\E_t(u_t,2u_tw_t)+\E_t(g_t,w_t)\\
&=&
-\E_t(u^2_t-g_t,w_t)-2\int_X \Gamma_t(u_t,u_t)w_t\,dm_t\\
&=&
-\E_t(w_t,w_t)-2\int_X \Gamma_t(u_t,u_t)w_t\,dm_t\le
0,
\end{eqnarray*}
where we applied the strong locality in the last equation.
Thus 
$$\int w_t^2dm_t\le e^{L(t-s)}\int w_s^2dm_s$$
for all $t>s$.
This proves the claim.

\end{proof}

As a direct consequence we obtain the following corollary.
\begin{corollary}\label{trivial-lp} For all $s<t$
\begin{itemize}
\item[(i)]
$\|P_{t,s}\|_{L^\infty(m_s)\to L^\infty(m_t)}\le 1, \qquad
\|P^*_{t,s}\|_{L^1(m_t)\to L^1(m_s)}\le 1$,

\item[(ii)]
$\|P_{t,s}\|_{L^1(m_s)\to L^1(m_t)}\le e^{L(t-s)},
\qquad
\|P^*_{t,s}\|_{L^\infty(m_t)\to L^\infty(m_s)}\le  e^{L(t-s)}$,

\item[(iii)]
$\|P_{t,s}\|_{L^2(m_s)\to L^2(m_t)}\le e^{L(t-s)/2},
\qquad
\|P^*_{t,s}\|_{L^2(m_t)\to L^2(m_s)}\le e^{L(t-s)/2}$.
\end{itemize}
\end{corollary}
The next result yields that the heat flow is a dynamic EVI$(-L/2,\infty)$-flow for $\frac12$ times the Dirichlet energy $\mathcal E_t$ 
on $L^2(X,m_t)$. For the definition of dynamic EVI-flows we 
refer to Section \ref{appendix}.
\begin{theorem}
\label{l2-est}
\begin{itemize}
\item[(i)] Then the heat  flow is a \emph{dynamic forward EVI${(-L/2,\infty)}$-flow for $\frac12\times$ the Dirichlet energy} on 
$L^2(X,m_t)_{t\in I}$, see Appendix. More precisely, for all solutions $(u_t)_{t\in(s,\tau)}$ to the heat equation,
for all $\tau\le T$ and all   $w\in \Dom(\E)$ 
  \begin{align}\label{evi--heat}
     -\frac12\partial_s^+ \big\|u_s-w\big\|^2_{s,t}\Big|_{s=t}+\frac L4\cdot \big\|u_t-w\big\|^2_{t}
       ~\ge \frac12\E_t(u_t)-\frac12\E_t(w)
  \end{align}
  where $\|.\|_{s,t}$ is  defined  according to  Definition \ref{ddist} with $d_t(v,w)=\big\|v-w\big\|_t=(\int|v-w|^2dm_t)^{1/2}$.
\item[(ii)] The heat flow is uniquely characterized by this property. 
 For all $t>s$ and all solutions to the heat equation
$\|u_t\|_{t} \le e^{L(t-s)/2}\|u_s\|_{s}$.
\end{itemize}
\end{theorem}

\begin{proof} (i) Assumption
\eqref{f-t+} implies $\partial_t \big\|v\big\|^2_{t}\le L\, 
\big\|v\big\|^2_{t}$  as well as (following the argumentation from Proposition \ref{d-diff-2})
$$\partial_s
\big\|v\big\|^2_{s,t}\big|_{s=t}\le\frac L2 \,
\big\|v\big\|^2_{t}$$
for all $v$ and $t$.
Therefore,  we can estimate
\begin{eqnarray*}
 \frac12\partial_s^+ \big\|u_s-w\big\|^2_{s,t}\Big|_{s=t}&\le&
 \limsup_{s\to t}\frac1{2(s-t)}\Big(
  \big\|u_s-w\big\|^2_{t}- \big\|u_t-w\big\|^2_{t} \Big)\\
  &&\quad+
   \limsup_{s\to t}\frac1{2(s-t)}\Big(
  \big\|u_s-w\big\|^2_{s,t} -\big\|u_s-w\big\|^2_{t}\Big)\\
  &\le&
  \langle u_t-w,\partial_tu_t\rangle_t
  +\frac L4\big\|u_t-w\big\|^2_{t}
  \\
  &=&-\E_t(u,u)+\E_t(w,u)+\frac L4\big\|u_t-w\big\|^2_{t}
   \\
  &\le&-\frac12\E_t(u,u)+\frac12\E_t(w,w)+\frac L4\big\|u_t-w\big\|^2_{t}.
\end{eqnarray*}

(ii) Uniqueness and the growth estimate immediately follow from the EVI-property.
Indeed, 
the distance $\big\|.\big\|_{t}$
and the function $\E$ on the
time-dependent geodesic space 
  $L^2(X,m_t)_{t\in I}$ satisfy all assumptions mentioned in the appendix on EVI-flows. In particular,
 the distance is   log-Lipschitz: $\partial_t \big\|v\big\|^2_{t}\le L\, 
\big\|v\big\|^2_{t}$ and the energy
   satisfies the growth bound $\E_s\le C_0\,\E_t$.
\end{proof}

\bigskip

The next lemma states semicontinuity of the heat flow and the adjoint heat flow with respect to the seminorm $\sqrt{\mathcal E}$.
\begin{lemma}\label{P*0}
 Let $u,g\in\Dom(\mathcal E)$, $0<r\leq t<T$. Then
 \begin{align*}
  \lim_{s\nearrow t}P^*_{t,s}g&=g\quad \text{ in }(\Dom(\mathcal E),\sqrt{\mathcal E}),\\
  \lim_{s\searrow r}P_{s,r}u&=u\quad \text{ in }(\Dom(\mathcal E),\sqrt{\mathcal E}).
 \end{align*}
\end{lemma}
\begin{proof}
 Since $P^*_{t,s}g\to g$ in $L^2(X)$ and the Dirichlet energy is lower semicontinuous we have
 \begin{align*}
  \mathcal E_t(g)\leq\liminf_{s\nearrow t}\mathcal  E_t(P^*_{t,s}g).
 \end{align*}
 On the other hand from Theorem \ref{energy-est}(iii)
 \begin{align*}
  \mathcal E_s(P^*_{t,s}g)+||P^*_{t,s}g||_{L^2(m_s)}\leq e^{L(t-s)}(\mathcal E_t(g)+||g||_{L^2(m_t)}),
 \end{align*}
for every $s<t$. Hence, again since $P^*_{t,s}g\to u$ in $L^2(X)$,
\begin{align*}
 \mathcal E_t(g)&\geq \limsup_{s\nearrow t}e^{-L(t-s)}(\mathcal E_s(P^*_{t,s}g)+||P^*_{t,s}g||_{L^2(m_s)})-||g||_{L^2(m_t)}\\
 &\geq \limsup_{s\nearrow t}\mathcal E_s(P^*_{t,s}g)=\limsup_{s\nearrow t}\mathcal E_t(P^*_{t,s}g),
\end{align*}
where the last identity follows from the Lipschitz property of the metrics and the logarithmic densities.
Then, since $\mathcal E_t$ is a bilinear form, the parallelogram identity yields
\begin{align*}
 \limsup_{s\nearrow t}\mathcal E_t(P^*_{t,s}g-g)&=\limsup_{s\nearrow t}(2\mathcal E_t(g)+2\mathcal E_t(P_{t,s}^*g)-\mathcal E_t(u+P^*_{t,s}g))\\
 &\leq4\mathcal E_t(g)-\liminf_{s\nearrow t}\mathcal E_t(g+P^*_{t,s}g))\leq 4\mathcal E_t(g)-\mathcal E_t(2g)\\
 &=0,
\end{align*}
where the last inequality is a consequence of the lower semicontinuity of $\mathcal E_t$.

The second assertion follows along the same lines replacing Theorem \ref{energy-est}(iii) by Theorem \ref{energy-est}(ii).
\end{proof}

\bigskip

\subsection{The Commutator Lemma}
In the static case, generator and semigroup commute. In the dynamic case, this is no longer true. However, we can estimate the error
$$\left|\int_X \big[
A_t (P_{t,s}u)
-
 P_{t,s}(A_s u)
\big]\, v
  \,dm_t\right|.$$
To guarantee well-definedness of all the expressions, we avoid `Laplacians' and use `gradients' instead.

\begin{lemma} For all $\sigma<\tau$, all solutions $u\in\F_{(\sigma,\tau)}$ to the heat equation, and all solutions $v\in\F_{(\sigma,\tau)}$ to the adjoint 
heat equation
\begin{equation}
\left|\E_t(u_t,v_t)-\E_s(u_s,v_s)\right|\le
C(u_s,v_t)\cdot|t-s|^{1/2}
\end{equation} for a.e.\ $s,t\in(\sigma,\tau)$ with $s<t$
where
\begin{equation}
C(u_s,v_t)=
C\cdot \Big[\E_s(u_s)+\E_t(v_t)+\|v_t\|^2_{L^2(m_t)}\Big]
\end{equation}
with $C:=Le^{3(L+1)T}$.
 \end{lemma}
In other words, the commutator lemma states
\begin{equation}
\left|\int_X \big[
A_t (P_{t,s}u_s)
-
 P_{t,s}(A_s u_s)
\big]\, v_t
  \,dm_t\right|\le C(u_s,v_t)\cdot|t-s|^{1/2}.
  \end{equation}

\begin{proof}
Obviously, the function 
$r\mapsto \E_r(u_r,v_r)$ is finite (even locally bounded) and measurable on $(\sigma,\tau)$. Therefore, by Lebesgue's density theorem for  a.e.\ $s,t\in(\sigma,\tau)$
$$\E_t(u_t,v_t)=\lim_{\delta\searrow 0}\frac1\delta\int_{t-\delta}^t\E_r(u_r,v_r)\,dr, \quad
\E_s(u_s,v_s)=\lim_{\delta\searrow 0}  \frac1\delta      \int_s^{s+\delta}\E_r(u_r,v_r)\,dr$$
and thus
\begin{eqnarray*}
\E_t(u_t,v_t)-\E_s(u_s,v_s)&=&
\lim_{\delta\searrow 0}\int^{t-\delta}_s\frac1\delta\Big(\E_{r+\delta}(u_{r+\delta},v_{r+\delta})
-\E_r(u_r,v_r)\Big)
\,dr.
\end{eqnarray*}
To proceed, we decompose the integrand into three terms
\begin{eqnarray*}
\frac1\delta\left[\E_{r+\delta}(u_{r+\delta},v_{r+\delta})-\E_r(u_r,v_r)\right]&= \ &
\frac1\delta\left[\E_{r+\delta}(u_{r+\delta},v_{r+\delta})-\E_{r+\delta}(u_r,v_{r+\delta})\right]\\
&&+\frac1\delta\left[\E_{r+\delta}(u_{r},v_{r+\delta})-\E_r(u_r,v_{r+\delta})\right]\\
&&+\frac1\delta\left[\E_{r}(u_{r},v_{r+\delta})-\E_r(u_r,v_r)\right]\\
&=:&\alpha_r(\delta)+\beta_r(\delta)+\gamma_r(\delta).
\end{eqnarray*}
Let us first estimate the second term
\begin{eqnarray*}
\beta_r(\delta)
&=&\frac1{4\delta}\left[\E_{r+\delta}(u_{r}+v_{r+\delta})+\E_{r+\delta}(u_{r}-v_{r+\delta})-\E_r(u_r+v_{r+\delta})-\E_r(u_r-v_{r+\delta})\right]\\
&\le&\frac{3L}4\, e^{3L\delta}\left[\E_r(u_r+v_{r+\delta})+\E_r(u_r-v_{r+\delta})\right]\\
&\le&\frac{3L}2\, e^{6L\delta}\left[\E_r(u_r)+\E_{r+\delta}(v_{r+\delta})\right]
\end{eqnarray*}
due to the fact that $|\partial_r\E_r(w)|\le 3L\, \E_r(w)$ for each $w\in\F$. According to Theorem \ref{energy-est},
the final expressions can be estimated (uniformly in $\delta$) in terms of $\E_s(u_s)$ and $\E_t(v_t)+\|v_t\|^2_{L^2(m_t)}$. Thus we finally  obtain
\begin{eqnarray*}
\lim_{\delta\searrow0}\int_s^{t-\delta}\beta_r(\delta)\,dr
&\le&
\frac{3L}2\, \int_s^t\left[\E_r(u_r)+\E_{r}(v_{r})\right]dr\\
&\le&(t-s)\,
\frac{3L}2\, e^{3L(t-s)}\,\Big[\E_s(u_s)+\E_t(v_t)+\|v_t\|^2_{L^2(m_t)}\Big].
\end{eqnarray*}
Now let us consider jointly  the first and third terms 
\begin{eqnarray*}
\int_s^{t-\delta}[\alpha_r(\delta)+\gamma_r(\delta)]\,dr&=&\frac1\delta\int_s^{t-\delta}\left[
\E_{r+\delta}\big((u_{r+\delta}-u_r), v_{r+\delta}\big)+
\E_{r}\big(u_r,(v_{r+\delta}-v_{r})\big)\right]\,dr\\
&=&-\frac1\delta\int_s^{t-\delta}\int_X\Big[
(u_{r+\delta}-u_r)\cdot A_{r+\delta}v_{r+\delta}\cdot e^{-f_{r+\delta}}\\
&&\qquad\qquad+A_ru_r\cdot(v_{r+\delta}-v_r)\cdot e^{-f_{r}}\Big]\,dm_\diamond\,dr\\
&=&-\frac1\delta\int_0^\delta\int_s^{t-\delta}\int_X\Big[
A_{r+\epsilon}u_{r+\epsilon}\cdot A_{r+\delta}v_{r+\delta}\cdot e^{-f_{r+\delta}}+\\
&&\qquad\qquad
A_ru_r\cdot
(-A_{r+\epsilon}v_{r+\epsilon}+
\dot f_{r+\epsilon}v_{r+\epsilon})\cdot e^{-f_{r}}\Big]\,dm_\diamond\,dr\,d\epsilon
\end{eqnarray*}
Integrability of $|A_ru_r|^2$ w.r.t.\ $dm_r\,dr$ implies that
$\int_{t-\delta}^t|A_ru_r|^2dm_r\,dr\to0$ as $\delta\to0$ as well as
$\int^{s+\delta}_s|A_ru_r|^2dm_r\,dr\to0$.
 Thus together with Lipschitz continuity  of $t\mapsto f_t$ this implies
\begin{eqnarray*}\frac1\delta\int_0^\delta\int_s^{t-\delta}\int_X\Big[
A_{r+\epsilon}u_{r+\epsilon}\cdot A_{r+\delta}v_{r+\delta}\cdot e^{-f_{r+\delta}}+
-A_ru_r\cdot
A_{r+\epsilon}v_{r+\epsilon}\cdot e^{-f_{r}}\Big]\,dm_\diamond\,dr\,d\epsilon\to0
\end{eqnarray*}
as $\delta\to0$. Thus (since $\dot f$ is bounded by $L$ and
since $r\mapsto \|v_r\|_{L^2(m_r)}$ is non-decreasing)
\begin{eqnarray*}
\lim_{\delta\to0}\left|
\int_s^{t-\delta}[\alpha_r(\delta)+\gamma_r(\delta)]\,dr\right|
&\le&-\frac1\delta\int_0^\delta\int_s^{t-\delta}
\int_X\big|
A_ru_r\cdot\dot
 f_{r+\epsilon}v_{r+\epsilon}\big|\,dm_r\,dr\,d\epsilon\\
&\le& L\cdot|t-s|^{1/2}\cdot \left(\int_s^t\big|
A_ru_r\big|^2\,dm_r\,dr\right)^{1/2}\cdot \|v_t\|_{L^2(m_t)}\\
&\le& L\cdot|t-s|^{1/2}\cdot \left(
\frac12e^{3L(t-s)}\E_s(u_s)
\right)^{1/2}\cdot \|v_t\|_{L^2(m_t)}.
\end{eqnarray*}
To summarize, we have
\begin{eqnarray*}
\big|\E_t(u_t,v_t)-\E_s(u_s,v_s)\big|&=&
\lim_{\delta\searrow 0}\left|\int^{t-\delta}_s\big(
\alpha_r(\delta)+\beta_r(\delta)+\gamma_r(\delta)\big)\,dr
\right|\\
&\le&
|t-s|\,
\frac{3L}2\, e^{3L(t-s)}\,\Big[\E_s(u_s)+\E_t(v_t)+\|v_t\|^2_{L^2(m_t)}\Big]\\
&&
+
L\cdot|t-s|^{1/2}\cdot \left(
\frac12e^{3L(t-s)}\E_s(u_s)
\right)^{1/2}\cdot \|v_t\|_{L^2(m_t)}\\
&\le&
C\cdot|t-s|^{1/2}\cdot \Big[\E_s(u_s)+\E_t(v_t)+\|v_t\|^2_{L^2(m_t)}\Big]
\end{eqnarray*}
with $C:=Le^{3(L+1)T}$ according to the energy estimates of the previous Theorem.
\end{proof}

\newpage

\section{Heat Flow and Optimal Transport on Time-dependent Metric Measure Spaces}

We are now going to define, construct, and analyze the heat equation on \emph{time-dependent metric measure spaces}
$\big(X,d_t,m_t\big)_{t\in I}$.

\subsection{The Setting}\label{gen-ass}

Here and for the rest of the paper, our setting is as follows:

The `state space' 
$X$ is a Polish space and the `parameter set' $I\subset \R$ will be a bounded open interval; for convenience we  assume $I=(0,T)$.
For each $t$ under consideration, 
$d_t$ will be a complete separable geodesic metric on $X$ 
and $m_t$ will be a $\sigma$-finite Borel measure on $X$.
We always assume 
 that there exist  constants  
 $C,K,L,N'\in\R$ such that
 \begin{itemize}
 \item
the metrics $d_t$  are uniformly bounded and equivalent to each other with 
\begin{equation}\label{d-lip}
\left| \log\frac{d_t(x,y)}{d_s(x,y)}\right|\le L\cdot |t-s|
\end{equation}
for all $s,t$ and all $x,y$ (`log Lipschitz continuity in $t$');

\item
the measures $m_t$ are mutually absolutely continuous with bounded, Lipschitz continuous logarithmic  densities; more precisely, choosing some reference 
measure $m_\diamond$ the measures can be represented as $m_t=e^{-f_t}m_\diamond$ with  
functions $f_t$ satisfying $\left| f_t(x)\right|\le C$, $\left| f_t(x)-f_t(y)\right|\le C\cdot d_t(x,y)$ and
\begin{equation}\label{f}
 |f_s(x)-f_t(x)|\le L\cdot |s-t|
\end{equation}
for all $s,t$ and all $x,y$;

 \item
 for each $t$ the static space $(X,d_t,m_t)$ is infinitesimally Hilbertian and satisfies a curvature-dimension condition CD$(K,N')$
 in the sense of \cite{sturm2006}, \cite{lott2009ricci}, \cite{agscalc}.
 \end{itemize}

\bigskip

In terms of the  metric $d_t$ for given $t$, we define the \emph{$L^2$-Kantorovich-Wasserstein  metric} 
$W_t$ on the space of probability measures on $X$:
$$W_t(\mu,\nu)=\inf\left\{\int_{X\times X}d^2_t(x,y)\,dq(x,y):\ q\in\Cpl(\mu,\nu)\right\}^{1/2}$$
where $\Cpl(\mu,\nu)$ as usual denotes the set of all probability measures on $X\times X$ with marginals $\mu$ and $\nu$.
In general, it is not really a metric but just a pseudo metric. Denote by 
$\Pz=\Pz(X)$ the set of all probability measures $\mu$ on $X$ (equipped with its Borel $\sigma$-field) with
$W_t(\mu,\delta_z)<\infty$ for some/all $z\in X$ and $t\in I$.

  The  log-Lipschitz bound \eqref{d-lip} implies that for all $s,t\in I$ and all $\mu,\nu\in\Pz$
\begin{equation}\label{w-lip}
\left|  \log \frac{W_t(\mu,\nu)}{ W_s(\mu,\nu)}\right|\le  L\cdot |t-s|,
\end{equation}
see Corollary 2.2 in \cite{sturm2015}.
Note that the latter is equivalent to weak differentiability of $t\mapsto  W_{t}(\mu,\nu)$ and 
$ |\partial_t W_{t}(\mu,\nu)|
\le L\cdot
W_t(\mu,\nu)$
for all $\mu,\nu\in\Pz$.

A powerful tool is the dual representation of $W_t^2$:
\begin{align*}
 \frac12W_t^2(\mu,\nu)=\sup\left\{\int\varphi d\mu+\int\psi d\nu:\varphi(x)+\psi(y)\leq\frac12d_t^2(x,y)\right\},
\end{align*}
where the supremum is taken among all continuous and bounded functions $\varphi,\psi$.
Closely related to this is the $d_t$-Hopf-Lax semigroup defined on bounded Lipschitz functions
$\varphi$ by 
\begin{align*}
 Q_a^t\varphi(x):=\inf_{y\in X}\left\{\varphi(y)+\frac1{2a}d_t^2(x,y)\right\}, \quad a>0,\ x\in X.
\end{align*}
The map $(a,x)\mapsto Q_a^t\varphi(x)$ satisfies the Hamilton-Jacobi equation
\begin{align}\label{hamilton}
 \partial_a Q_a^t\varphi(x)=-\frac12(\lip_t Q_a^t\varphi)^2(x),\quad \lim_{a\to0}Q_a^t\varphi(x)=\varphi(x).
\end{align}
In addition, since $(X,d_t)$ is assumed to be geodesic,
\begin{align*}
 \Lip(Q_a^t\varphi)\leq2\Lip(\varphi), \quad \Lip(Q_{.}^tf(x))\leq2[\Lip(\varphi)]^2.
\end{align*}
See for instance \cite[Section 3]{agsbe} for these facts.

For $\mu,\nu\in\mathcal P(X)$ the Kantorovich duality can be written as
\begin{align}\label{Kantdual}
 \frac12W_t^2(\mu_0,\mu_1)=\sup_\varphi\left\{\int Q_1^t\varphi d\mu_1-\int\varphi d\mu_0\right\}.
\end{align}
\bigskip

We say that a curve $\mu\colon J\to \mathcal P(X)$ belongs to $AC^p(J;\mathcal P(X))$ if 
\begin{align*}
 W_t(\mu^a,\mu^b)\leq\int_a^b g(r)dr \quad \forall a<b\in J
\end{align*}
for some $g\in L^p(J)$. We will exclusively treat the case $p=2$ and call $\mu$ a \emph{2-absolutely continuous }curve. 
Recall that there exists a minimal function 
$g$, called \emph{metric speed} and denoted by $|\dot\mu_a|_t$ such that
\begin{align*}
 |\dot\mu^a|_t:=\lim_{b\to a}\frac{W_t(\mu^a,\mu^b)}{|b-a|}.
\end{align*}
See for example \cite[Theorem 1.1.2]{ags}. For continuous curves $\mu\in\mathcal C([0,1],\mathcal P(X))$ satisfying $\mu^a=u^am$ with $u^a\leq R$, 
$\mu$ belongs to $AC^2([0,1],\mathcal P(X))$ if and only if for each $t\in(0,T)$ there exists a velocity potential $(\Phi^a_t)_a$ such that 
$\int_0^1\int\Gamma_t(\Phi^a_t)d\mu^a da<\infty$ and
\begin{align}\label{cont.eq.}
 \int\varphi d\mu^{a_1}-\int\varphi d\mu^{a_0}=\int_{a_0}^{a_1}\int\Gamma_t(\varphi,\Phi^a_t)d\mu^a da, \text{ for every }\varphi\in\Dom(\mathcal E).
\end{align}
Moreover we can express the metric speed in the following way
\begin{align}\label{cont.eq.2}
 |\dot\mu^a|_t^2=\int\Gamma_t(\Phi_t^a)d\mu^a.
\end{align}
See section 6 and 8 in \cite{ams} for a detailed discussion.

Occasionally, we have to measure the `distance' between points $x,y\in X$ which belong to different time sheets. In this case, 
for $s,t\in I$ and $\mu,\nu\in \Pz(X)$ we define
\begin{equation*}
W_{s,t}(\mu,\nu):=\inf\lim_{h\to0} \ \sup_{\substack{0=a_0<\dots < a_n=1,\\a_i-a_{i-1}\leq h}}\left\{\sum_{i=1}^n(a_i-a_{i-1})^{-1}W_{s+a_{i-1}(t-s)}^2(\mu^{a_{i-1}},\mu^{a_i})\right\}^{1/2}
\end{equation*}
where the infimum runs over all 2-absolutely continuous curves $\mu\colon[0,1]\to\mathcal P(X)$ with $\mu_0=\mu$, $\mu_1=\nu$.
See Section 6.1 for a detailed 
discussion and in particular for the equivalent characterization
\begin{equation}
W_{s,t}(\mu,\nu)=\inf\left\{\int_0^1|\dot\mu^a|_{W_{s+a(t-s)}}^2da\right\}^{1/2}
\end{equation}
where the infimum runs over all 2-absolutely continuous curves $(\rho^a)_{a\in[0,1]}$ in $\Pz(X)$ connecting $\mu$ and $\nu$.

In the following we will make frequently use of the concept of regular curves, which has already been successfully used in \cite{agsbe,eks2014,ams}. 
We use the refined version of \cite{ams}.
\begin{definition}\label{defregular}
For fixed $t\in[0,T]$, let $\rho^a=u^am_t\in\mathcal P(X)$, $a\in[0,1]$. We say that the curve $\rho$ is regular (w.r.t. $m_t$) if:
\begin{enumerate}
 \item $u\in\mathcal C^1([0,1],L^1(X))\cap\Lip([0,1],\mathcal F^*)$,
 \item there exists a constant $R>0$ such that $u^a\leq R$ $m$-a.e. for every $a\in[0,1]$,
 \item there exists a constant $E>0$ such that $\mathcal E_t(\sqrt{u^a})\leq E$ for every $a\in[0,1]$.
\end{enumerate}

\end{definition}

\begin{remark*}
 Due to our assumptions on the measures, $(\rho^a)_a$ is a regular curve w.r.t $m_t$ if and only if it is also a regular curve w.r.t $m_s$. 
 In this case, it is also a regular curve w.r.t $m_{\vartheta}$, where $\vartheta$ is a function belonging to $\Cz^1([0,1],\mathbb R)$.
 So we will just say regular curve. 
\end{remark*}

We will use the following approximation result which is a combination of \cite[Lemma 12.2]{ams} and \cite[Lemma 4.11]{eks2014}. 
For this we define for a fixed time $t$ the semigroup mollification $h^t_\varepsilon$ given by
\begin{equation}\label{mollifier}
h^t_\varepsilon\psi=\frac1\varepsilon\int_0^\infty H_a^t\psi\kappa\left(\frac{a}\varepsilon\right)da,
\end{equation}
where $(H^t_a)_{a\geq0}$ denotes the semigroup associated to the Dirichlet form $\mathcal E_t$, and $\kappa\in\mathcal C^\infty_c((0,\infty))$ with 
$\kappa\geq0$ and $\int_0^\infty \kappa(a)da=1$.
Recall that for $\psi\in L^2(m_t)\cap L^\infty(m_t)$, $h_\varepsilon^t\psi,\Delta_t(h_\varepsilon^t\psi)\in \Dom(\Delta_t)\cap\Lip_b(X)$. 
Moreover $||h^t_\varepsilon\psi-\psi||\to0$ in $\Dom(\mathcal E)$ as $\varepsilon\to0$ for $\psi\in\Dom(\mathcal E)$.
\begin{lemma}\label{regularcurves}
Let $X$ be a RCD$(K,\infty)$ space.
Let $\rho^0,\rho^1\in\Pz(X)$ and $(\rho^a)_{a\in[0,1]}$ be the $W_t$-geodesic connecting them. Then there exists a sequence of regular curves 
$(\rho^a_n)_{a\in[0,1]}$, $n\in\mathbb N$, such that 
\begin{align}
 &W_t(\rho^a_n,\rho^a)\to 0 \text{ for every }a\in[0,1],\label{reg1}\\
 &\limsup_{n\to\infty}\int_0^1|\dot\rho^a_n|_{t}^2da\leq W_t^2(\rho_0,\rho_1)\label{reg2}.
\end{align}
If we additionally impose that $\rho^0,\rho^1\in\Dom(S)$, then
\begin{align}\label{reg3}
 &S_t(\rho^a_n)\to S_t(\rho^a) \text{ for every }a\in[0,1],
 \end{align}
 and 
 \begin{align}\label{reg4}
 &\limsup_{n\to\infty} \sup_{a\in[0,1]} S_t(\rho^a_n)\leq \sup_{a\in[0,1]} S_t(\rho^a)=\max_{a\in[0,1]} S_t(\rho^a).
 \end{align}
\end{lemma}

\begin{proof}
 We follow the argumentation in \cite[Lemma 12.2]{ams} and approximate $\rho^0,\rho^1$ by two sequences of measures $\{\sigma^i_n\}_n$ with bounded densities.
 Then as in \cite[Proposition 4.11]{agsbe} one employs a threefold regularization procedure to the $W_t$-geodesic $(\nu^a_n)_a$ connecting $\sigma^0_n$ and 
 $\sigma^1_n$: Given $k\in\mathbb N$,
 we first define 
 $\rho^a_{n,k,1}=H_{1/k}^t\nu^a_n$, where $H^t$ denotes the static semigroup. Then we set $\rho^a_{n,k,2}=\int_\mathbb R\rho^{a-a'}_{n,k,1}\chi_{k}(a')da'$,
 where $\chi_k(a)=k\chi(k a)$ for some smooth kernel
 $\chi\in\mathcal C_c(\mathbb R)$. Finally we set $\rho^a_{n,k}=h^t_{1/k}\rho^a_{n,k,2}$, where $h^t_{1/k}$ is given by \eqref{mollifier}. Then by a standard diagonal argument one obtains a sequence of regular curves in the sense of Definition \ref{defregular} 
 satisfying \eqref{reg1} and \eqref{reg2}.
 
 In order to show \eqref{reg3} and \eqref{reg4} note that since $X$ is a RCD$(K,\infty)$ space we have that $a\mapsto S_t(\rho^a)$ is $K$-convex, where $(\rho^a)$
 denotes the $W_t$ geodesic.
 Together with the lower semicontinuity of the entropy the map $a\mapsto S_t(\rho^a)$ is
 continuous. Using the convexity properties we follow the argumentation in \cite[Lemma 4.11]{eks2014} and insert the explicit formulas of the regularization
 $(\rho^a_n)$
 to obtain
 \begin{equation}
 \begin{aligned}\label{reg3,4}
  S_t(\rho^a_{n})&\leq S_t(\rho^a_{n,2})\leq\int_\mathbb R\chi_n(a')S_t(\rho^{a-a'})da'\\
  &\leq S_t(\rho^a)+\int_\mathbb R\chi_n(a')|S_t(\rho^{a-a'})-S_t(\rho^a)|da'.
 \end{aligned}
 \end{equation}
 Since $a\mapsto S_t(\rho^a)$ is uniformly continuous by compactness, the last term vanishes as $n\to\infty$. Thus we obtain 
 $\limsup_{n\to\infty}S_t(\rho^a_n)\leq S_t(\rho^a)$. The lower semicontinuity in turn implies \eqref{reg3}.\\
 One obtains \eqref{reg4} from \eqref{reg3,4} by exploiting the uniform continuity of the entropy along geodesics on compact intervals once more.
\end{proof}

Later on in this paper (Section \ref{sec4.2}), we will see that there is an easier construction of regular curves based on the `dual heat flow' to be introduced next.

\subsection{The Heat Equation on Time-dependent Metric Measure Spaces}

Due to the 
CD$(K,N')$-condition for each of the static spaces $(X,d_t,m_t)$, the detailed analysis of energies, gradients and  heat flows on mm-spaces due 
to Ambrosio, Gigli and Savar\'e \cite{ags,agscalc, agsmet, agsbe} applies. In particular, for each $t$ there is a well-defined energy functional
\begin{equation}
\E_t(u)=\int_X |\nabla_t u|^2\-dm_t=
\liminf_{\stackrel{v\to u\; \mbox{\tiny in}\; L^2(X,m_t)}{v\in \Lip(X,d_t)}}\int_X (\mathrm{lip}_tv)^2\,dm_t
\end{equation}
for $u\in L^2(X,m_t)$ where $\mathrm{lip}_tu(x)$ denotes the pointwise Lipschitz constant (w.r.t. the metric $d_t$) at the point $x$ and $|\nabla_t u|$ denotes the minimal weak upper gradient (again w.r.t. $d_t$).
Since $(X,d_t,m_t)$ is assumed to be infinitesimally Hilbertian,
for each $t$ under consideration  $\E_t$  is a quadratic form.
Indeed, it is a strongly local, regular Dirichlet form with intrinsic metric $d_t$ and square field operator 
\begin{align*}
 \Gamma_t(u)=|\nabla_t u|^2. 
\end{align*}
In the sequel, we freely switch between these two notations of the same object.

The Laplacian $\Delta_t$ is defined as the generator of $\E_t$, i.e. as the unique non-positive self-adjoint operator on $L^2(X,m_t)$ with domain $\mathcal D(\Delta_t)\subset\mathcal D(\E_t)$ and
$$-\int_X \Delta_tu \, v\,dm_t=\E_t(u,v)\qquad(\forall u\in\mathcal D(\Delta_t), v\in \mathcal D(\E_t)).$$
Thanks to the RCD$(K,\infty)$-condition, for each $t$ the domain of the Laplacian coincides with the domain of the Hessian \cite{gigli2014nonsmooth}, 
i.e. $\Dom(\Delta_t)=W^{2,2}(X,d_t,m_t)$. Indeed, the `self-improved Bochner inequality' implies that
$$\Gamma_{2,t}(u)\ge K\,|\nabla_t u|^2+|\nabla^2_t u|_{HS}^2$$
which after integration w.r.t.\ $m_t$, integration by parts, and application of Cauchy-Schwarz inequality gives \begin{equation}\label{dom=dom}
\|\nabla^2_t u\|^2\le (1+K_-/2)\cdot \Big(\|\Delta_t u\|^2 +\|u\|^2 \Big)
\end{equation} 
with $K_-:=\max\{-K,0\}$ and $\|.\|^2:=\|.\|^2_{L^2(m_t)}$.

Note that in general, $\Dom(\Delta_t)$ may depend on $t$, see Example \ref{dom-dep}.

\medskip

Due to our assumptions that the measures are uniformly equivalent and that the metrics are uniformly equivalent, the sets $L^2(X,m_t)$ and $W^{1,2}(X,d_t,m_t):=\mathcal D(\E_t)$ do not depend on $t$ and the respective norms for varying $t$ are equivalent to each other.
We put  $\Hil= L^2(X,m_\diamond)$ and  $\F=\mathcal D(\E_\diamond)$ as well as  $$\F_{(s,\tau)}=L^2\big((s,\tau)\to \F\big)\cap H^1\big( (s,\tau)\to \F^*\big)\subset {\mathcal C}\big([s,\tau]\to \Hil\big)$$ for each $0\le s<\tau\le T$.
For the definition of {\lq solution to the heat equation\rq} and for the existence of the heat propagator we refer to the previous chapter.

\begin{theorem}
(i) For each $0\le s<\tau\le T$ and each $h\in \Hil$ there exists a unique solution  $u\in \F_{(s,\tau)}$ to the heat equation $\partial_tu_t=\Delta_tu_t$ on $(s,\tau)\times X$ with $u_s=h$.

(ii) The heat propagator $P_{t,s}: h\mapsto u_t$ admits a kernel $p_{t,s}(x,y)$ w.r.t.\ $m_s$, i.e.
\begin{equation}
\label{heat-kernel}P_{t,s}h(x)=\int p_{t,s}(x,y) h(y)\, dm_s(y).
\end{equation}
If $X$ is bounded, for each $(s',y)\in (s,T)\times X$ the function $(t,x)\mapsto p_{t,s}(x,y)$ is a solution to the heat equation on $(s',T)\times X$.

(iii) All  solutions $u: (t,x)\mapsto u_t(x)$ to the heat equation on $(s,\tau)\times X$ are H\"older continuous in $t$ and $x$. All nonnegative solutions satisfy a scale invariant parabolic Harnack inequality of Moser type.

(iv) The heat kernel $p_{t,s}(x,y)$ is H\"older continuous in all variables, it is Markovian 
$$\int p_{t,s}(x,y)\,dm_s(y)=1\qquad\quad (\forall s<t, \forall x)$$
and has the propagator property
$$p_{t,r}(x,z)=\int p_{t,s}(x,y)\, p_{s,r}(y,z)\, dm_s(y)\qquad\quad (\forall r<s<t, \forall s,z).$$
\end{theorem}

\begin{proof}
(i) It remains to verify the boundedness and regularity assumptions on $f_t$ and $\Gamma_t$ which were made for Theorem \ref{heat}.
Choose a reference point $t_0\in I$ and put ${\Gamma}_\diamond=\Gamma_{t_0}$. Then ${\mathcal E}_\diamond(u)=\int {\Gamma}_{t_0}(u)e^{-f_{t_0}}d m_\diamond$.
The uniform bounds on $f_t$ and on $\Gamma_\diamond(f_t)$ are stated as assumption \eqref{f}. 
The log Lipschitz bound \eqref{d-lip} on $d_t$  implies the requested uniform bound on $\Gamma_t$.
The claim thus follows from Theorem  \ref{heat}.

(ii), (iii), (iv) The RCD-condition with finite $N'$ implies scale invariant Poincar\'e inequalities and doubling properties for each of the static spaces $(X,d_t,m_t)$
with uniform constants. Together with the uniform bounds on $f_t$, $\Gamma_t(.)$ and $\Gamma_t(f_t)$ this allows to apply results of  
\cite{lierl2015} which provides all the assertions of the Theorem.
\end{proof}

\begin{remark}
The formula \eqref{heat-kernel} allows to  give a pointwise definition for $P_{t,s}h(x)$ for each $h\in L^2(X,m_\diamond)$ (or, in other words, to select a `nice' version) and, moreover, it allows to extend its definition to $h\in L^1\cup L^\infty$. 

Recall, however, that in general the operator $P_{t,s}$ is not symmetric w.r.t. any of the involved measures ($m_t,m_s$ or $m_\diamond$) and that in general  the operator norm in $L^p$ for $p\not=\infty$ will not be bounded by 1.
\end{remark}

\subsection{The Dual Heat Equation}

By duality, the propagator $(P_{t,s})_{s\le t}$ acting on bounded continuous functions induces a \emph{dual propagator} $(\hat P_{t,s})_{s\le t}$ acting on 
probability measures as follows
\begin{equation}
\int u\, d(\hat P_{t,s}\mu)=\int (P_{t,s}u)d\mu\qquad(\forall u\in{\mathcal C}_b(X),  \forall \mu\in\Pz(X)).
\end{equation}
It obviously has the `dual propagator property' $\hat P_{t,r}=\hat P_{s,r}\circ \hat P_{t,s}$.
Whereas the time-dependent function $v_t(x)= P_{t,s}u(x)$ is a solution to the  heat equation
\begin{equation}
\partial_t v=\Delta_t v,
\end{equation}
the time-dependent measure $\nu_s(dy)=\hat P_{t,s}\mu(dy)$ is a solution to the \emph{dual heat equation}
$$-\partial_s \nu=\hat\Delta_s \nu.$$
Here again $\hat \Delta_s$ is defined by duality:
$\int u\, d(\hat \Delta_s\mu)=\int \Delta_s u \,d\mu\quad(\forall u, \forall\mu).$

\medskip

If we define Markov kernels $p_{t,s}(x,dy)$ for $s\le t$ by $p_{t,s}(x,dy)=p_{t,s}(x,y)\,dm_s(y)$ then 
$$P_{t,s}u(x)=\int u(y)p_{t,s}(x,dy)
=\int u(y)p_{t,s}(x,y)\,dm_s(y)$$ and  the dual propagator is given by
$$(\hat P_{t,s}\mu)(dy)=\int p_{t,s}(x,dy)\,d\mu(x)=
\left[\int p_{t,s}(x,y)\,d\mu(x)\right]dm_s(y)
.$$ 
In particular, $(\hat P_{t,s}\delta_x)(dy)=p_{t,s}(x,dy)$.
Note that 
$\hat P_{t,s}\mu(X)=\int P_{t,s}1(x)d\mu(x)=1$.

\begin{theorem} 
(i) For each $0\le\sigma<t\le T$ and each $g\in \Hil$ there exists a unique solution  $v\in \F_{(0,t)}$ to the adjoint heat equation $\partial_sv_s=-\Delta_sv_s+(\partial_s f_s)v_s$ on $(\sigma,t)\times X$ with $v_t=g$.

(ii) This solution is given as $v_s(y)= P^*_{t,s}g(y)$ in term of the adjoint  heat propagator 
\begin{equation}
 P^*_{t,s}g(y)=\int p_{t,s}(x,y) g(x)\, dm_t(x).
\end{equation}
If $X$ is bounded, for each $(t',x)\in (0,t)\times X$ the function $(s,y)\mapsto p_{t,s}(x,y)$ is a solution to the adjoint heat equation on $(0,t')\times X$.

(iii) All  solutions $v: (s,y)\mapsto v_s(y)$ to the adjoint heat equation on $(\sigma,t)\times X$ are H\"older continuous in $s$ and $y$. All nonnegative solutions satisfy a scale invariant parabolic Harnack inequality of Moser type.
\end{theorem}

\begin{proof}
The assumption on Lipschitz continuity of $t\mapsto f_t$ implies that all the regularity assumptions requested in 
\cite{lierl2015} also hold for the time-dependent operators $\Delta_s - (\partial_s f_s)$ (which then are just the operators  $\Delta_s$ perturbed by multiplication operators in terms of bounded functions). Thus all the previous results apply without any changes.
\end{proof}

\begin{corollary} For all $g, h\in L^1(X)$
$$\int h\cdot P^*_{t,s}g\, dm_s=\int P_{t,s}h \cdot g\, dm_t$$
and
\begin{equation}
\hat P_{t,s}\big(g\cdot m_t\big)=\big(P^*_{t,s}g\big)\cdot m_s.
\end{equation}
\end{corollary}

\begin{lemma}\label{p-cont}
(i)
$\hat P_{t,s}$ is continuous on $\Pz(X)$ w.r.t.\ weak convergence.

(ii) The dual heat flow $s\mapsto \mu_s=\hat P_{t,s}\mu$ is uniformly H\"older continuous (w.r.t.\ any of the metrics $W_\tau, r\in I$, see next section). More precisely, there exists a constant $C$ such that for all $s,s'<t$, all $\tau$ and all $\mu$
\begin{equation}
W^2_\tau(\mu_s,\mu_{s'})\le C\cdot |s-s'|.
\end{equation}

(iii) If $X$ is compact then for each  $s<t$
$$\hat P_{t,s}: \Pz(X)\to \D$$
where $\D=\{\mu\in\Pz(X): \ \mu=u\,m_\diamond, \ u\in\F\cap L^\infty,  \ 1/u\in L^\infty\}$.

(iv) For $\mu\in\mathcal P(X)$ such that $\mu\in\Dom(S)$, the dual heat flow $(\hat P_{t,s}\mu)_{s<t}$ belongs to \\
$AC^2([0,t],\mathcal P(X))$.
\end{lemma}

\begin{proof} (i) For each bounded continuous $u$ on $X$ the function $P_{t,s}u$ is bounded continuous. Thus
$\mu_n\to\mu$ implies
$$\int u \, d\hat P_{t,s}\mu_n=\int P_{t,s}u \, d\mu_n\to \int P_{t,s}u \, d\mu=\int u \, d\hat P_{t,s}\mu$$
which proves the requested convergence $\hat P_{t,s}\mu_n\to\hat P_{t,s}\mu$.

(ii) Given 
$\mu_s=\hat P_{t,s}\mu$ and $\mu_{s'}=\hat P_{t,s'}\mu$ for $s<s'<t$. Then
\begin{eqnarray*}
W_\tau^2(\mu_s,\mu_{s'})&\le&\int\int d_\tau^2(x,y)\, p_{s',s}(x,y)\,dm_s(y)\,d\mu_{s'}(x).
\end{eqnarray*}
According to  \cite{sturm1995analysis, lierl2015}, the heat kernel admits upper Gaussian estimates of the form
\begin{eqnarray*}
p_{s',s}(x,y)\le \frac C{m_\tau(B_\tau(\sqrt\sigma, x))}\cdot \exp\Big(-\frac{d^2_\tau(x,y)}{C\sigma}\Big)
\end{eqnarray*}
with  $\sigma:=|s-s'|$ and $B_\tau(r,x)$ denoting the ball of radius $r$ around $x$ in the metric space $(X,d_\tau)$.
Moreover, Bishop-Gromov volume comparison in RCD$(K,N)$-spaces provides an upper bound for the volume of spheres
\begin{eqnarray*}
A(R,x)\le \Big(\frac R{r}\Big)^{N-1}\cdot e^{R\sqrt{|K|(N-1)}}\cdot A(r,x)
\end{eqnarray*}
for $R\ge r$
where $A(r,x)=\partial_{r+} m_\tau(B_\tau(r,x))$ and thus (by integrating from 0 to $\sqrt\sigma$)
\begin{eqnarray*}
A(R,x)\le N\frac {R^{N-1}}{{\sigma}^{N/2}}\cdot e^{R\sqrt{|K|(N-1)}}\cdot m_\tau(B_\tau( \sqrt\sigma,x))
\end{eqnarray*}
for $R\ge \sqrt\sigma$. Hence, we finally obtain
\begin{eqnarray*}
W_\tau^2(\mu_s,\mu_{s'})&\le&\int\int d_\tau^2(x,y)\, p_{s',s}(x,y)\,dm_s(y)\,d\mu_{s'}(x)\\
&\le&\int_X\Big[\frac C{m_\tau(B_\tau(\sqrt\sigma, x))}\cdot \int_X d^2_\tau(x,y)\cdot \exp\Big(-\frac{d^2_\tau(x,y)}{C\sigma}\Big)dm_\tau(y)
\Big]d\mu_{s'}(x)\\
&\le&C\sigma+
C\int_X \int_{\sqrt{\sigma}}^\infty R^2\cdot 
\exp\Big(-\frac{R^2}{C\sigma}\Big)
 N\frac {R^{N-1}}{{\sigma}^{N/2}}\cdot e^{R\sqrt{|K|(N-1)}}\,
dR
\,d\mu_{s'}(x)\\
&\le &C'\cdot \sigma.
\end{eqnarray*}

(iii)
By definition of solution to the adjoint heat equation, the densities $u_s$ of $\hat P_{t,s}\mu$ (w.r.t.\ $m_s$) lie in $\Dom(\E)$. Parabolic Harnack inequality implies continuity and positivity. Together with compactness of $X$ this yields upper and lower bounds (away from 0) for $u$.

(iv)
In a similar calculation as in Proposition \ref{entropy-est}, we find for $\mu=vm_t$, $\mu_s=\hat P_{t,s}\mu$ since the dual heat flow is mass preserving,
\begin{align*}
 \int_s^t\int\Gamma_r(\log v_r)d\mu_rdr&= S_t(\mu)-S_s(\mu_s)-\int_s^t\int v_r\partial_r f_rdm_rdr\\
 &\leq S_t(\mu)+m_t(X)+L(t-s).
\end{align*}
Now choose $\phi\in \Dom(\mathcal E)$ with $\phi,\Gamma(\phi)\in L^\infty(X)$. Then
\begin{align*}
 \left|\int\phi v_tdm_t-\int \phi v_sdm_s\right|&=\left|\int_s^t\mathcal E_r(\phi,v_r)dr\right|\\
 &\leq \int_s^t\left(\int\Gamma_r(\phi)v_rdm_r\right)^{1/2}\left(\int\Gamma_r(\log v_r)v_rdm_r\right)^{1/2}dr\\
 &\leq \int_s^t\left(\int\Gamma_t(\phi)v_rdm_r\right)^{1/2}\left(e^{2L(s-t)}\int\Gamma_r(\log v_r)v_rdm_r\right)^{1/2}dr
\end{align*}
Then, Theorem 7.3 in \cite{aes} yields
\begin{align*}
 |\dot \mu_r|_t^2\leq e^{2L(s-t)}\int\Gamma_r(\log v_r)v_rdm_r\in L^1_{loc}((0,t)),
\end{align*}
where the last conclusion is due to our previous calculation.
\end{proof}

\begin{lemma}\label{P*}
 Let $u,g\in\Dom(\mathcal E)$ and $t\in(0,T)$ with $g\in L^1(X,m_t)$. 
  Then 
 \begin{align*}
  \lim_{h\searrow0}\frac1h\left(\int ugdm_t-\int uP_{t,t-h}^*gdm_{t-h}\right)=\int\Gamma_t(u,g)dm_t
 \end{align*}
 and for a.e. $s<t$
  \begin{align*}
  \lim_{h\searrow0}\frac1h\left(\int u P_{t,s+h}^*gdm_{s+h}-\int uP_{t,s}^*gdm_{s}\right)=\int\Gamma_s(u,P^*_{t,s}g)dm_s
 \end{align*}
\end{lemma}
\begin{proof}
Without loss of generality assume that $g\geq 0$ and $\int g\, dm_t=1$. The general case can be obtained by considering the positive and negative parts separately
and normalization.
 We first prove that for $g\in\Dom(\mathcal E)$ and $u\in\Lip(X)$
\begin{align}\label{eqforlip}
  \frac1h\left(\int ugdm_t-\int uP_{t,t-h}^*gdm_{t-h}\right)=\int_0^1\int\Gamma_{t-rh}(u,P^*_{t,t-rh}g)dm_{t-rh}dr.
 \end{align}
Note that for $0\leq r_1\leq r_2\leq 1$
\begin{align*}
 \left|\int uP_{t,t-r_2h}^*gdm_{t-r_2h}-\int uP_{t,t-r_1h}^*gdm_{t-r_1h}\right|\leq\Lip(u)W_2(\hat P_{t,t-r_2h}(gm_t),\hat P_{t,t-r_1h}(gm_t)),
\end{align*}
and hence, as a consequence of Lemma \ref{p-cont}(ii), the map $r\mapsto \int uP_{t,t-rh}^*gdm_{t-rh}$ is absolutely continuous.
Thus
\begin{align*}
  &\frac1h\left(\int ugdm_t-\int uP_{t,t-h}^*gdm_{t-h}\right)=-\frac1h\int_0^1\partial_r\int uP_{t,t-rh}^*gdm_{t-rh}dr\\
  =&-\frac1h\int_0^1\int ue^{-f_{t-rh}}\partial_rP_{t,t-rh}^*gdm_\diamond-\frac1h\int_0^1\int uP_{t,t-rh}^*g\partial_re^{-f_{t-rh}}dm_\diamond dr\\
  =&\int_0^1\mathcal E^\diamond_{t-rh}(P_{t,t-rh}^*g,ue^{-f_{t-rh}})dr+\int_0^1\int P_{t,t-rh}^*gue^{-f_{t-rh}}\partial_rf_{t-rh}dm_\diamond dr\\
  -&\int_0^1\int P_{t,t-rh}^*gue^{-f_{t-rh}}\partial_rf_{t-rh}dm_\diamond dr\\
  =&\int_0^1\mathcal E^\diamond_{t-rh}(P_{t,t-rh}^*g,ue^{-f_{t-rh}})dr=\int_0^1\mathcal E_{t-rh}(P_{t,t-rh}^*g,u)dr,
\end{align*}
where we used that $r\mapsto P_{t,t-rh}^*g$ is a rescaled solution to the adjoint heat equation. 

Since we assume that the space has a lower Riemannian Ricci bound, we obtain equation \eqref{eqforlip} for every $u\in\Dom(\mathcal E)$ 
by approximating with Lipschitz functions $u_n$, satisfying $u_n\to u$ strongly in $(\Dom(\mathcal E),\sqrt{||\cdot||_{L^2(X)}^2+\mathcal E(\cdot)})$,
see \cite[Proposition 4.10]{agsmet}. Hence
\begin{align*}
 \lim_{h\searrow0}\frac1h\left(\int ugdm_t-\int uP^*_{t,t-h}gdm_{t-h}\right)&=\lim_{h\searrow0}\int_0^1\int\Gamma_{t-rh}(u,P_{t,t-rh}^*g)dm_{t-rh}dr\\
 &=\int_0^1\lim_{h\searrow0}\int\Gamma_{t-rh}(u,P_{t,t-rh}^*g)dm_{t-rh}dr\\
 &=\int\Gamma_{t}(u,g)dm_{t},
\end{align*}
where the third inequality directly follows from Lemma \ref{P*0} and the second equality follows from dominated convergence.

Similarly for the second claim we write for $h<t-s$
 \begin{align*}
 & \frac1h\left(\int u P_{t,s+h}^*gdm_{s+h}-\int uP_{t,s}^*gdm_{s}\right)=\frac1h\int_s^{s+h}\partial_r\int uP_{t,r}^*g\, dm_r\, dr\\
 & =\frac1h\int_s^{s+h}\int\Gamma_r(u,P^*_{t,r}g)dm_r\, dr,
 \end{align*}
 which converges for a.e. $s$ to $\int\Gamma_s(u,P_{t,s}^*g)\, dm_s$ as $h\searrow0$.
\end{proof}
\bigskip

To summarize:
\begin{itemize}
\item[$\triangleright$]
Given any $h\in L^2(X,m_s)$  the function
$(t,x)\mapsto u_t(x)=P_{t,s}h(x)$ solves the \emph{heat equation} $\partial_t u_t=\Delta_t u_t$ in $(s,T) \times X$ with initial condition $u_s=h$.
In Markov process theory, this is the \emph{Kolmogorov backward  equation} (in reverse time direction).
\item[$\triangleright$]
By duality we obtain the \emph{dual propagator} $\hat P_{t,s}$ acting on probability measures. Given any $\nu\in( {\mathcal P}(X),W_t)$, the probability measures
$(s,y)\mapsto \mu_s=\hat P_{t,s}\nu$ solve the \emph{dual heat equation} $-\partial_s \mu_s=\hat \Delta_s \mu_s$ in $[0,t) \times X$ with terminal condition $\mu_t=\nu$.
\item[$\triangleright$]
Their densities $v_s=\frac{d\mu_s}{dm_s}$ solve the
 \emph{Fokker-Planck equation} or \emph{Kolmogorov forward equation} (in reverse time direction)
  $$-\partial_s v_s= \Delta_s v_s{-} \partial_s f_s \cdot v_s$$ in $(0,t) \times X$. The latter is also called \emph{adjoint heat equation}.
\end{itemize}

\newpage

\section{Towards Transport Estimates}

In the sequel, $N$ always will denote an extended number in $(0,\infty]$.
The assumptions from  section \ref{gen-ass} will always be in force (in particular, we assume RCD$^*(K,N')$ and the bounds \eqref{d-lip} and \eqref{f}). 
Moreover, $X$ will be assumed to be bounded (and thus compact).

\subsection{From Dynamic Convexity to Transport Estimates}

\begin{definition}
We say that the time-dependent mm-space $\big(X,d_t,m_t\big)_{t\in I}$  is a \emph{super-$N$-Ricci flow} if the Boltzmann entropy $S$ is  
\emph{dynamical $N$-convex} on  $I \times \Pz$
in the following sense: \  
for a.e.\ $t\in I$ and every  $W_t$-geodesic  $(\mu^a)_{a\in[0,1]}$ in $\Pz$ with
$\mu^0,\mu^1\in \Dom(S)$
\begin{eqnarray}\label{dyn-Nconv-W}
\partial^+_a S_t(\mu^{a})\big|_{a=1-}-\partial^-_a S_t(\mu^{a})\big|_{a=0+}
&\ge&- \frac 12\partial_t^- W_{t-}^2(\mu^0,\mu^1)
+\frac1N \Big| S_t(\mu^0)-S_t(\mu^1)\Big|^2.
\end{eqnarray}
$N$-super Ricci flows in the case $N=\infty$ are simply called super Ricci flows. 
\end{definition}

Recall that 
$\D=\{\mu\in\Pz(X): \ \mu=u\,m_\diamond, \ u\in\F\cap L^\infty,  \ 1/u\in L^\infty\}$.

\begin{proposition}\label{pro1}
Given probability measures $\mu,\nu\in\D\subset\Pz$, then the
$W_t$-geodesic $(\eta^a)_{a\in[0,1]}$ connecting $\mu$ and $\nu$ 
has uniformly bounded densities $\frac{d\eta^a}{d m_t}\le C$ and
there exist $W_t$-Kantorovich potentials $\phi$ from $\mu$ to $\nu$ and  $\psi$ from $\nu$ to $\mu$  (both conjugate to each other) such that
$$\partial_aS_t(\eta^{a})\big|_{a=0+}\ge- {\mathcal E}_t(\phi,u), \qquad
\partial_aS_t(\eta^{a})\big|_{a=1-}\le +{\mathcal E}_t(\psi,v).$$
\end{proposition}

\begin{proof}
This result uses only properties of the static mm-space $(X,d_t,m_t)$. It can be found as estimate (6.19) in the proof of Theorem 6.5 in \cite{agmr}. 
Note that due to our (upper and lower) boundedness assumption on $u,v$, no extra regularization is requested.
\end{proof}

\begin{proposition}\label{pro2}
Given $\tau\le T$ and  $\mu,\nu\in\D\subset\Pz$, put $\mu_t=\hat P_{\tau,t}\mu$ and $\nu_t=\hat P_{\tau,t}\nu$. For each $t\in (0,\tau)$, let $\phi_t$ and $\psi_t$ be any conjugate $W_t$-Kantorovich potentials from $\mu_t$ to $\nu_t$ and vice versa. Then for every $0<r<t<s<\tau$
\begin{equation}
\frac12 \partial_{r}^+ W_t^2 (\mu_r,\nu_r)|_{r=t-}\le {\mathcal E}_t(\phi_t,u_t)+ {\mathcal E}_t(\psi_t,v_t),
\end{equation}
and
\begin{equation}
\frac12\liminf_{\delta\searrow 0}\frac1{\delta}\int_r^s \left[W_t^2(\mu_{t+\delta},\nu_{t+\delta})-W_t^2(\mu_t,\nu_t)\right]\, dt\ge \int^s_r{\mathcal E}_t(\phi_t,u_t)+ {\mathcal E}_t(\psi_t,v_t)\, dt.
\end{equation}

Here $u_t$ and $v_t$ denote the densities  of $\mu_t$ and $\nu_t$, resp., w.r.t. $m_t$.
\end{proposition}

\begin{proof}
We closely follow the argumentation of the proof of Theorem 6.3 in  \cite{agmr}.
According to Proposition \ref{energy-est}, $u_t,v_t\in\Dom(\E)$.
Moreover, due to boundedness of $X$, the Kantorovich potentials $\phi_t$ and $\psi_t$ are Lipschitz and thus also
lie in $\Dom(\E)$.
Since $\phi_t$ and $\psi_t$ are conjugate $W_t$-Kantorovich potentials from $\mu_t$ to $\nu_t$ and vice versa, we get
$$\frac12 W_t^2(\mu_t,\nu_t)=\int \phi_t d\mu_t+\int \psi_td\nu_t$$
whereas
$$\frac12 W_t^2(\mu_r,\nu_r)\ge\int \phi_t d\mu_r+\int \psi_td\nu_r$$
for $r\not= t$. Thus with the help of Lemma \ref{P*} and Theorem \ref{adj-heat} (ii)
\begin{eqnarray*}
\lefteqn{\frac12\limsup_{r\nearrow t}\frac1{t-r} \left[W_t^2(\mu_t,\nu_t)-W_t^2(\mu_r,\nu_r)\right]}\\
&\le&
\limsup_{r\nearrow t}\frac1{t-r}\left[ \int \phi_t [d\mu_t-d\mu_r]
+\int \psi_t [d\nu_t-d\nu_r]\right]\\
&=&{\mathcal E}_t(\phi_t,u_t)+{\mathcal E}_t(\psi_t,v_t).
\end{eqnarray*}
This proves the first claim. With the same notation as before note that $\sup_t\E_t(\phi_t)<\infty$ as well as $\sup_t\E_t(\psi_t)<\infty$ since each $(X,d_t)$ is bounded (Proposition 2.2 in \cite{agmr}). We then find again by Lemma \ref{P*} and Fatou's Lemma
\begin{align*}
&\frac12\liminf_{\delta\searrow 0}\frac1{\delta}\int_r^s\left[W_t^2(\mu_{t+\delta},\nu_{t+\delta})-W_t^2(\mu_t,\nu_t)\right]\, dt\\
&\geq
\liminf_{\delta\searrow 0}\frac1{\delta}\int_r^s\left[ \int \phi_t [d\mu_{t+\delta}-d\mu_t]
+\int \psi_t [d\nu_{t+\delta}-d\nu_t]\right]\, dt
\\
&\geq\int_r^s{\mathcal E}_t(\phi_t,u_t)+{\mathcal E}_t(\psi_t,v_t)\, dt.
\end{align*}
\end{proof}

\begin{theorem}\label{mono-1}
Assume that 
$\big(X,d_t,m_t\big)_{t\in(0,T)}$ is a  super-Ricci flow and that $(\mu_t)_{t\le \tau}$ and  $(\nu_t)_{t\le \tau}$ are  dual heat  flows started in  
probability measures $\mu_\tau, \nu_\tau\in\D$.
Then for a.e. $t\in(0,T)$
$$ \partial_{t} W_t^2 (\mu_t,\nu_t)\ge0.
$$
\end{theorem}

\begin{proof}
The assumptions on the densities are preserved by the dual heat flow, that is, 
$\mu_t$ and $\nu_t$ will 
 have densities in $\Dom(\E)$ which are bounded from above and bounded away from 0, uniformly in $t$.
Using the absolute continuity of $t\mapsto W^2_t(\mu_t,\nu_t)$, we obtain for all $r<s$ 
\begin{eqnarray*}
W^2_{s}(\mu_{s},\nu_{s})-W^2_{r}(\mu_{r},\nu_{r})
&\geq&
 \limsup_{\delta\searrow 0}\int_r^s\frac1{\delta}\Big[
W^2_{t}(\mu_{t+\delta},\nu_{t+\delta})-W^2_{t}(\mu_{t},\nu_{t}) \\
&&\qquad\qquad\qquad\qquad
+ W^2_{t+\delta}(\mu_{t+\delta},\nu_{t+\delta})-W^2_{t}(\mu_{t+\delta},\nu_{t+\delta})\Big]dt\\
&\ge&
 \liminf_{\delta\searrow 0}\int_r^s\frac1{\delta}\Big(
W^2_{t}(\mu_{t+\delta},\nu_{t+\delta})-W^2_{t}(\mu_{t},\nu_{t}) \Big)dt\\
&&
+ \liminf_{\delta\searrow 0}\frac1{\delta}
\int_r^s\Big(W^2_{t+\delta}(\mu_{t+\delta},\nu_{t+\delta})-W^2_{t}(\mu_{t+\delta},\nu_{t+\delta})\Big)dt
\\
&\geq&
\int_r^s 2\Big(\E_t(u_t,\phi_t)+\E_t(v_t,\psi_t) \Big)dt\\
&&
+ \liminf_{\delta\searrow 0}\frac1{\delta}
\int_r^s\Big(W^2_{t}(\mu_{t},\nu_{t})-W^2_{t-\delta}(\mu_{t},\nu_{t})\Big)dt
\\
&\ge&
\int_r^s 2\Big(\E_t(u_t,\phi_t)+\E_t(v_t,\psi_t) \Big)dt\\
&&\qquad-
\int_r^s 2\Big(\E_t(u_t,\phi_t)+\E_t(v_t,\psi_t) \Big)dt\ge0,
\end{eqnarray*}
where we used Proposition \ref{pro2} in the third inequality while  the fourth inequality is due to Proposition \ref{pro1} and the definition of super-Ricci flow, i.e.
\begin{align*}
-\frac12\partial_{r}^- W_r^2 (\mu_t,\nu_t)\big|_{r=t-}\le
\partial_aS(\eta_t^{1-})-
\partial_aS(\eta_t^{0+})
\end{align*}
for every $W_t$-geodesic
$(\eta_t^b)_{b\in[0,1]}$ connecting $\mu_t$ and $\nu_t$.
In the previous argumentation, we used in the third and fourth inequality that $\frac1\delta [W^2_{t+\delta}-W_t^2]$ is uniformly bounded,
which is due to the log-Lipschitz bound on the distances.
\end{proof}

\begin{corollary}\label{mono-2}
Assume that 
$\big(X,d_t,m_t\big)_{t\in(0,T)}$ is a  super-Ricci flow and that $(\mu_t)_{t\le \tau}$ and  $(\nu_t)_{t\le \tau}$ are  dual heat flows started in points 
$\mu_\tau$ and $\nu_\tau\in\Pz$, resp., for some $\tau\in(0,T]$.
Then for all  $0\le s<t\le \tau$
\begin{equation}\label{w-mono}
W_s (\mu_s,\nu_s)\le W_t (\mu_t,\nu_t).
\end{equation}
\end{corollary}

\begin{proof}
For measures  $\mu_\tau, \nu_\tau$ with densities in $\Dom(\E)$ which are bounded from above and bounded away from 0 the estimate \eqref{w-mono} immediately 
follows from the previous theorem
 and the fact that the map $t\mapsto W_t(\mu_t,\nu_t)$ is absolutely continuous (Lemma \ref{p-cont}).
 
 The set of such probability measures is dense in $\Pz$ (w.r.t.\ weak topology) and according to Lemma \ref{p-cont},
 $\hat P_{t,s}$ is continuous on $\Pz$. Thus the estimate \eqref{w-mono} carries over to all $\mu_\tau,\nu_\tau\in\Pz$.
\end{proof}

\begin{theorem}[``{\bf(I$_N$)} $\Rightarrow$ {\bf(II$_N$)}'']\label{mono-3}
Assume that 
$\big(X,d_t,m_t\big)_{t\in(0,T)}$ is a  super-$N$-Ricci flow and that
probability measures  $\mu_\tau, \nu_\tau\in\Pz$ are given for some  $\tau\in (0 ,T]$. Then the dual heat flows 
$(\mu_t)_{t\le \tau}$ and  $(\nu_t)_{t\le \tau}$ starting in these points satisfy
for all $0\le s<t\le \tau$
\begin{equation}\label{II.N}
W^2_s (\mu_s,\nu_s)\le W^2_t (\mu_t,\nu_t)-\frac2N\int_s^t \left[ S_r(\mu_r)-S_r(\nu_r)\right]^2dr.
\end{equation}
\end{theorem}

\begin{proof}
For measures  $\mu_\tau, \nu_\tau$ within the subset $\D$  we follow the proof of the previous Theorem \ref{mono-1} line by line and finally use the enforcement of the super Ricci flow property to deduce
\begin{eqnarray*}
-\frac12 \liminf_{\delta\searrow0}\frac1{\delta} \Big[W^2_{t+\delta}(\mu_{t+\delta},\nu_{t+\delta})-W^2_{t}(\mu_{t+\delta},\nu_{t+\delta})\Big]
&\le&
\partial_a
S_t(\eta_t^{1-})- \partial_aS_t(\eta_t^{0+})\\
&&-\frac1N \left[ S_t(\mu_t)-S_t(\nu_t)\right]^2.
\end{eqnarray*}
Together with the other estimates from the proof of the previous theorem this gives
\begin{align*}
W^2_{s}(\mu_{s},\nu_{s})-W^2_{t}(\mu_{t},\nu_{t})
&\leq-\frac2N \int_s^t\left[ S_r(\mu_r)-S_r(\nu_r)\right]^2\, dr.
\end{align*}

For general $\mu_\tau, \nu_\tau\in\Pz$ we apply the previous result to the pair $\mu_t, \nu_t\in\D$
(cf. Lemma \ref{p-cont}) which already yields the claim for all $0\le s<t< \tau$. The claim for $t=\tau$ now follows by approximation
\begin{eqnarray*}
W^2_s (\mu_s,\nu_s)&\le& W^2_t (\mu_t,\nu_t)-\frac2N\int_s^t \left[ S_r(\mu_r)-S_r(\nu_r)\right]^2dr\\
&\to& W^2_\tau (\mu_\tau,\nu_\tau)-\frac2N\int_s^\tau \left[ S_r(\mu_r)-S_r(\nu_r)\right]^2dr
\end{eqnarray*} as $t\uparrow\tau$. Here the convergence of the integrals is obvious. The convergence of the first term on the right-hand side follows from 
Lemma \ref{p-cont}.
\end{proof}

\subsection{From Gradient Estimates to Transport Estimates}\label{sec4.2}

\begin{theorem}[``{\bf(III$_N$)} $\Rightarrow$ {\bf(II$_N$)}'']\label{from III to II}
Assume that 
$\big(X,d_t,m_t\big)_{t\in(0,T)}$ satisfies the Bakry-Ledoux gradient estimate (III$_N$) for the primal heat flow. Then the dual heat flow
starting in arbitrary points $\mu^0_\tau,\mu^1_\tau\in\mathcal P(X)$ satisfies
for all $0< s< \tau<T$
\begin{equation}\label{II-N}
W^2_s (\mu^0_s,\mu^1_s)\le W^2_\tau (\mu^0_\tau,\mu^1_\tau)-\frac{2}N\int_s^\tau \left[ S_t(\mu^0_t)-S_t(\mu^1_t)\right]^2dt.
\end{equation}
\end{theorem}

\begin{proof}
(i) Given $\tau\in I$ and a regular curve  (see chapter 3) $(\mu^a_\tau)_{a\in [0,1]}$, define of each $t\le\tau$ the $W_t$-\emph{action}
$${\mathcal A}_t\big(\mu_t^{\cdot}\big)=\sup \,\left\{\sum_{i=1}^k
\frac1{a_i-a_{i-1}}\,W^2_t\big(\mu_t^{a_{i-1}},\mu_t^{a_{i}}\big): \ k\in \N, \ 0=a_0<a_1<\ldots<a_k=1\right\}$$
of the curve $a\mapsto \mu_t^a=\hat P_{\tau,t}\mu_\tau^a$.
Let $t\in(0,\tau]$ be given with ${\mathcal A}_t\big(\mu_t^{\cdot}\big)<\infty$. In other words, such that the curve $a\mapsto \mu_t^a$ is 2-absolutely 
continuous.
(Obviously, this is true for $t=\tau$. The subsequent discussion indeed will show that this holds for all $t\le\tau$.)
Let $(u^a_t)_{a\in [0,1]}$ and $(\Phi^a_t)_{a\in [0,1]}$ denote the densities and velocity potentials for the curve $(\mu^a_t)_{a\in [0,1]}$ 
(see \cite[Theorem 8.2]{ams}, or \eqref{cont.eq.},\eqref{cont.eq.2}) in the static space $(X,d_t,m_t)$. Then, in particular,
$${\mathcal A}_t\big(\mu_t^{\cdot}\big)=\int_0^1\big|\dot\mu_t^a\big|_{W_t}\,da=\int_0^1\int_X
\big|\nabla_t\Phi^a_t\big|^2\,d\mu_t^a\,da.$$

Given $s\in (0,t)$ and  $\epsilon>0$ choose  bounded Lipschitz functions $-\varphi^0_s,\varphi^1_s$ which are in $W_s$-duality to each other such that
\begin{eqnarray*}
W_s^2(\mu_s^0,\mu_s^1)
&\le&
2\Big[\int_X\varphi_s^1d\mu_s^1-
 \int_X\varphi_s^0d\mu_s^0
 \Big]+\epsilon(t-s) 
\end{eqnarray*}
and let $(\varphi^a_s)_{a\in [0,1]}$ denote the Hopf-Lax interpolation of $\varphi^0_s,\varphi^1_s$ in the static space $(X,d_s,m_s)$.

Then applying the continuity equation \eqref{cont.eq.} and the Hamilton-Jacobi equation \eqref{hamilton} yields
\begin{eqnarray*}
\epsilon&+& \frac1{t-s}\Big[{\mathcal A}_t(\mu_t^\cdot)-W_s^2(\mu_s^0,\mu_s^1)
\Big]\\
&\ge&
 \frac1{t-s}\int_0^1\big| \dot\mu_t^a\big|^2da -\frac2{t-s}\Big[\int_X\varphi_s^1d\mu_s^1-
 \int_X\varphi_s^0d\mu_s^0
 \Big]
 \\
&=&
\frac1{t-s}\int_0^1\Big[\int_X\big|\nabla_t\Phi_t^a\big|^2d\mu^a_t-2\partial_a\int_XP_{t,s}\varphi_s^ad\mu_t^a\Big]da
 \\ 
 &=&
 \frac1{t-s}\int_0^1\int_X\Big[\big|\nabla_t\Phi^a_t-\nabla_tP_{t,s}\varphi_s^a
 \big|^2 -\big|\nabla_tP_{t,s}\varphi_s^a
 \big|^2+P_{t,s}\big|\nabla_s\varphi_s^a\big|^2\Big]d\mu^a_tda \\
  &\ge&
 \frac1{t-s}\int_0^1\int_X\big|\nabla_t\Phi^a_t-\nabla_tP_{t,s}\varphi_s^a
 \big|^2d\mu^a_t\,da\\
 &&\qquad+
 \frac 2{N(t-s)}\int_s^t\int_0^1\int_X\Big[P_{t,r}\Delta_r P_{r,s}\varphi_s^a\Big]^2d\mu^a_tda\,dr
 \
 \ge0
\end{eqnarray*}
where for the second last inequality we have used the Bakry-Ledoux gradient estimate (III$_N$).

In the case $N=\infty$ this already proves the claim. Indeed, since $\epsilon>0$ was arbitrary it states that 
$$W_s^2(\mu_s^0,\mu_s^1)\le  {\mathcal A}_\tau(\mu_\tau^\cdot)$$
for any regular curve $(\mu^a_{\tau})_{a\in [0,1]}$.
Given any $\mu_\tau^0,\mu_\tau^1\in\Pz(X)$ we can choose regular curves $(\mu^a_{\tau,n})_{a\in [0,1]}$ for $n\in\N$
such that
${\mathcal A}_\tau(\mu_{\tau,n}^\cdot)\to W_\tau^2(\mu_{\tau}^0,\mu_\tau^1)$
and
$W_\tau(\mu_{\tau,n}^0,\mu_\tau^0)\to0$
as well as
$W_\tau(\mu_{\tau,n}^1,\mu_\tau^1)\to0$
for $n\to\infty$.
According to Lemma \ref{p-cont}, the latter also implies
$W_s(\mu_{s,n}^0,\mu_s^0)\to0$
as well as
$W_s(\mu_{s,n}^1,\mu_s^1)\to0$
for $n\to\infty$ where $\mu^a_{s,n}:=\hat P_{\tau,s}\mu^a_{\tau,n}$.
Together with the previous estimate (applied with $t=\tau$ to the regular curves $(\mu^a_{\tau,n})_{a\in [0,1]}$) we obtain
$$W_s^2(\mu_s^0,\mu_s^1)=\lim_{n\to\infty}
W_s^2(\mu_{s,n}^0,\mu_{s,n}^1)\le\lim_{n\to\infty}
{\mathcal A}_\tau(\mu_{\tau,n}^\cdot)=
W_\tau^2(\mu_\tau^0,\mu_\tau^1).$$
 This is the claim.

Moreover, applying this monotonicity result to each pair $\mu_\tau^{a_{i-1}},\mu_\tau^{a_{i}}$ of points on the initial regular curve selected by an arbitrary 
partition $(a_i)_{i=1,\ldots,k}$  yields
$${\mathcal A}_s(\mu_s^\cdot)\le {\mathcal A}_\tau(\mu_\tau^\cdot)$$
for all $s\le\tau$. 
In particular, this implies that the previous argumentation is valid for all $t\le\tau$.

\bigskip

(ii)
Moreover, the previous estimates for given $s,t,\epsilon$ can be tightened up by choosing $k\in\N$ and $(a_i)_{i=1,\ldots,k}$
as well as for $i=1,\ldots,k$ suitable  bounded Lipschitz functions $-\varphi^{0,i}_s,\varphi^{1,i}_s$ which are in $W_s$-duality to each other and which 
are `almost maximizers' of the dual representation of $W^2_s\big(\mu_s^{a_{i-1}},\mu_s^{a_{i}}\big)$
such that

\begin{eqnarray*}
\epsilon&+& \frac1{t-s}\Big[{\mathcal A}_t(\mu_t^\cdot)-{\mathcal A}_s(\mu_s^\cdot)
\Big]\\
&\ge&\epsilon/2+ \frac1{t-s}\Big[{\mathcal A}_t(\mu_t^\cdot)-
\sum_{i=1}^k
\frac1{a_i-a_{i-1}}\,W^2_s\big(\mu_s^{a_{i-1}},\mu_s^{a_{i}}\big)
\Big]\\
&\ge&
 \frac1{t-s}\int_0^1\big| \dot\mu_t^a\big|^2da -\frac2{t-s}\sum_{i=1}^k
\frac1{a_i-a_{i-1}}\Big[\int_X\varphi_s^{1,i}d\mu_s^1-
 \int_X\varphi_s^{0,i}d\mu_s^0
 \Big]
 \\
&=&
\frac1{t-s}\int_0^1\Big[\int_X\big|\nabla_t\Phi_t^a\big|^2d\mu^a_t-2\partial_a\int_XP_{t,s}\varphi_s^{a,k}d\mu_t^a\Big]da
 \\ 
 &=&
 \frac1{t-s}\int_0^1\int_X\Big[\big|\nabla_t\Phi^a_t-\nabla_tP_{t,s}\varphi_s^{a,k}
 \big|^2 -\big|\nabla_tP_{t,s}\varphi_s^{a,k}
 \big|^2+P_{t,s}\big|\nabla_s\varphi_s^{a,k}\big|^2\Big]d\mu^a_tda \\
  &\ge&
 \frac1{t-s}\int_0^1\int_X\big|\nabla_t\Phi^a_t-\nabla_tP_{t,s}\varphi_s^{a,k}
 \big|^2d\mu^a_t\,da\\
 &&\qquad+
 \frac 2{N(t-s)}\int_s^t\int_0^1\int_X\Big[P_{t,r}\Delta_r P_{r,s}\varphi_s^{a,k}\Big]^2d\mu^a_tda\,dr
 \
 \
 =:(\alpha)
\end{eqnarray*}
The function $\varphi_s^{a,k}$ here is obtained for $a\in (a_{i-1},a_i)$  by Hopf-Lax interpolation of the Lipschitz functions  $\varphi_s^{a_{i-1}+,k}:=\frac1{a_i-a_{i-1}}\varphi_s^{0,i}$ and $\varphi_s^{a_i-,k}:=\frac1{a_i-a_{i-1}}\varphi_s^{1,i}$.
\medskip

Now let us choose $t$ to be a Lebesgue density point of $t\mapsto \int_0^1 \E_t(P_{t,s}\varphi_s^a, P_{\tau,t}^*u_\tau^a)\,da$.
Then for $s$ sufficiently close to $t$ the commutator lemma (applied to time points $r$ and $t$) implies that 
\begin{eqnarray*}
\Big[\frac 1{(t-s)}\int_s^t\int_0^1\int_X P_{t,r}\Delta_r P_{r,s}\varphi_s^{a,k}d\mu^a_tda\,dr\Big]^2
\ge
\Big[\frac 1{(t-s)}\int_s^t\int_0^1\int_X \Delta_t P_{t,s}\varphi_s^{a,k}d\mu^a_tda\,dr\Big]^2
- \epsilon\cdot N/2.
\end{eqnarray*}
Let us also briefly remark that the densities $u^a_t$ of the measures $\mu^a_t$ are bounded away from 0, uniformly in $a$ (due to the smooth dependence on $a$ of the measures in the regularized curve we started with) and locally uniformly in $t$ (due to the parabolic Harnack inequality for solutions to the adjoint heat equation).
In particular, in the subsequent calculations the singularity of the logarithm at 0 does not matter.
Thus applying Young' inequality $(a-b)^2\geq \frac{\delta}{1+\delta}a^2-\delta b^2$ where $\delta=N/\varepsilon$
\begin{eqnarray*}
(\alpha)
&=&
 \frac1{t-s}\int_0^1\int_X\big|\nabla_t\Phi^a_t-\nabla_tP_{t,s}\varphi_s^{a,k}\big|^2
 d\mu^a_t\,da+
 \frac 2N\Big|\int_0^1\int_X\nabla_tP_{t,s}\varphi_s^{a,k}\cdot \nabla_t\log u^a_t\, d\mu^a_tda\Big|^2-\epsilon\\
 &\ge&
 \frac 2{N+\epsilon}\Big|\int_0^1\int_X\nabla_t\Phi^a_t\cdot \nabla_t\log u^a_t\, d\mu^a_tda\Big|^2-\epsilon\\
 &&\quad+\Big[\frac1{t-s}-\frac2\epsilon\int_X\big|\nabla_t\log u^a_t
 \big|^2
 d\mu^a_t\,da
 \Big]\cdot \int_X\big|\nabla_t\Phi^a_t-\nabla_tP_{t,s}\varphi_s^{a,k}\big|^2
 d\mu^a_t\,da\\
 &\ge&
 \frac 2{N+\epsilon}\Big|\int_0^1\int_X\nabla_t\Phi^a_t\cdot \nabla_t\log u^a_t\, d\mu^a_tda\Big|^2-\epsilon
 \
 =:(\beta)
\end{eqnarray*}
provided $s$ is sufficiently close to $t$.
Finally, using the continuity equation for the curve  $(\mu^a_t)_{a\in [0,1]}$ (and its velocity potentials $\Phi^a_t$) we obtain
\begin{eqnarray*}
(\beta)
 &=&
 \frac 2{N+\epsilon}\Big|S_t(\mu^1_t)-S_t(\mu^0_t)\Big|^2-\epsilon.
\end{eqnarray*}
Passing to the limit $s\nearrow t$ yields
\begin{eqnarray*}
\epsilon&+& \partial^-_{t-}{\mathcal A}_t(\mu_t^\cdot)
\ge
 \frac 2{N+\epsilon}\Big|S_t(\mu^1_t)-S_t(\mu^0_t)\Big|^2-\epsilon
\end{eqnarray*}
and thus (since $\epsilon>0$ was arbitrary)
\begin{eqnarray}\label{act-diff}
 \partial^-_{t-}{\mathcal A}_t(\mu_t^\cdot)
\ge
 \frac 2{N}\Big|S_t(\mu^1_t)-S_t(\mu^0_t)\Big|^2.
\end{eqnarray}
Recall that this holds for a.e.\ $t\in(0,\tau)$. 
Moreover, note that $t\mapsto {\mathcal A}_t(\mu_t^\cdot)$ is absolutely continuous. Indeed, by Lemma \ref{p-cont} and the log-Lipschitz assumption 
\eqref{d-lip}
\begin{eqnarray*}
\Big|W^2_{t+\epsilon}(\mu_{t+\epsilon}^a,\mu_{t+\epsilon}^b)-W^2_t(\mu_t^a,\mu_t^b)\big|&\le&
\Big|W^2_{t+\epsilon}(\mu_{t+\epsilon}^a,\mu_{t}^b)-W^2_t(\mu_t^a,\mu_t^b)\big|\\
&&+\Big|W^2_{t}(\mu_{t+\epsilon}^a,\mu_{t+\epsilon}^b)-W^2_t(\mu_t^a,\mu_t^b)\big|\\
&\le&2L\epsilon \,e^{2L\epsilon}\,W^2_t(\mu_t^a,\mu_t^b)\\
&&
+\frac{2\sqrt\epsilon}{1-2\sqrt\epsilon}W_t^2(\mu_t^a,\mu_t^b)+\frac1{\sqrt\epsilon}W^2_t(\mu_{t+\epsilon}^a,\mu_t^a)++\frac1{\sqrt\epsilon}W^2_t(\mu_{t+\epsilon}^b,\mu_t^b)\\
&\le&C_0 \sqrt\epsilon\,W^2_t(\mu_t^a,\mu_t^b)+C_1\sqrt\epsilon.
\end{eqnarray*}
Thus we may integrate \eqref{act-diff}  from any $s\in (0,\tau)$ to $\tau$ to obtain
\begin{equation}\label{II-act}
{\mathcal A}_s(\mu_s^\cdot)
\le
{\mathcal A}_\tau(\mu_\tau^\cdot) -\frac{2}N\int_s^\tau \left[ S_t(\mu^0_t)-S_t(\mu^1_t)\right]^2dt.
\end{equation}

Finally, given arbitrary $\mu^0_\tau,\mu^1_\tau\in\Pz(X)$
the subsequent lemma provides a construction of 2-absolutely continuous, regular curves $(\tilde\mu^a_\sigma)_{a\in[0,1]}$ connecting $\mu^0_\sigma,\mu^1_\sigma$ for a.e.\ $\sigma<\tau$ with
$${\mathcal A}_\sigma(\tilde\mu_\sigma^\cdot)\to W^2_\tau(\mu^0_\tau,\mu^1_\tau)$$
as $\sigma\nearrow \tau$. 
Carrying out the previous estimations, finally resulting in \eqref{II-act},
with $(\tilde\mu^a_\sigma)_{a\in[0,1]}$
 in the place of $(\mu^a_\tau)_{a\in[0,1]}$ yields 
\begin{eqnarray*}
 W^2_s(\mu^0_s,\mu^1_s)&\le&
 {\mathcal A}_s(\tilde\mu_s^\cdot)\\
 &\le&
 {\mathcal A}_\sigma(\tilde\mu_\sigma^\cdot)-\frac{2}N\int_s^\sigma \left[ S_t(\mu^0_t)-S_t(\mu^1_t)\right]^2dt\\
 &\to&
 W^2_\tau(\mu^0_\tau,\mu^1_\tau)
-\frac{2}N\int_s^\tau \left[ S_t(\mu^0_t)-S_t(\mu^1_t)\right]^2dt.
\end{eqnarray*}
This proves the claim.  
\end{proof}

\begin{lemma}
(i) Assume {\rm(III)} (with $N=\infty$) and let $(\mu^a)_{a\in[0,1]}$ be an arbitrary $W_\tau$-geodesic in $\Pz(X)$.
Let $\chi$ be a standard convolution kernel on $\R$. Then for a.e.\ $t<\tau$ and every $\delta>0$
the measures 
$$\mu_t^{a,\delta}:=\int_\R \left(\hat P_{\tau,t}\mu^{\vartheta(a+\delta b)}\right)\chi(b)db=
 \hat P_{\tau,t}\left(\int_\R\mu^{\vartheta(a+\delta b)}\chi(b)db\right)
$$
constitute  a regular curve $(\mu_t^{a,\delta})_{a\in[0,1]}$ (in the sense of Definition \ref{defregular}).
Here
 $\vartheta(a)=0$ for $a\in [0,\delta]$, 
$\vartheta(a)=1$ for $a\in [1-\delta,1]$, and $\vartheta(a)=\frac{a-\delta}{1-2\delta}$ for $a\in [\delta,1-\delta]$.

Choosing $t_n\nearrow \tau$ and $\delta_n\searrow 0$ yields a sequence of regular curves satisfying 
\eqref{reg1} - \eqref{reg4}.
In addition, for these approximations the endpoints are simply given by the dual heat flow:
$$\mu_{t_n}^{a,\delta_n}=\hat P_{\tau,t_n}\mu^a$$
for $a=0$ as well as $a=1$ and for all $n$.
\end{lemma}

\begin{proof}
The re-parametrization by means of $\vartheta$ forces the curve to be constant for some short interval around the endpoints and  squeeze it in-between. The 
latter leads to a moderate increase of the metric speed. The former guarantees that the endpoints remain unchanged under the convolution. The convolution w.r.t.\ the kernel $\chi$ guarantees smooth dependence on $a$, i.e.\ (1) of Def \ref{defregular}.  \eqref{reg1} follows from Lemma \ref{p-cont}. Smoothness in $a$ (thanks to the convolution) and H\"older continuity in $(t,x)$ (being a solution to the adjoint heat equation) guarantee uniform boundedness of $u^a_t(x)$ for $(a,t,x)\in[0,1]\times(0,t]\times X$ for each $t<\tau$, i.e.\ (2) of Def \ref{defregular}.
Moreover, $u_t^a(x)$ is uniformly bounded away from 0.
Thus  (3) of Def \ref{defregular} is equivalent to a uniform bound for the energy $\E_t(u^a)$.

Boundedness of $u^a_r$ for $r<\tau$ implies 
$$\int_0^1 \int_0^r\E_t(u^a_t)\,dt\,da\le\frac12 \int_0^1 \|u^a_r\|^2_{L^2(m_r)}da<\infty.$$
Thus for a.e.\ $t<\tau$
$$\int_0^1 \E_t(u^a_t)da<\infty \quad\mbox{and}\quad \E_t(u^0_t)<\infty,\quad \E_t(u^1_t)<\infty.$$
Convolution w.r.t.\ the kernel $\chi$ thus turns the integrable function
$a\mapsto \E_t\left(u^{\vartheta(a)}_t\right)$ into a bounded function:
$\int_\R\E_t\left(u^{\vartheta(a+\delta b)}_t\right)\chi(b)db\le C.$
Since the energy $u\mapsto \E_t(u)$ is convex,
Jensen's inequality implies
$$
\E_t\left(\int_\R u^{\vartheta(a+\delta b)}_t\,\chi(b)db\right)\le
\int_\R\E_t\left(u^{\vartheta(a+\delta b)}_t\right)\chi(b)db\le C.$$

The action estimate \eqref{reg2} follows from part (i) of the previous proof. Indeed, the dual heat flow decreases the action. Also convolution in the 
$a$-parameter decreases the action. The re-parametrization increases the action by a factor bounded by $\frac1{(1-2\delta)^2}$.

The entropy estimates \eqref{reg3} and \eqref{reg4} follow as in the proof of Lemma \ref{regularcurves}
\end{proof}

\subsection{Duality between Transport and Gradient Estimates in the Case $N=\infty$}

In  the subsequent chapter, we will  prove the implication
{\bf(II$_N$)} $\Rightarrow$ {\bf(III$_N$)} by composing the results
{\bf(II$_N$)} $\Rightarrow$ {\bf(IV$_N$)} and
{\bf(IV$_N$)} $\Rightarrow$ {\bf(III$_N$)}.
Partly, these arguments are quite involved. (And actually, for the last one, we freely make use of the subsequent Theorem \ref{IIeqIII}).

Here we
present a direct, much simpler proof in the particular case $N=\infty$. 
Indeed, this proof will yield a slightly stronger statement: the equivalence of the respective estimates for given pairs $s,t$. 
See also \cite{kuwada2015space} for a related result.

\begin{theorem}[``{\bf(II)} $\Leftrightarrow$ {\bf(III)}''] \label{IIeqIII}
For fixed $0<s<t< T$ the following are equivalent:
\begin{itemize}
\item[(II)$_{t,s}$]
For all   $\mu,\nu\in\Pz$
\begin{equation}\label{II}
W_s(\hat P_{t,s}\mu, \hat P_{t,s}\nu)\le W_t(\mu,\nu)
\end{equation}
\item[(III)$_{t,s}$]
For all $u\in\Dom(\E)$
\begin{equation}\label{III}
\Gamma_t(P_{t,s}u)\le P_{t,s}(\Gamma_s(u))\quad m\mbox{-a.e. on }X.
\end{equation}
\end{itemize}
\end{theorem}

\begin{proof}
``(II)$_{t,s}$ $\Rightarrow$ (III)$_{t,s}$'': \ 
Given a bounded Lipschitz function $u$ on $X$, points $x,y\in X$, and a $d_t$-geodesic $(\gamma^a)_{a\in[0,1]}$ connecting $x$ and $y$, put
$\mu^a_t=\delta_{\gamma^a}$ and $\mu^a_s=\hat P_{t,s}\mu^a_t$.
The transport estimate $W_s(\mu_s^a, \mu^b_s)\le W_t(\mu_t^a,\mu_t^b)$ implies that
$$\big|\dot\mu_s\big|_{W_s}\le \big|\dot\mu_t\big|_{W_t}=\big|\dot\gamma\big|_{d_t}=d_t(x,y).$$
Thus following the argumentation from \cite{agsmet}, Theorem 6.4, we obtain
\begin{eqnarray*}
\Big| P_{t,s}u(x)-P_{t,s}u(y)\Big|&=&
\Big| \int u\, d\hat P_{t,s}\delta_x-\int u\, d\hat P_{t,s}\delta_y\Big|\\
&\le&
\int_0^1
\Big( \big|\nabla_s u\big|^2d\mu_s^a\Big)^{1/2}\cdot
\big|\dot\mu_s\big|_{W_s}
da\\
&\le&
\int_0^1
\Big(P_{t,s} \big|\nabla_s u\big|^2(\gamma^a)\Big)^{1/2}\cdot
\big|\dot\gamma\big|_{d_t}da\\
&\le&
d_t(x,y)\cdot \sup\Big\{ P_{t,s} \big|\nabla_s u\big|^2(z): \ d_t(x,z)+d_t(z,y)=d_t(x,y)\Big\}.
\end{eqnarray*}
The H\"older continuity of $z\mapsto P_{t,s} \big|\nabla_s u\big|^2(z)$, therefore, allows to conclude that $(P_{t,s} \big|\nabla_s u\big|^2)^{1/2}$ is an 
upper gradient for $P_{t,s}u$. This proves the claim for bounded Lipschitz functions. 
The extension to $u\in\Dom(\E)$ follows as in \cite{agsmet}.

\bigskip

``(III)$_{t,s}$ $\Rightarrow$ (II)$_{t,s}$'': \  previous Theorem.
 \end{proof}

\newpage

\section{From Transport Estimates  to Gradient Estimates and Bochner Inequality}

As before, for the sequel a time-dependent mm-space $(X,d_t,m_t)_{t\in I}$ will be given such that
\begin{itemize}
\item for each $t\in I$ the static space satisfies the RCD$^*(K,N')$ condition for some finite  numbers $K$ and $N'$
\item
the distances are bounded and log-Lipschitz in $t$, that is,
$|\partial_t d_t(x,y)|\le L\cdot d_t(x,y)$ for some $L$ uniformly in $t,x,y$
(existence of $\partial_t d_t$ for a.e.\ $t$)
\item $f$ is $L$-Lipschitz in $t$ and $x$.
\end{itemize}
\subsection{The Bochner Inequality}
\subsubsection*{The Time-Derivative of the $\Gamma$-Operator}
\begin{definition}
Given an interval $J\subset I$ and $u\in \F_{J}$ with $\Gamma_r(u_r)(x)\le C$ uniformly in $(r,x)\in J\times X$.
Then we define $\stackrel{\bullet}{\Gamma}_r(u_r)(x)$ as (one of the) weak subsequential limit(s) of
\begin{equation}\label{diff-gamm}\frac1{2\delta}\Big[\Gamma_{r+\delta}(u_r)-\Gamma_{r-\delta}(u_r)\Big](x)
\end{equation}
in $L^2(J\times X)$
for $\delta\to0$. That is, for a suitable 0-sequence $(\delta_n)_n$ and all $g\in L^2(J\times X)$
$$\frac1{2\delta_n}\int_J\int_X\Big[\Gamma_{r+\delta_n}(u_r)-\Gamma_{r-\delta_n}(u_r)\Big]\, g_r\,dm_r\,dr\to
\int_J\int_X\stackrel{\bullet}{\Gamma}_r(u_r)\, g_r\,dm_r\,dr$$
as $n\to\infty$.
\end{definition}
Actually,  thanks to Banach-Alaoglu theorem, such a weak limit always exists since \eqref{diff-gamm} -- due to the log-Lipschitz continuity of the distances -- defines a family of functions in $L^2(J\times X)$ with bounded norm. Thus in particular we will have
\begin{align}
&\liminf_{\delta\to0}\frac1{2\delta}\int_J\int_X\Big[\Gamma_{r+\delta}(u_r)-\Gamma_{r-\delta}(u_r)\Big]\, g_r\,dm_r\,dr\nonumber\\
&\le
\int_J\int_X\stackrel{\bullet}{\Gamma}_r(u_r)\, g_r\,dm_r\,dr\label{thats-it}\\
&\le
\limsup_{\delta\to0}\frac1{2\delta}\int_J\int_X\Big[\Gamma_{r+\delta}(u_r)-\Gamma_{r-\delta}(u_r)\Big]\, g_r\,dm_r\,dr.\nonumber
\end{align}
\begin{remark}
All the subsequent statements involving $\stackrel{\bullet}{\Gamma}_r(u_r)$ will be independent of the choice of the sequence $(\delta_n)_n$ and of the accumulation point in $L^2(J\times X)$. For instance, the precise meaning of Theorem \ref{Main} is that each of the properties {\bf(I)}, {\bf(II)} or {\bf(III)} will imply {\bf(IV)} for \emph{every} choice of the weak subsequential limit $\stackrel{\bullet}{\Gamma}_r(u_r)$. Conversely, if  {\bf(IV)} is satisfied for \emph{some} choice of the weak subsequential limit $\stackrel{\bullet}{\Gamma}_r(u_r)$ then it implies properties {\bf(I)}, {\bf(II)} and {\bf(III)}.
Indeed, the only property of $\stackrel{\bullet}{\Gamma}_r(u_r)$ which enters the calculations is \eqref{thats-it}.
\end{remark}
Note that the log-Lipschitz continuity of the distances also immediately implies that 
\begin{equation}\Big|\stackrel{\bullet}{\Gamma}_r(u_r)\Big|\le 2L\cdot {\Gamma}_r(u_r).\end{equation}
\begin{lemma}\label{gamm-diff-lemm}  For every $u\in \F_{J}$ with $\sup_{r,x}\Gamma_r(u_r)(x)<\infty$ 
and every $g\in L^\infty(J\times X)$
$$\int_J\int_X\stackrel{\bullet}{\Gamma}_r(u_r)\, g_r\,dm_r\,dr=
\lim_{n\to\infty}
\frac1{\delta_n}\int_J\int_X\Big[\Gamma_{r+\delta_n}(u_r,u_{r+\delta_n})-\Gamma_{r}(u_r,u_{r+\delta_n})\Big]\, g_r\,dm_r\,dr.$$
In particular,
\begin{align*}
&\liminf_{\delta\searrow0}\frac1{\delta}\int_J\int_X\Big[\Gamma_{r+\delta}(u_{r+\delta},u_r)-\Gamma_{r}(u_{r+\delta},u_r)\Big]\, g_r\,dm_r\,dr\\
&\le
\int_J\int_X\stackrel{\bullet}{\Gamma}_r(u_r)\, g_r\,dm_r\,dr\\
&\le
\limsup_{\delta\searrow0}\frac1{\delta}\int_J\int_X\Big[\Gamma_{r+\delta}(u_{r+\delta},u_r)-\Gamma_{r}(u_{r+\delta},u_r)\Big]\, g_r\,dm_r\,dr.
\end{align*}
\end{lemma}

\begin{proof}
\begin{align*}
\int_J\int_X\stackrel{\bullet}{\Gamma}_r(u_r)\, g_r\,dm_r\,dr
=&\lim_{n\to\infty}\Big(
\frac1{2\delta_n}\int_J\int_X\Big[\Gamma_{r+\delta_n}(u_r)-\Gamma_{r}(u_r)\Big]\, g_r\,dm_r\,dr
\\
&\quad+
\frac1{2\delta_n}\int_J\int_X\Big[\Gamma_{r}(u_r)-\Gamma_{r-\delta_n}(u_r)\Big]\, g_r\,dm_r\,dr
\Big)\\
=&\lim_{n\to\infty}\Big(
\frac1{2\delta_n}\int_J\int_X\Big[\Gamma_{r+\delta_n}(u_r)-\Gamma_{r}(u_r)\Big]\, g_r\,dm_r\,dr
\\
&\quad+
\frac1{2\delta_n}\int_J\int_X\Big[\Gamma_{r+\delta_n}(u_{r+\delta_n})-\Gamma_{r}(u_{r+\delta_n})\Big]\, g_r\,dm_r\,dr
\Big)\\
=&\lim_{n\to\infty}\Big(
\frac1{\delta_n}\int_J\int_X\Big[\Gamma_{r+\delta_n}(u_r,u_{r+\delta_n})-\Gamma_{r}(u_r,u_{r+\delta_n})\Big]\, g_r\,dm_r\,dr
\\
&\quad+
\frac1{2\delta_n}\int_J\int_X\Big[\Gamma_{r+\delta_n}(u_{r+\delta_n}-u_r)-\Gamma_{r}(u_{r+\delta_n}-u_r)\Big]\, g_r\,dm_r\,dr
\Big)\\
=&\lim_{n\to\infty}
\frac1{\delta_n}\int_J\int_X\Big[\Gamma_{r+\delta_n}(u_r,u_{r+\delta_n})-\Gamma_{r}(u_r,u_{r+\delta_n})\Big]\, g_r\,dm_r\,dr.
\end{align*}
Here for the second equality we used index shift and Lusin's theorem (to replace $g_{r+\delta_n}dm_{r+\delta_n}$ again by $g_rdm_r$). The last equality follows from 
the log-Lipschitz continuity of $r\mapsto d_r$ which allows to estimate
\begin{align*}
\frac1{\delta}\Big|\int_J\int_X&\Big[\Gamma_{r+\delta}(u_{r+\delta}-u_r)-\Gamma_{r}(u_{r+\delta}-u_r)\Big]\, g_r\,dm_r\,dr\Big|\\
&\le
2L\cdot \int_J\int_X\Gamma_r(u_{r+\delta}-u_r)\, g_r\,dm_r\,dr\\
&\le C'\cdot \int_J\E_r(u_{r+\delta}-u_r)dr
\to0
\end{align*}
as $\delta\to0$ since $r\mapsto u_r$, as a map from $J$ to $\F$, is `nearly continuous' 
(Lusin's theorem).
\end{proof}

\subsubsection*{The Distributional $\Gamma_2$-Operator}

\begin{definition}
For $r\in(0,T)$ and $u\in\Dom(\Delta_r)$ with $|\nabla_r u|\in L^\infty$ we define the distribution valued $\Gamma_2$-operator as a continuous linear operator
$${\bf \Gamma}_{2,r}(u): \F\cap L^\infty\to\R$$
by
\begin{equation}\label{Gamma-zwei}
{\bf \Gamma}_{2,r}(u)(g):= \int\Big[
 -\frac12\Gamma_r\big(\Gamma_{r}(u),g\big)
 +(\Delta_r u)^2g
 +\Gamma_r(u,g)\Delta_ru\Big]dm_r.
 \end{equation}
\end{definition}
Note that 
\begin{eqnarray*}\Big|{\bf \Gamma}_{2,r}(u)(g)\Big|&\le& 2\|\nabla_ru\|_{\infty}\cdot  \|\nabla_r^2u\|_{2} \cdot \|\nabla_rg\|_{2}+  \|g\|_{\infty}\cdot  \|\Delta_ru\|_{2}^2
+
\|\nabla_ru\|_{\infty}\cdot \|\nabla_rg\|_{2}\cdot  \|\Delta_ru\|_{2}\\
& 
\le&  \|g\|_{\infty}\cdot  \|\Delta_ru\|_{2}^2
+C\cdot
\|\nabla_ru\|_{\infty}\cdot \|\nabla_rg\|_{2}\cdot ( \|\Delta_ru\|_{2}+\|u\|_{2})
\end{eqnarray*}
thanks to the fact that $\|\nabla_r^2u\|_{2}^2 \le (1+K_-)\cdot (\|\Delta_r u\|_2^2+\|u\|_2^2)$, cf. \eqref{dom=dom}.

Also note that the assumptions on $u$ will be preserved under the heat flow (at least for a.e.\ $r$) and the assumptions on $g$ are preserved under the adjoint heat flow.
If $u$ is sufficiently regular (i.e. $\Delta u\in\Dom(\E_r)$ and $|\nabla_r u|^2\in\Dom(\Delta_r)$) then obviously
$${\bf \Gamma}_{2,r}(u)(g)=\int \Gamma_{2,r}(u)\cdot g\, dm_r$$
for all $g$ under consideration where as usual $\Gamma_{2,r}(u)=\frac12\Delta_r |\nabla_r u|^2-\Gamma_r(u,\Delta_r u)$.

On the other hand, if $g\in\Dom(\Delta_r)$
then  in  \eqref{Gamma-zwei}  we  may replace the term $-\Gamma_r\big(\Gamma_{r}(u),g\big)$ by  $\Gamma_{r}(u)\Delta_r g$. 
\medskip

\subsubsection*{The Bochner Inequality
}

\begin{definition}\label{new bochner}
(i) We say that $(X,d_t,m_t)_{t\in I}$ satisfies the dynamic Bochner inequality  with parameter $N\in(0,\infty]$  if for all $0<s<t<T$ and for all $u_s,g_t\in\F$ with $g_t\ge0$,  $g_t\in L^\infty$, $u_s\in \Lip(X)$
and for a.e.\ $r\in(s,t)$ 
 \begin{equation}\label{weak bochner}
 {\bf \Gamma}_{2,r}(u_r)(g_r)
 \geq\frac12\int\stackrel{\bullet}{\Gamma}_r(u_r)g_rdm_r+\frac1N\Big(\int\Delta_r u_rg_rdm_r\Big)^2
\end{equation}
where $u_r=P_{r,s}u_s$ and $g_r=P^*_{t,r}g_t$, cf. \eqref{est-IV-N}.

(ii) We say that  $(X,d_t,m_t)_{t\in I}$ satisfies property {\bf (IV$_N$)} if it satisfies the dynamic Bochner inequality with parameter $N$  
as above and in addition the regularity assumption \eqref{reg-boch} is satisfied, i.e.\
 $u_r\in\Lip(X)$ for all $r\in  (s,t)$ with $\sup_{r,x}\lip_ru_r(x)<\infty$.
\end{definition}

Note that in the case $N=\infty$ inequality \eqref{weak bochner} simply states that
$${\bf \Gamma}_{2,r}(u_r)\ge\frac12\stackrel\bullet\Gamma_r(u_r) m_r$$
as inequality between distributions,
tested against nonnegative functions $g_r$ as above.

\bigskip

\subsection{From  Bochner Inequality to Gradient Estimates}
\begin{theorem}[``${\bf (IV_N)} \Rightarrow {\bf (III_N)}$'']
 Suppose that the mm-space $(X,d_t,m_t)_{t\in I}$ satisfies the dynamic Bochner inequality \eqref{weak bochner} and the regularity assumption from Definition \ref{new bochner} (ii).
Then for a.e.\ $x\in X$
 \begin{align}\label{Nbakry}
 \Gamma_t(P_{t,s}u)(x)- P_{t,s}\Gamma_s(u)(x)\le 
   -\frac2N \int_s^t \big[P_{t,r}\Delta_ru_r(x)\big]^2dr.
     \end{align}  
\end{theorem}
\begin{proof}
Given $s,t\in (0,T)$ as well as $u\in \Lip(X)$ and  $g\in\F\cap L^\infty$ with $g\ge0$, put
$u_r=P_{r,s}u$, $g_r=P^*_{t,r}g$ for $r\in [s,t]$ and consider the function 
 \begin{equation*}
  h_r:=\int g_r\Gamma_r(u_r)dm_r=\int \Gamma_r(u_r)d\mu_r
 \end{equation*}
 with $\mu_r:=g_r\,m_r$.

(a) Choose $s\le\sigma<\tau\le t$ such that
  \begin{equation}\label{ini-assu}
 h_\tau\le\liminf_{\delta\searrow 0}\frac1\delta\int_{\tau-\delta}^\tau h_rdr
 \quad\mbox{and}\quad
 h_\sigma\ge\limsup_{\delta\searrow 0}\frac1\delta\int^{\sigma+\delta}_\sigma h_rdr.
 \end{equation}
Note that by Lebesgue's density theorem, the latter is true at least for a.e.\  $\sigma\ge s$ and for a.e.\ $\tau\le t$.
 (Moreover, at the end of this proof (as part (b)) we will present an argument which allows to conclude that \eqref{ini-assu} holds for $\sigma=s, \tau=t$.)
 Then 
 \begin{align*}
 h_\tau-h_\sigma\le&
 \liminf_{\delta\searrow 0}\frac1\delta\int_\sigma^{\tau-\delta} \big[h_{r+\delta}-h_r\big]dr\\
 \le&
  \limsup_{\delta\searrow 0}\frac1\delta\int_\sigma^{\tau-\delta}\int_X\Gamma_{r+\delta}(u_{r+\delta})d(\mu_{r+\delta}-\mu_r)\,dr\\
  &+  \liminf_{\delta\searrow 0}\frac1\delta\int_\sigma^{\tau-\delta}\int_Xg_r\Big[ \Gamma_{r+\delta}(u_{r+\delta},u_r)-\Gamma_r(u_{r+\delta},u_r)\Big]dm_r\,dr\\
   &+  \limsup_{\delta\searrow 0}\frac1\delta\int_\sigma^{\tau-\delta}\int_Xg_r\Big[
   \Gamma_{r+\delta}(u_{r+\delta},u_{r+\delta}-u_r)+\Gamma_r(u_{r+\delta}-u_r,u_r)\Big]dm_r\,dr\\
   =:& (I) + (II) + (III')+ (III'').
 \end{align*}
 Each of the four terms will be considered separately. Since $r\mapsto \mu_r$ is a solution to the dual heat equation, we obtain
  \begin{align*}
  (I)=& \limsup_{\delta\searrow 0}\frac1\delta\int_\sigma^{\tau-\delta}\int_X\Gamma_{r+\delta} (u_{r+\delta})\cdot\Big(- \int_r^{r+\delta}\Delta_q g_q\,dm_q\,dq\Big)dr\\
  =& -\liminf_{\delta\searrow 0}\int_{\sigma+\delta}^{\tau}\int_X\Gamma_{r} (u_{r})\Big( \frac{1}\delta\int^r_{r-\delta}\Delta_q g_q e^{-f_q}\,dq\Big)dm_\diamond\,dr\\
  =&-\int_\sigma^\tau\int_X \Gamma_r(u_r)\cdot \Delta_rg_r\, dm_r\,dr
   \end{align*}
   due 
   Lebesgue's density theorem applied to $r\mapsto  \Delta_rg_re^{-f_r}$.
   Note that the latter function is in $L^2$ (Theorem \ref{energy-est}) and the function $r\mapsto  \Gamma_r(u_r)$ is in $L^\infty$ thanks to Definition \ref{new bochner} (ii).
   
   The second term can easily estimated in terms  $\stackrel{\bullet}\Gamma_r$ according to Lemma \ref{gamm-diff-lemm}:
  \begin{align*}
  (II)=&
    \liminf_{\delta\searrow 0}\frac1\delta\int_\sigma^{\tau-\delta}\int_Xg_r\Big[ \Gamma_{r+\delta}(u_{r+\delta},u_r)-\Gamma_r(u_{r+\delta},u_r)\Big]dm_r\,dr\\
 \le& \int_\sigma^\tau\int_Xg_r\,\stackrel{\bullet}\Gamma_r(u_r)dm_rdr.
    \end{align*}

  The  term $ (III')$ is transformed as follows
   \begin{align*}
  (III')=&- \liminf_{\delta\searrow 0}\frac1\delta\int_\sigma^{\tau-\delta}\int_X\Big(
  \Gamma_{r+\delta}(g_r,u_{r+\delta})+g_r\,\Delta_{r+\delta}u_{r+\delta}\Big)\cdot\Big( \int_r^{r+\delta}\Delta_qu_q\,dq\Big)dm_r\,dr\\
  =&- \liminf_{\delta\searrow 0}\int^\tau_{\sigma+\delta}\int_X\Big(
  \Gamma_{r}(g_{r-\delta},u_{r})+g_{r-\delta}\,\Delta_{r}u_{r}\Big)\cdot\Big(\frac1\delta \int^r_{r-\delta}\Delta_qu_q\,dq\Big)dm_r\,dr\\
  =&-\int_\sigma^{\tau}\int_X\Big(
  \Gamma_{r}(g_r,u_{r})+g_r\,\Delta_{r}u_{r}\Big)\cdot \Delta_ru_r\,dm_r\,dr.
           \end{align*}
           Here again we used Lebesgue's density theorem (applied to $r\mapsto  \Delta_ru_r$) and the `nearly continuity' of $r\mapsto  g_r$ as map from $(s,t)$ into $L^2(X,m)$ and as map into $\F$  (Lusin's theorem). 
           Moreover, we used  the boundedness (uniformly in $r$ and $x$) of $g_r$ and of $\nabla_ru_r$ as well as the square integrability of $\Delta_ru_r$.

Similarly, the  term $ (III'')$ will be transformed:
   \begin{align*}
  (III'')=&- \liminf_{\delta\searrow 0}\frac1\delta\int_\sigma^{\tau-\delta}\int_X\Big(
  \Gamma_{r}(g_r,u_{r})+g_r\,\Delta_{r}u_{r}\Big)\cdot\Big( \int_r^{r+\delta}\Delta_qu_q\,dq\Big)dm_r\,dr\\
  =&-\int_\sigma^{\tau}\int_X\Big(
  \Gamma_{r}(g_r,u_{r})+g_r\,\Delta_{r}u_{r}\Big)\cdot\Big( \Delta_ru_r\Big)dm_r\,dr.
           \end{align*}
Summarizing and then using
  \eqref{weak bochner}, 
  we therefore obtain
  \begin{align*}
 h_\tau-h_\sigma=&(I) + (II) + (III')+ (III'')\\
\le& \int_\sigma^\tau\int_X\Big[
- \Gamma_r(u_r)\cdot \Delta_rg_r
 +g_r\,\stackrel{\bullet}\Gamma_r(u_r)
 -2\big(
  \Gamma_{r}(g_r,u_{r})+g_r\,\Delta_{r}u_{r}\big)\,\Delta_ru_r
 \Big] dm_r\,dr\\
 \le&-\frac2N\int_\sigma^\tau \Big[ \int_X \Delta_ru_r\, g_r\,dm_r\Big]^2dr=
 -\frac2N\int_\sigma^\tau  \Big[ \int_X P_{\tau,r}\Delta_ru_r\, g\,dm_\tau\Big]^2dr.
   \end{align*}
   Thus
  \begin{equation}
  \label{int-grad-est}\int_X\Gamma_\tau(P_{\tau,\sigma}u)g\, dm_\tau
 -\int_X P_{\tau,\sigma}\Gamma_\sigma(u)\, g\,dm_\tau
  \le  -\frac2N\int_\sigma^\tau  \Big[ \int_X P_{\tau,r}\Delta_ru_r\, g\,dm_\tau\Big]^2dr.
     \end{equation}

\medskip

 (b) 
 Recall that, given $u$ and $g$, this holds for a.e.\ $\tau$ and a.e.\ $\sigma$.
 Now let us forget for the moment the term with $N$.
Choosing $g$'s from a dense countable set one may achieve that  the exceptional sets for $\sigma$ and $\tau$ in \eqref{int-grad-est}  do not depend on $g$.
 Next we may assume that    $\sigma,\tau\in [s,t]$ with $\sigma<\tau$ is chosen  such that \eqref{int-grad-est} with $N=\infty$ simultaneously holds for all $u$ from a dense countable set ${\mathcal C}_1$ in $\Lip(X)$. 
    Approximating arbitrary $u \in\Lip(X)$ by $u_n\in{\mathcal C}_1$ yields
     \begin{equation*}
 \int_X\Gamma_\tau(P_{\tau,\sigma}u)g\, dm_\tau
 -\int_X P_{\tau,\sigma}\Gamma_\sigma(u)\, g\,dm_\tau\le
  \liminf_n\int_X\Gamma_\tau(P_{\tau,\sigma}u_n)g\, dm_\tau
 -\lim_n\int_X P_{\tau,\sigma}\Gamma_\sigma(u_n)\, g\,dm_\tau
  \le 0.
     \end{equation*}
    due to lower semicontinuity of the weighted energy on $L^2$.
   In other words, we have derived the gradient estimate {\bf(III)} for almost all times $\sigma$ and $\tau$. Thanks to Theorem \ref{IIeqIII}
   this implies the transport estimate  {\bf(II)} for these time instances.
 But both sides of the transport estimate are continuous in time (thanks to the continuity of $r\mapsto W_r$ and the continuity of the dual heat flow). 
 This implies that the transport estimate holds for all $\sigma,\tau\in[s,t]$ with $\sigma<\tau$. In particular, it holds for $\sigma=s$ and $\tau=t$. Again 
 by Theorem \ref{IIeqIII} it yields the gradient estimate for given $s$ and $t$ and thus our initial assumption \eqref{ini-assu} is satisfied for the 
 choice $\sigma=s$ and $\tau=t$.
 
\medskip

 (c)  Taking this into account, we may conclude that \eqref{int-grad-est} (for given $N$) holds with the choice $\sigma=s$ and $\tau=t$.
 Finally, choosing sequences of $g$'s which approximate the Dirac distribution at a given $x\in X$ then implies that  for all $u\in \Lip(X)$ 
    \begin{equation}\label{ptw-grad-est}
    \Gamma_t(P_{t,s}u)(x)- P_{t,s}\Gamma_s(u)(x)\le 
   -\frac2N \int_s^t \big[P_{t,r}\Delta_ru_r(x)\big]^2dr
     \end{equation}  
     for a.e.\ $x\in X$.
      This proves the claim for bounded Lipschitz functions. The extension to $u\in\Dom(\E)$ follows as in \cite{agsmet}.
\end{proof}

\subsection{From Gradient Estimates to Bochner Inequality}

In the previous chapter and the previous sections of this chapter, we have proven the
implications
{\bf (III$_N$)} $\Rightarrow$ {\bf (II$_N$)} and {\bf (IV$_N$)}  $\Rightarrow$ {\bf (III$_N$)}. Taking the subsequent section into account, where we show
{\bf (II$_N$)} $\Rightarrow$ {\bf (IV$_N$)}, we already have proven that
 {\bf (III$_N$)}  $\Rightarrow$ {\bf (IV$_N$)}. In the sequel, we will present another, more direct proof for this implication.

 \begin{theorem}[``${\bf (III_N)} \Rightarrow {\bf (IV_N)}$'']\label{Nbakrytobochner}
 Suppose that the mm-space $(X,d_t,m_t)_{t\in I}$ satisfies the 
 gradient estimate \eqref{Nbakry}. Then the
 dynamic Bochner inequality \eqref{weak bochner} holds true as well as the regularity assumption from Definition \ref{new bochner} (ii).
\end{theorem}

 \begin{proof}
 Assume that the gradient estimate  {\bf (III$_N$)} holds true. It immediately implies the 
 regularity assumption  \eqref{reg-boch}.
   To derive the dynamic Bochner inequality, let 
  $s,t\in (0,T)$ as well as $u\in \Lip(X)$ and  $g\in\F\cap L^\infty$ with $g\ge0$ be given. Put
$u_r=P_{r,s}u$, $g_r=P^*_{t,r}g$ for $r\in [s,t]$ and as before consider the function 
 \begin{equation*}
  h_r:=\int g_r\Gamma_r(u_r)dm_r.
 \end{equation*}
Then   {\bf (III$_N$)} implies that for all $s<\sigma<\tau<t$
\begin{align*}
 h_\tau-h_\sigma\le&
 \liminf_{\delta\searrow 0}\frac1\delta\int_\sigma^{\tau-\delta} \big[h_{r+\delta}-h_r\big]dr\\
 =&
\liminf_{\delta\searrow0}\frac1\delta\int_\sigma^{\tau-\delta}\int_X\Big[\Gamma_{r+\delta}(u_{r+\delta})-P_{r+\delta,r}\Gamma_r(u_r)\Big]g_{r+\delta}dm_{r+\delta}\,dr\\
\le&
 -\frac2N\limsup_{\delta\searrow0}\int_\sigma^{\tau-\delta}\int_X\frac1\delta\int_r^{r+\delta}\big(P_{r+\delta,q}\Delta_qu_q\big)^2dq\, g_{r+\delta}dm_{r+\delta}\,dr\\
 \le&
 -\frac2N\int_\sigma^{\tau}\liminf_{\delta\searrow0}\Big(\int_X\frac1\delta\int_r^{r+\delta}P_{r+\delta,q}\Delta_qu_q\,dq\, g_{r+\delta}dm_{r+\delta}\Big)^2\\
  =&
 -\frac2N\int_\sigma^{\tau}\liminf_{\delta\searrow0}\Big(\frac1\delta\int_r^{r+\delta}\int_X\Delta_qu_q\,g_{q}dm_{q}\,dq \Big)^2dr\\
  =&
 -\frac2N\int_\sigma^{\tau}\Big(\int_X\Delta_ru_r\,g_{r}dm_{r} \Big)^2dr\\
 \end{align*}
according to Lebesgue's density theorem. On the other hand, similarly to the argumentation in the previous section, we have
 \begin{align*}
  h_\tau-h_\sigma\ge&
 \limsup_{\delta\searrow 0}\frac1\delta\int^\tau_{\sigma-\delta} \big[h_{r+\delta}-h_r\big]dr\\
 \ge&
  \liminf_{\delta\searrow 0}\frac1\delta\int^\tau_{\sigma-\delta}\int_X\Gamma_{r+\delta}(u_{r+\delta})d(\mu_{r+\delta}-\mu_r)\,dr\\
  &+  \limsup_{\delta\searrow 0}\frac1\delta\int^\tau_{\sigma-\delta}\int_Xg_r\Big[ \Gamma_{r+\delta}(u_{r+\delta},u_r)-\Gamma_r(u_{r+\delta},u_r)\Big]dm_r\,dr\\
   &+  \liminf_{\delta\searrow 0}\frac1\delta\int^\tau_{\sigma-\delta}\int_Xg_r\Big[
   \Gamma_{r+\delta}(u_{r+\delta},u_{r+\delta}-u_r)+\Gamma_r(u_{r+\delta}-u_r,u_r)\Big]dm_r\,dr\\
   =:& (I) + (II) + (III')+ (III'').
 \end{align*}
Each of the four terms can be treated as before  which then yields
  \begin{align*}
&h_{\tau}-h_\sigma\ge(I) + (II) + (III')+ (III'')\\
\ge& \int_\sigma^\tau\int_X\Big[
- \Gamma_r(u_r)\cdot \Delta_rg_r
 +g_r\,\stackrel{\bullet}\Gamma_r(u_r)
 -2\big(
  \Gamma_{r}(g_r,u_{r})+g_r\,\Delta_{r}u_{r}\big)\,\Delta_ru_r
 \Big] dm_r\,dr\\
 =&\int_\sigma^\tau\Big[-2{\bf\Gamma}_{2,r}(u_r)(g_r)+\int \stackrel{\bullet}\Gamma_r(u_r)\,g_r\,m_r
  \Big] \,dr.
   \end{align*}
   Combining this with the previous upper estimate and varying $\sigma$ and $\tau$, we thus  have proven the dynamic Bochner inequality 
    \begin{align*}
2{\bf\Gamma}_{2,r}(u_r)(g_r)\ge\int \stackrel{\bullet}\Gamma_r(u_r)\,g_r\,m_r+\frac2N\Big(\int_X\Delta_ru_r\,g_{r}dm_{r} \Big)^2
   \end{align*}
   for a.e.\ $r\in (s,t)$.
   \end{proof}

\subsection{From Transport Estimates to Bochner Inequality}

\bigskip

 \begin{theorem}[``${\bf (II_N)} \Rightarrow {\bf (IV_N)}$'']
 Suppose that the mm-space $(X,d_t,m_t)_{t\in I}$ satisfies the 
 transport estimate \eqref{est-II-N}=\eqref{II.N}. Then the dynamic 
 Bochner inequality  \eqref{est-III-N}=\eqref{weak bochner} with parameter $N$ holds true  as well as the regularity assumption  \eqref{reg-boch}.
\end{theorem}
 
 \begin{proof}[Proof of the regularity assumption]
 Thanks to Theorem \ref{IIeqIII}, we already know that the transport estimate {\bf (II$_N$)}
implies 
the gradient estimate {\bf (III$_N$)}  in the case $N=\infty$. This proves the requested regularity.
 \end{proof}

\begin{proof}[Proof of the dynamic Bochner inequality]
We follow the argumentation from  \cite{bggk2015} with significant modifications due to time-dependence of functions, gradients, and operators and mainly because of lack of regularity.

Let  $0<s<t<T$ and  $g_t\in\F\cap L^\infty$ with $g_t\ge0$,  $g_t\not\equiv0$ as well as $u_s\in \Lip(X)$ be given and fixed for the sequel. Without restriction $\int g_tdm_t=1$. For $\tau\in (s,t)$, put  $u_\tau=P_{\tau,s}u_s$ and $g_\tau=P^*_{t,\tau}g_t$. Note that -- thanks to the parabolic Harnack inequality -- $g$ is uniformly bounded from above and bounded from below, away from 0, on $(s',t')\times X$ for each $s<s'<t'<t$.
In the beginning, let us also assume that $||u_s||_\infty\leq1/4$.

\medskip

For each $\tau\in(s,t)$, define a Dirichlet form $\mathcal E_\tau^g$ on $L^2(X,g_\tau m_\tau)$ with domain $\Dom(\mathcal E^g_\tau):=\Dom(\mathcal E)$ by
\begin{equation*}
 \mathcal E_\tau^g(u):=\int\Gamma_\tau(u)g_\tau dm_\tau \quad\text{ for }u\in \Dom(\E).
\end{equation*}
Associated with the closed bilinear form  $(\E_\tau^g,\Dom(\E_\tau^g))$  on $L^2(X,g_\tau m_\tau)$, there is the
self-adjoint operator $\Delta_\tau^g$ and the semigroup $(H_a^{\tau,g})_{a\geq0}$, i.e.
$u_a=H_a^{\tau,g}u$ solves
\begin{equation*}
\partial_a u_a=\Delta_\tau^g u_a \text{ on }(0,\infty)\times X,\qquad u_0=u
\end{equation*}
where $\Delta_\tau^gu=\Delta_\tau u+\Gamma_\tau(\log g_\tau,u)$.
For fixed $\sigma\in(s,\tau)$, we define the path $(g_\tau^{\sigma,a})_{a\geq 0}$ to be
\begin{equation}
g_\tau^{\sigma,a}:=g_\tau(1+u_\sigma-H_a^{\tau,g}u_\sigma).
\end{equation}
Note that these are probability densities w.r.t.\ $m_\tau$. Indeed, for all $a>0$ and all $s<\sigma<\tau<t$ 
$$\int g_\tau^{\sigma,a}dm_\tau= 1+\int u_\sigma (1-H_a^{\tau,g}1)\, g_\tau m_\tau=1$$
thanks to conservativeness and  symmetry of $H_a^{\tau,g}$ w.r.t.\ the measure $g_\tau m_\tau$. Moreover,
$g_\tau^{\sigma,a}\ge0$ for all $a$, $\sigma$ and $\tau$ since the uniform bound $||u_s||_\infty\leq1/4$ is preserved under the evolution of the time-dependent heat flow, thus  $||u_\sigma||_\infty\le ||P_{\sigma,s}u_s||_\infty
\leq1/4$, as well as under the heat flow in the static mm-space at fixed time $\tau$, thus
 $||H_a^{\tau,g}u_\sigma||_\infty\le ||u_\sigma||_\infty
\leq1/4$.

\medskip

Now let us assume that the transport estimate {\bf (II$_N$)} holds true and apply it to the probability measures
$g_\tau m_\tau$ and $g^{\sigma,a}_\tau m_\tau$.
Then for all $s<\sigma<\tau<t$ and all  $a>0$

\begin{eqnarray*}
 W_\sigma^2(\hat P_{\tau,s}(g_\tau m_\tau),\hat P_{\tau,\sigma}(g^{\sigma,a}_\tau m_\tau))&\leq& W_\tau^2(g_\tau  m_\tau,g^{\sigma,a}_\tau m_\tau)\\
 &&-\frac{2}N\int_\sigma^\tau[S_r(\hat P_{\tau,r}(g_\tau m_\tau))-S_r(\hat P_{\tau,r}(g^{\sigma,a}_\tau m_\tau))]^2dr.
\end{eqnarray*}
Dividing by $2a^2$ and passing to the limit $a\searrow0$, the subsequent Lemmata \ref{bochner:west1},
\ref{bochner:west2} and \ref{bochner:entest} allow to estimate term by term. We thus obtain
 \begin{align*}
  &-\frac{1}2 \int P_{\tau,\sigma}(\Gamma_\sigma(u_\sigma)) g_\tau dm_\tau+\int \Gamma_\tau(P_{\tau,\sigma}u_\sigma,u_\sigma)g_\tau dm_\tau\\
  &\leq\frac{1}{2(1-2||u_\sigma||_\infty)}\int\Gamma_\tau(u_\sigma)g_\tau dm_\tau-\frac{1}{N}\int_\sigma^\tau\left[\int \Gamma_\tau\big(P_{\tau,r}(\log P^*_{\tau,r} g_\tau), u_\sigma\big) \, g_\tau dm_\tau\right]^2dr.
 \end{align*}
Replacing $u_s$ by $\eta\, u_s$ for $\eta\in\R_+$ sufficiently small, we can get rid of the constraint $||u_s||_{\infty}\le 1/4$.
 Then Lemma \ref{bochner:west1}, Lemma \ref{bochner:west2} and Lemma \ref{bochner:entest}
 applied to $\eta u_s$ instead of $u_s$ gives us
 \begin{align*}
  &-\frac{\eta^2}2 \int P_{\tau,\sigma}(\Gamma_\sigma(u_\sigma)) g_\tau dm_\tau+\eta^2\int \Gamma_\tau(P_{\tau,\sigma}u_\sigma,u_\sigma)g_\tau dm_\tau\\
  &\leq\frac{\eta^2}{2(1-2\eta||u_\sigma||_\infty)}\int\Gamma_\tau(u_\sigma)g_\tau dm_\tau-\frac{\eta^2}{N}\int_\sigma^\tau\left[\int \Gamma_\tau\big(P_{\tau,r}(\log P^*_{\tau,r} g_\tau), u_\sigma\big) \, g_\tau dm_\tau\right]^2dr.
\end{align*}
Dividing by $\eta^2$ and letting $\eta\to0$ this inequality becomes
  \begin{align*}
  &-\frac{1}2 \int P_{\tau,\sigma}(\Gamma_\sigma(u_\sigma)) g_\tau dm_\tau+\int \Gamma_\tau(P_{\tau,\sigma}u_\sigma,u_\sigma)g_\tau dm_\tau\\
  &\leq\frac12\int\Gamma_\tau(u_\sigma)g_\tau dm_\tau-\frac{1}{N}\int_\sigma^\tau\left[\int \Gamma_\tau\Big(P_{\tau,r}(\log P^*_{\tau,r} g_\tau) , u_\sigma\big) \, g_\tau dm_\tau\right]^2dr.
 \end{align*}
 This can be reformulated into
 \begin{equation}
  \begin{aligned}\label{whatwehave}
  &\frac12\int\Gamma_\tau(u_\tau)g_\tau dm_\tau-\frac{1}2 \int \Gamma_\sigma(u_\sigma) g_\sigma dm_\sigma\\
  &-\frac12\int\Gamma_\tau(u_\sigma)g_\tau dm_\tau-\frac12\int\Gamma_\tau(u_\tau)g_\tau dm_\tau
  +\int \Gamma_\tau(u_\tau,u_\sigma)g_\tau dm_\tau\\
  &\leq-\frac{1}{N}\int_\sigma^\tau\left[\int \Gamma_\tau\Big(P_{\tau,r}(\log P^*_{\tau,r} g_\tau) , u_\sigma\big) \, g_\tau dm_\tau\right]^2dr.
 \end{aligned}
 \end{equation}
Now let us try to follow the argumentation from the proof of Theorem \ref{Nbakrytobochner} and consider again the function
  \begin{equation*}
  h_r:=\int g_r\Gamma_r(u_r)dm_r
 \end{equation*}
 for $r\in (s,t)$.
 Recall that we already know from Theorem \ref{IIeqIII} that the transport estimate {\bf (II$_N$)} implies the gradient estimate {\bf (III)} (`without $N$').
 Thus for all $s<\sigma<\tau< t$
 $$\limsup_{\delta\searrow0}\frac1\delta\int_{\sigma-\delta}^\tau\big(h_{r+\delta}-h_r\big)dr
 \le h_\tau-h_\sigma\le 
 \liminf_{\delta\searrow0}\frac1\delta\int_{\sigma}^{\tau-\delta}\big(h_{r+\delta}-h_r\big)dr
 $$
 Arguing as in the proof of Theorem \ref{Nbakrytobochner} we get 
 \begin{align*}
   h_\tau-h_\sigma
 \geq\int_\sigma^\tau\Big[-2{\bf\Gamma}_{2,r}(u_r)(g_r)+\int \stackrel{\bullet}\Gamma_r(u_r)\,g_r\,m_r
  \Big] \,dr.
  \end{align*}
  On the other hand, applying the previous estimate \eqref{whatwehave} (with $r+\delta$, $r$ and $q$ in the place of $\tau$, $\sigma$ and $r$) we obtain
  \begin{align*}
   h_\tau-h_\sigma
  \leq\liminf_{\delta\searrow0}\frac1\delta\int_\sigma^{\tau-\sigma}\Big[&-\frac{2}{N}\int_r^{r+\delta}\left[\int \Gamma_{r+\delta}\Big(P_{{r+\delta},q}(\log P^*_{{r+\delta},q} g_{r+\delta}) , u_r\big) \, g_{r+\delta} dm_{r+\delta}\right]^2dq\\
  &+ \int\Gamma_{r+\delta}(u_{r+\delta}-u_r)g_{r+\delta} dm_{r+\delta}\Big]\,dr.\\
  \end{align*}
  We estimate the term with the square from below using Young's inequality
  \begin{align*}
 &\left[\int \Gamma_{r+\delta}\Big(P_{{r+\delta},q}(\log P^*_{{r+\delta},q} g_{r+\delta}) , u_r\Big) \, g_{r+\delta} dm_{r+\delta}\right]^2\\
 & \geq\frac1{1+\epsilon} \left[\int \Gamma_{r}\Big(P_{{r},q}(\log g_q) , u_r\Big) \, g_r dm_{r}\right]^2\\
 & -\frac1{\epsilon}\left[\int \Gamma_{r+\delta}\Big(P_{{r+\delta},q}(\log P^*_{{r+\delta},q} g_{r+\delta}) , u_r\Big) \, g_r dm_{r+\delta}-\int \Gamma_{r}\Big(P_{{r},q}(\log g_q) , u_r\Big) \, g_r dm_{r}\right]^2,
  \end{align*}
  where $\epsilon>0$ is arbitrary.
 Further estimating and using the log-Lipschitz continuity $r\mapsto \Gamma_r$ yields
  \begin{align*}
 & \left[\int \Gamma_{r+\delta}\Big(P_{{r+\delta},q}(\log P^*_{{r+\delta},q} g_{r+\delta}) , u_r\Big) \, g_{r+\delta} dm_{r+\delta}-\int \Gamma_{r}\Big(P_{{r},q}(\log g_q) , u_r\Big) \, g_r dm_{r}\right]^2\\
  &\leq 2\left[\int \Gamma_{r+\delta}\Big(P_{{r+\delta},q}(\log g_{q}) , u_r\Big) \, g_{r+\delta} dm_{r+\delta}-\int \Gamma_{r}\Big(P_{{r+\delta},q}(\log g_{q}) , u_r\Big) \, g_{r+\delta}dm_{r+\delta}\right]^2\\
  &+2\left[\int \Gamma_{r}\Big(P_{{r+\delta},q}(\log g_q) , u_r\Big) \, g_{r+\delta}dm_{r+\delta}-\int \Gamma_{r}\Big(P_{{r},q}(\log g_q) , u_r\Big) \, g_{r+\delta}dm_{r+\delta}\right]^2\\
  &+2\left[\int \Gamma_{r}\Big(P_{{r},q}(\log g_q) , u_r\Big) \, g_{r+\delta}dm_{r+\delta}
  -\int \Gamma_{r}\Big(P_{{r},q}(\log g_q) , u_r\Big) \, g_rdm_{r}\right]^2\\
  &\leq 16L^2\delta^2 \left[\int \Gamma_{r+\delta}\Big(P_{{r+\delta},q}(\log g_{q}) , u_r\Big) \, g_{r+\delta}dm_{r+\delta}+C\int \Gamma_{r+\delta}\Big(P_{{r+\delta},q}(\log g_{q}) -u_r\Big) \, g_{r+\delta}dm_{r+\delta}\right]^2\\
  &+2\left[\int \Gamma_{r}\Big(P_{{r+\delta},q}(\log g_q)-P_{{r},q}(\log g_q) , u_r\Big) \, g_{r+\delta}dm_{r+\delta}\right]^2\\
  &+2\left[\int \Gamma_{r}\Big(P_{{r},q}(\log g_q) , u_r\Big) \,d( g_{r+\delta}dm_{r+\delta}
  - g_rm_{r})\right]^2,\\
  \end{align*}
  which, after integration over $[r,r+\delta]$ and division by $\delta>0$, converges to 0 as $\delta$ goes to 0.
  Indeed, 
  \begin{align*}
\delta \int_r^{r+\delta}\Big| \int \Gamma_{r+\delta}\Big(P_{{r+\delta},q}(\log P^*_{{r+\delta},q} g_{r+\delta}) , u_r\Big) \, g_{r+\delta}dm_{r+\delta}\Big|^2dq\\
\leq C\delta \Big(\int_r^{r+\delta} \int \Gamma_{q}(\log g_{q}) \, g_{q}dm_{q}dr\Big)\mathcal E_r(u_r)\xrightarrow[\delta\to0]{}0,
  \end{align*}
  and Lemma \ref{P*0} and Lebesgue differentiation theorem
  \begin{align*}
  \frac1\delta\int_r^{r+\delta}\Big|\int \Gamma_{r}\Big(P_{{r+\delta},q}(\log g_q)-P_{{r},q}(\log g_q) , u_r\Big) \, g_{r+\delta}dm_{r+\delta}\Big|^2dq\xrightarrow[\delta\to0]{}0,
  \end{align*}
  while
  \begin{align*}
\frac1\delta \int_r^{r+\delta} \left[\int \Gamma_{r}\Big(P_{{r},q}(\log g_q) , u_r\Big) \,d( g_{r+\delta}dm_{r+\delta}
  - g_rm_{r})\right]^2dq\xrightarrow[\delta\to0]{}0.
  \end{align*}
  Thus, since $\epsilon$ is arbitrary, and from the Lebesgue differentiation theorem we get
  \begin{align*}
  \liminf_{\delta\to0}\frac1\delta \int_r^{r+\delta}\left[\int \Gamma_{r+\delta}\Big(P_{{r+\delta},q}(\log P^*_{{r+\delta},q} g_{r+\delta}) , u_r\Big) \, g_{r+\delta} dm_{r+\delta}\right]^2dr\\
  \geq\left[\int \Gamma_{r}\Big(\log g_q , u_r\Big) \, g_r dm_{r}\right]^2=
  \left[\int (\Delta_{r} u_r) g_r dm_{r}\right]^2.
  \end{align*}
  Finally, with Corollary \ref{trivial-lp}, the log-Lipschitz continuity of $r\mapsto \Gamma_r$, Lemma \ref{P*0}, and Lebesgue differentiation theorem applied to $r\mapsto\Delta_ru_r$, which is in $L^2((s,t),\mathcal H)$
thanks to Theorem \ref{energy-est},
  \begin{align*}
 \limsup_{\delta\to0} \frac1\delta& \int_\sigma^{\tau-\delta}\int\Gamma_{r+\delta}(u_{r+\delta}-u_r)g_{r+\delta} dm_{r+\delta}\,dr\\
  \leq \limsup_{\delta\to0}\frac1\delta& \int_\sigma^{\tau-\delta}||g_{r+\delta}||_\infty\int\Gamma_{r+\delta}(u_{r+\delta}-u_r,u_{r+\delta})dm_{r+\delta}\,dr\\
 \leq \limsup_{\delta\to0}\frac1\delta& \int_\sigma^{\tau-\delta}e^{L|r+\delta-t|}
  ||g_{t}||_\infty\Big(\int\Gamma_{r+\delta}(u_{r+\delta}-u_r,u_{r+\delta}) dm_{r+\delta}
  - \int\Gamma_{r+\delta}(u_{r+\delta}-u_r,u_{r}) dm_{r+\delta}\Big)\,dr\\
  =\limsup_{\delta\to0}\frac1\delta& \int_\sigma^{\tau-\delta}e^{L|r+\delta-t|}
  ||g_{t}||_\infty\Big(-\int\int_r^{r+\delta}\Delta_qu_qdq\Delta_{r+\delta}u_{r+\delta} dm_{r+\delta}
  - \int\Gamma_{r}(u_{r+\delta}-u_r,u_{r}) dm_{r}\Big)\,dr\\
  =\limsup_{\delta\to0}\Big( &\int_{\sigma+\delta}^{\tau}-e^{L|r-t|}
  ||g_{t}||_\infty\int\frac1\delta\int_{r-\delta}^{r}\Delta_qu_qdq\Delta_{r}u_{r} dm_{r}dr\\
  &+ \int_\sigma^{\tau-\delta}e^{L|r+\delta-t|}
  ||g_{t}||_\infty\int\frac1\delta\int_r^{r+\delta}\Delta_qu_qdq\Delta_ru_{r} dm_{r}\,dr\Big)\\
  =&\int_\sigma^{\tau}e^{L|r-t|}
  ||g_{t}||_\infty\Big(-\int(\Delta_{r}u_{r})^2 dm_{r}
  + \int(\Delta_ru_r)^2 dm_{r}\Big)=0.
  \end{align*}
  Combining the previous estimates we get
  
  \begin{align*}
  h_\tau-h_\sigma\leq  -\frac2N\int_\sigma^\tau &\Big( \int \Delta_r u_r\, g_rdm_r\Big)^2dr,
 \end{align*}  
 and then
 \begin{align*}
  -\frac2N\int_\sigma^\tau &\Big( \int \Delta_r u_r\, g_rdm_r\Big)^2dr\geq\int_\sigma^\tau\Big[-2{\bf\Gamma}_{2,r}(u_r)(g_r)+\int \stackrel{\bullet}\Gamma_r(u_r)\,g_r\,m_r
  \Big] \,dr,
 \end{align*}
 which proves the claim.
\end{proof}

\bigskip

\begin{lemma}\label{bochner:west1} 
For every $s<\sigma\leq\tau<t$,
\begin{equation*}
 \liminf_{a\to0} \frac{ W_\sigma^2(\hat P_{\tau,\sigma}(g_\tau^{\sigma,a}m_\tau),\hat P_{\tau,\sigma}(g_\tau m_\tau))}{2a^2}\geq-\int\frac12 P_{\tau,\sigma}(\Gamma_\sigma(u_\sigma)) g_\tau dm_\tau+\int \Gamma_\tau(u_\tau,u_\sigma)g_\tau dm_\tau.
\end{equation*}

\end{lemma}
\begin{proof}
We denote by $Q_a^\sigma$ the Hopf-Lax semigroup with respect to the metric $d_\sigma$.
Note that $aQ_a^\sigma(\phi)=Q_1^\sigma(a\phi)$, so the Kantorovich duality \eqref{Kantdual} can be written as
\begin{align*}
 \frac{W_\sigma^2(\nu_1,\nu_2)}{2a^2}=\frac1{a}\sup_\phi\left[\int Q_a^\sigma\phi d\nu_1-\int\phi d\nu_2\right].
\end{align*}
We deduce
\begin{align*}
&\frac{W_\sigma^2(\hat P_{\tau,\sigma}(g_\tau^{\sigma,a}m_\tau),\hat P_{\tau,\sigma}(g_\tau m_\tau))}{2a^2}\geq \int\frac{Q_a^\sigma u_\sigma P_{\tau,\sigma}^*(g_\tau^{\sigma,a})-u_\sigma P_{\tau,\sigma}^*g_\tau}{a}dm_s\\
&\geq\int\frac{Q_au_\sigma-u_\sigma}a  P^*_{\tau,\sigma}(g_\tau^{\sigma,a}-g_\tau)dm_\sigma+\int\frac{Q_au_\sigma-u_\sigma}a P^*_{\tau,\sigma}g_\tau dm_\sigma+\int u_\sigma\frac{P^*_{\tau,\sigma}(g_\tau^{\sigma,a}-g_\tau)}a dm_\sigma.
\end{align*}
Note that, since $u_s$ is a Lipschitz function, $u_\sigma$ is a Lipschitz function as well. Indeed, from the dual representation of the Kantorovich-Rubinstein
distance $W^1_s$ with respect to the metric $d_s$, we deduce
\begin{align*}
 &|u_\sigma(x)-u_\sigma(y)|=\left|\int u_s(z)d\hat P_{\sigma,s}(\delta_x)(z)-\int u_s(z)d\hat P_{t,s}(\delta_y)(z)\right|\\
 &\leq \Lip_s(u_s) W_s^1(\hat P_{\sigma,s}(\delta_x),\hat P_{t,s}(\delta_y))\leq \Lip_s(u_s) W_s(\hat P_{\sigma,s}(\delta_x),\hat P_{t,s}(\delta_y))\\
 &\leq \Lip_s(u_s)W_\sigma(\delta_x,\delta_y)=\Lip_s(u_s) d_\sigma(x,y),
\end{align*}
where the last inequality is a consequence of Theorem \ref{IIeqIII}

Since $0\geq (Q_a^\sigma u_\sigma(x)-u_\sigma(x))/a\geq-2\Lip(u_\sigma)^2$ and $g_\tau^{\sigma,a}\to g_\tau$ in $L^2(X)$ the first integral vanishes.
For the second integral we use \eqref{hamilton} and estimate by Fatou's Lemma
\begin{equation*}
 \liminf_{a\to0} \int\frac{Q_a^\sigma u_\sigma-u_\sigma}a P^*_{\tau,\sigma}g_\tau dm_\sigma\geq-\frac12\int\lip_\sigma(u_\sigma)^2P^*_{\tau,\sigma}g_\tau dm_\sigma.
\end{equation*}
For the last integral an argument similar to Lemma \ref{P*} for $H_a^{\tau,g}$ (compare Lemma 4.14 in \cite{agsbe}) yields 
$$\lim_{a\to0}\int\psi_\sigma\frac{P^*_{\tau,\sigma}(g_\tau^{\sigma,a}-g_\tau)}a dm_\sigma=\int\Gamma_\tau( P_{\tau,\sigma}u_\sigma,u_\sigma) g_\tau dm_\tau.$$
Combining the last two estimates we obtain
\begin{align*}
&\liminf_{a\to0}\frac{ W_\sigma^2(\hat P_{\tau,\sigma}(g_\tau^{\sigma,a}m_\tau),\hat P_{\tau,\sigma}(g_\tau m_\tau))}{2a^2}
\geq -\frac12\int\lip_\sigma(u_\sigma)^2P^*_{\tau,\sigma}g_\tau dm_\sigma+\int\Gamma_\tau( P_{\tau,\sigma}u_\sigma,u_\sigma) g_\tau dm_\tau\\
&=-\frac12\int\Gamma_\sigma(u_\sigma)P^*_{\tau,\sigma}g_\tau dm_\sigma+\int\Gamma_\tau( P_{\tau,\sigma}u_\sigma,u_\sigma) g_\tau dm_\tau,
\end{align*}
where the last inequality follows from 
our static RCD$(K,N')$ assumption, which implies Poincar\'e inequality and doubling property for the static space $(X,d_\sigma,m_\sigma)$,
and the fact that $u_\sigma$ is a Lipschitz function (cf. \cite{Cheeger}).
\end{proof}

\begin{lemma}\label{bochner:west2}
For every $s<\sigma\leq\tau<t$,
\begin{equation*}
 \limsup_{a\to0}\frac{W_\tau^2(g_\tau^{\sigma,a}m_\tau,g_\tau m_\tau)}{2a^2}\leq \frac1{2(1-2||\psi_\sigma||_\infty)}\int\Gamma_\tau(u_\sigma)g_\tau dm_\tau.
\end{equation*}
\end{lemma}
\begin{proof}
Let $(Q_a^\tau)_{a\geq0}$ be the $d_\tau$ Hopf-Lax semigroup and fix a bounded Lipschitz function $\phi$.
Note that
\begin{align*}
 \partial_a\int Q_a^\tau(\phi)g_\tau^{\sigma,a}dm_\tau&\leq -\int\frac12\lip_\tau (Q_a^\tau\phi)^2g_\tau^{\sigma,a}dm_\tau+\int\Gamma_\tau(Q_a^\tau\phi,H_a^{\tau,g}u_\sigma)g_\tau dm_\tau\\
 &=\int\left[-\frac12\lip_\tau (Q_a^\tau\phi)^2(1+u_\sigma-H_a^{\tau,g}u_\sigma)+\Gamma_\tau(Q_a^\tau\phi,H_a^{\tau,g}u_\sigma)\right]g_\tau dm_\tau,
\end{align*}
where the inequality follows from \cite[Lemma 4.3.4]{ags} and dominated convergence.
Applying the Cauchy-Schwartz inequality and that $\Gamma_\tau(\psi)\leq \lip_\tau(\psi)$ $m_\tau$-a.e., we find
\begin{align*}
 \int\Gamma_\tau(Q_a^\tau\phi,H_a^{\tau,g}u_\sigma)g_\tau dm_\tau\leq \sqrt{\mathcal E_g(Q_a^\tau\phi)\mathcal E_g(H_a^{\tau,g}u_\sigma)}
 \leq \sqrt{\int\lip_\tau(Q_a^\tau\phi)^2g_\tau dm_\tau\mathcal E_g(H_a^{\tau,g}u_\sigma)}.
\end{align*}
Then, since $1+u_\sigma-H_a^{\tau,g}u_\sigma\geq 1-2||u_\sigma||_\infty$, we obtain using Young's inequality
\begin{align*}
 \partial_a\int Q_a^\tau(\phi)g_\tau^{\sigma,a}dm_\tau&\leq\frac1{2(1-2||u_\sigma||_\infty)}\mathcal E_g(H_a^{\tau,g}u_\sigma)\leq 
 \frac1{2(1-2||u_\sigma||_\infty)}\mathcal E_g(u_\sigma)\\
 &=\frac1{2(1-2||u_\sigma||_\infty)}\int\Gamma_\tau(u_\sigma)g_\tau dm_\tau.
\end{align*}
Integrating over $[0,a]$,
\begin{align*}
 \int Q_a^\tau\phi g_\tau^{\sigma,\tau}dm_\tau-\int\phi g_\tau dm_\tau\leq\frac{a}{2(1-2||u_\sigma||_\infty)}\int\Gamma_\tau(u_\sigma)g_\tau dm_\tau,
\end{align*}
and dividing by $a>0$ proves the claim since the Kantorovich duality can be written as
\begin{align*}
 \frac{W_\tau^2(\nu_1,\nu_2)}{2a^2}=\frac1{a}\sup_\phi\left[\int Q_a^\tau\phi d\nu_1-\int\phi d\nu_2\right]
\end{align*}
and $\phi$ was an arbitrary bounded Lipschitz function.
\end{proof}

\begin{lemma}\label{bochner:entest}
\begin{equation*}
 \liminf_{a\to0}\int_s^\tau\bigg[\frac{S_r(\hat P_{\tau,r}(g_\tau^{\sigma,a}m_\tau))-S_r(\hat P_{\tau,r}(g_\tau m_\tau))}{a}\bigg]^2dr\geq 
 \int_s^\tau\bigg[\int \Gamma_\tau\big( P_{\tau,r}(\log g_r), u_\sigma\big) g_\tau dm_\tau\bigg]^2dr.
\end{equation*}
\end{lemma}
\begin{proof}
With the same estimates as in \cite{bggk2015} we have
 \begin{align*}
& [S_r(\hat P_{\tau,r}(g_\tau^{\sigma,a}m_\tau))-S_r(\hat P_{\tau,r}(g_\tau m_\tau))]^2\\
& \geq \frac1{(1+\delta)}\bigg[\int(P^*_{\tau,r}(g_\tau^{\sigma,a})-g_r)\log g_r dm_r\bigg]^2
 -\frac1\delta\bigg[\int\frac{(P^*_{\tau,r}g_\tau^{\sigma,a}-g_r)^2}{g_r}dm_r\bigg]^2.
 \end{align*}
 Next we apply Jensen's inequality to the convex function $\alpha\colon\mathbb R\times\mathbb R_+\to\mathbb R\cup\{+\infty\}$ defined by
 \begin{align*}
  \alpha(r,s)=\begin{cases}0, &\text{ if }r=0=s,\\
               \frac{r^2}{s}, &\text{ if }s\neq 0,\\
               +\infty, &\text{ if }s=0\text{ and }r\neq 0.
              \end{cases}
\end{align*}
Recall that  the map $dx\mapsto p_{\tau,r}(x,y)dm_\tau(x)$ is not Markovian, but Lemma \ref{trivial-lp} implies
\begin{align*}
 0\leq M_{\tau,r}(y):=\int_Xp_{\tau,r}(x,y)dm_\tau(x)\leq e^{L(\tau-r)}.
\end{align*}

Hence we can write
 \begin{align*}
 &\int\alpha(P^*_{\tau,r}g_\tau^{\sigma,a}-P^*_{\tau,r}g_\tau,P^*_{\tau,r}g_\tau)dm_r\\
 &\leq \int\int \frac{\alpha((g_\tau^{\sigma,a}(x)-g_\tau(x))M_{\tau,r}(y),g_\tau(x)M_{\tau,r}(y))}{M_{\tau,r}}p_{\tau,r}(x,y)dm_\tau(x)dm_r(y)\\
 &=\int\int\alpha((g_\tau^{\sigma,a}(x)-g_\tau(x)),g_\tau(x))p_{\tau,r}(x,y)dm_\tau(x)dm_r(y)\\
 &=\int\alpha((g_\tau^{\sigma,a}(x)-g_\tau(x)),g_\tau(x))dm_\tau(x)=\int g_\tau(\psi_\sigma-H_a^{\tau,g}u_\sigma)^2dm_\tau,
 \end{align*}
 where we applied Jensen's inequality in the second, Fubini in the third, and the definition of $g_\tau^{\sigma,a}$ in the last line. 
 Dividing by $a$ and taking the $\limsup$ we end up with
 \begin{align*}
 &\limsup_{a\to0}\frac1{a}\int\frac{(P^*_{\tau,r}g_\tau^{\sigma,a}-P^*_{\tau,r}g_\tau)^2}{P^*_{\tau,r}g_\tau}dm_r
 \leq\limsup_{a\to0}\frac1{a}\int g_\tau(u_\sigma-H_a^{\tau,g}u_\sigma)^2dm_\tau\\
 &\leq\limsup_{a\to0}2||u_\sigma||_\infty\int g_\tau\left(\frac{H_a^{\tau,g}u_\sigma-u_\sigma}{a}\right)dm_\tau
 =-2||u_\sigma||_\infty\int g_\tau\Gamma_\tau(u_\sigma,1)dm_\tau
 =0.
 \end{align*}
 The first equality follows from the fact that $\frac1a(H^{\tau,g}_au_\sigma-u_\sigma)\to\Delta_\tau^g u_\sigma$ weakly in $\F^*$ (cf. Lemma \ref{P*} and 
 \cite[Lemma 4.14]{agsbe}).\\
 Since $\delta>0$ is arbitrary it suffices to show
 \begin{align*}
  \lim_{a\to0}\frac1{a}\int P^*_{\tau,r}(g(H_a^{\tau,g}u_\sigma-u_\sigma))\log P^*_{\tau,r}g dm_r=\int \Gamma_\tau\big(P_{\tau,r}(\log P^*_{\tau,r}g), u_\sigma \big) gdm_\tau.
 \end{align*}
This, indeed, follows from the fact that $P_{\tau,r}(\log P^*_{\tau,r}g)\in \F=\Dom(\E_\tau)=\Dom(\E_\tau^g)$ (thanks to uniform boundedness of $P^*_{\tau,r}g$ from above and away from 0) and from the fact that 
 $\frac1a(H^{\tau,g}_au_\sigma-u_\sigma)\to\Delta_\tau^g u_\sigma$ weakly in $\F^*$
as $a\searrow0$, more precisely (cf. Lemma \ref{P*})
 $$\frac1a\int(H_a^{\tau,g}u_\sigma-u_\sigma)\phi g_\tau dm_\tau\to-\int\Gamma_\tau(u_\sigma,\phi)g_\tau dm_\tau$$
 for all $\phi\in\F$ as $a\searrow0$.
\end{proof}

\newpage

\section{From Gradient Estimates to Dynamic EVI}

In this section we will prove that the dual heat flow is a dynamic backward EVI-gradient flow presumed that the Bakry-\'Emery gradient estimate {\bf(III)} 
holds for the (`primal') heat equation. We will present the argument only in the 
 case $N=\infty$. That is, we now assume that
 for all $u\in\Dom(\E)$ and $0<s<t< T$
\begin{equation}\label{eq:Bakry}
 \Gamma_t(P_{t,s}u)\leq P_{t,s}(\Gamma_s(u)) \quad m\text{-a.e. on }X.
\end{equation}
For the notion of dynamic backward EVI$^\pm$-gradient flow we refer to the Appendix.

As in the previous chapters, the assumptions from  section \ref{gen-ass} will always be in force, in particular, we assume the RCD$^*(K,N')$-condition for each static mm-space $(X,d_t,m_t)$ as well as boundedness and $L$-Lipschitz continuity (in $t$) for $\log d_t(x,y)$ and (in $t$ and $x$) for $f_t(x)$.

\subsection{Dynamic Kantorovich-Wasserstein Distances}\label{secdynamickant}

For the subsequent discussions let us fix a pair $(s,t)\in I\times I$ and -- if not stated otherwise -- let $\vartheta: [0,1]\to \R$ denote the linear interpolation 
\begin{equation}\label{definterpolation}
\vartheta(a)=(1-a)s+ta
\end{equation}
starting in $s$ and ending in $t$.

In the following we introduce dynamic notions of the distance between two measures `living in different time sheets'. The first notion seems to be natural 
and is defined via the length of curves, while the second one uses the approach of Hamilton Jacobi equations.

\begin{definition}
 For $s<t$ and a 2-absolutely continuous curve  $(\mu^a)_{a\in[0,1]}$ we define the action
 \begin{align*}
  \mathcal A_{s,t}(\mu)=\lim_{h\to0}\sup\Big\{\sum_{i=1}^n&(a_i-a_{i-1})^{-1}W_{\vartheta(a_{i-1})}^2(\mu^{a_{i-1}},\mu^{a_i})\Big|\\
  &0=a_0<\dots < a_n=1,a_i-a_{i-1}\leq h\Big\}.
 \end{align*}
 For two probability measures $\mu,\nu\in\mathcal P(X)$ we define
 \begin{align*}
  W_{s,t}^2(\mu,\nu)=\inf\Big\{\mathcal A_{s,t}(\mu)\Big|\mu\in AC^2([0,1],\mathcal P(X)) \text{ with } \mu_0=\mu,\mu_1=\nu\Big\}.
 \end{align*}
\end{definition}

\begin{lemma}
The following holds true.
\begin{enumerate}[i)]
 \item The action $\mu\mapsto \mathcal A_{s,t}(\mu)$ is lower semicontinuous, i.e. if $\mu^a_j\to\mu^a$ for every $a$ as $j\to\infty$ we have
 \begin{equation*}
  \mathcal A_{s,t}(\mu)\leq \liminf_{j\to\infty}\mathcal A_{s,t}(\mu_j).
 \end{equation*}
 \item For every absolutely continuous curve $\mu$
 \begin{align*}
  \mathcal A_{s,t}(\mu)
  =\lim_{h\to0}\inf\Big\{\sum_{i=1}^n(a_i-a_{i-1})^{-1}W_{\vartheta(a_{i-1})}^2(\mu^{a_{i-1}},\mu^{a_i})|0=a_0<\dots < a_n=1,a_i-a_{i-1}\leq h\Big\}.
 \end{align*}
\end{enumerate}
\end{lemma}

\begin{proof}
Since $\mu_a^j\to\mu_a$ for every $a\in[0,1]$ in the Wasserstein sense we have for every partition $0=a_0<\dots<a_n=1$
\begin{equation*}
 \sum_{i=1}^n(a_i-a_{i-1})^{-1}W^2_{\vartheta(a_{i-1})}(\mu^{a_{i-1}},\mu^{a_i})=
 \lim_{j\to\infty}\sum_{i=1}^n(a_i-a_{i-1})^{-1}W^2_{\vartheta(a_{i-1})}(\mu_j^{a_{i-1}},\mu_j^{a_i}),
\end{equation*}
and hence
\begin{equation*}
 \sum_{i=1}^n(a_i-a_{i-1})^{-1}W^2_{\vartheta(a_{i-1})}(\mu^{a_{i-1}},\mu^{a_i})\leq\liminf_{j\to\infty}\mathcal A_{s,t}(\mu_j).
\end{equation*}
Taking the supremum over each partition and letting $h\to0$ proves
\begin{equation*}
 \mathcal A_{s,t}(\mu)\leq\liminf_{j\to\infty}\mathcal A_{s,t}(\mu_j).
\end{equation*}

We prove the second assertion by contradiction. Assume that there exists a sequence $h_j\to0$, and a partition $0=a_0^j<\dots < a_{n^j}^j=1$ such that
\begin{equation*}
 a_i^j-a_{i-1}^j\leq h \quad\text{ and } \quad \lim_{j\to\infty}\sum_{i=1}^n(a_i^j-a_{i-1}^j)^{-1}W_{\vartheta(a_{i-1}^j)}^2(\mu^{a_{i-1}^j},\mu^{a_i^j})
 <\mathcal A_{s,t}(\mu).
\end{equation*}
For every $j\in\mathbb N$ we define the curve $(\mu_j^a)_{a\in[0,1]}$ by
\begin{equation*}
 \mu_j^a=\mu_{a_{i-1}^j,a_i^j}^a, \text{ if }a\in[a_{i-1}^j,a_i^j],
\end{equation*}
where $(\mu_{a_{i-1}^j,a_i^j}^a)_{a\in[a_{i-1}^j,a_i^j]}$ denotes the $W_{\vartheta(a_{i-1}^j)}$-geodesic connecting $\mu^{a_{i-1}^j}$ and $\mu{a_i^j}$.
Note that for every partition $\{\bar a_i\}_{i=1}^N$ with $\bar a_i-\bar a_{i-1}\ll h_j$
\begin{align*}
 \sum_{i=1}^N(\bar a_i-\bar a_{i-1})^{-1}W^2_{\vartheta(\bar a_{i-1})}(\mu^{\bar a_i}_j,\mu^{\bar a_{i-1}}_j)
 \leq e^{2Lh_j}\sum_{i=1}^n( a_i^j- a_{i-1}^j)^{-1}W^2_{\vartheta( a_{i-1}^j)}(\mu^{ a_i^j},\mu^{ a_{i-1}^j}),
\end{align*}
since for every $a_{i-1}^j\leq\bar a_{k-1}<\bar a_k\leq a_i^j$
\begin{equation*}
 W^2_{\vartheta(a_{i-1}^j)}(\mu_j^{\bar a_k},\mu_j^{\bar a_{k-1}})
 \leq\frac{(\bar a_k-\bar a_{k-1})^2}{(a_{i}^j-a_{i-1}^j)^2}W^2_{\vartheta(a_{i-1}^j)}(\mu^{ a_{i-1}^j},\mu^{ a_{i}^j}).
\end{equation*}

Hence
\begin{align*}
 \mathcal{A}_{s,t}(\mu_j)\leq e^{2Lh_j}\sum_{i=1}^n( a_i^j- a_{i-1}^j)^{-1}W^2_{\vartheta( a_{i-1}^j)}(\mu^{ a_i^j},\mu^{ a_{i-1}^j}).
\end{align*}
This is a contradiction since $\mu_j^a\to\mu_a$ for every $a$ and hence
\begin{equation*}
 \liminf_{j\to\infty}\mathcal{A}_{s,t}(\mu_j)\geq \mathcal A_{s,t}(\mu).
\end{equation*}

\end{proof}

\begin{proposition}
For $s<t\in I$ and $\mu^0,\mu^1\in\Pz$ we have
\begin{equation}\label{eq:ddist}
W_{s,t}^2(\mu_0,\mu_1)=\inf\left\{\int_0^1|\dot\mu^a|_{s+a(t-s)}^2da\right\}
\end{equation}
where the infimum runs over all 2-absolutely continuous curves $(\mu^a)_{a\in[0,1]}$ in $\Pz$ connecting $\mu^0$ and $\mu^1$.
\end{proposition}
\begin{proof}
 Choose an arbitrary partition $0=a_0<a_1<\dots<a_n=1$ with $a_i-a_{i-1}\leq h$. Let $(\mu^a)_{a\in[0,1]}\in AC^2([0,1],\mathcal P(X))$.
 Then, from the absolute continuity of $(\mu^a)$, and the log Lipschitz property \eqref{d-lip} we deduce
 \begin{align*}
  \sum_{i=1}^n(a_i-a_{i-1})^{-1}W^2_{\vartheta(a_{i-1})}(\mu^{a_{i-1}},\mu^{a_i})
  &\leq \sum_{i=1}^n(a_i-a_{i-1})^{-1}\left(\int_{a_i}^{a_{i-1}}|\dot\mu^{a}|_{\vartheta(a_{i-1})}da\right)^2\\
  &\leq \sum_{i=1}^n\int_{a_i}^{a_{i-1}}|\dot\mu^{a}|^2_{\vartheta(a_{i-1})}da\\
  &\leq e^{2Lh}\int_0^1|\dot\mu^{a}|^2_{\vartheta(a)}da.
 \end{align*}
Taking the supremum over all partitions and letting $h\to0$ we obtain
\begin{align*}
 \mathcal A_{s,t}(\mu)\leq \int_0^1|\dot\mu^{a}|^2_{\vartheta(a)}da,
\end{align*}
and consequently
\begin{align*}
 W_{s,t}^2(\mu_0,\mu_1)\leq\inf\left\{\int_0^1|\dot\mu^a|_{s+a(t-s)}^2da\right\}.
\end{align*}

To verify the other inequality, we fix again a curve $(\mu_a)_{a\in[0,1]}\in AC^2([0,1],\mathcal P(X))$ with finite energy $\mathcal A_{s,t}(\mu)$. 
For each $h>0$ we consider the partition 
$0=a_0<a_1<\dots<a_n\leq1<a_{n+1}$ with $a_i=ih$ and $nh\leq 1$. We extend $\mu_a$ by $\mu_1$ whenever $a>1$. 
We define $\mu_a^h$ to be the $W_{\vartheta(a_{i-1})}$-geodesic connecting $\mu_{a_{i-1}}$ with $\mu_{a_i}$ whenever $a\in [a_{i-1},a_i]$.
Then we clearly have that $\mu^h\in AC^2([0,1],\mathcal P(X))$ and since $\mu$ is 
absolutely continuous, for each $a\in[0,1]$, $\mu_a^h\to\mu_a$ in $(\mathcal P(X),W)$. Note that $|\dot \mu_a^h|_{\vartheta(a)}$ is a uniformly bounded function 
in $L^2([0,1])$
\begin{align*}
 &\int_0^1|\dot\mu_a^h|^2_{\vartheta(a)}da\leq e^{2Lh}\sum_{i=1}^{n+1}\int_{a_{i-1}}^{a_i}|\dot\mu_a^h|^2_{\vartheta(a_{i-1})}da\\
 &\leq e^{2Lh}\sum_{i=1}^{n+1}(a_{i}-a_{i-1})^{-1}W^2_{\vartheta(a_{i-1})}(\mu_{a_{i-1}},\mu_{a_i})<\infty,
\end{align*}
since $\mu_a^h$ is a piecewise geodesic and $\mathcal A_{s,t}(\mu)<\infty$. Then, by the Banach-Alaoglu Theorem there exists a subsequence (not relabeled) 
$h\to0$, and a function $A\in L^2([0,1])$ 
such that $|\dot \mu^h|_{\vartheta(.)}\rightharpoonup A$ in $L^2([0,1])$. Hence from the convergence of $\mu_a^h\to\mu_a$ we get
\begin{align*}
 &W_{\vartheta(a)}(\mu_a,\mu_{a+\delta})=\lim_{h\to0}W_{\vartheta(a)}(\mu_a^h,\mu_{a+\delta}^h)\\
 &\leq\liminf_{h\to0}\int_a^{a+\delta}|\dot\mu_b|_{\vartheta(a)}db\leq \liminf_{h\to0}e^{\delta(t-s)}\int_a^{a+\delta}|\dot\mu_b|_{\vartheta(b)}db\\
 &=e^{\delta(t-s)}\int_a^{a+\delta}A(b)db,
\end{align*}
and hence
\begin{align*}
 |\dot\mu_a|_{\vartheta(a)}\leq A(a) \text{ for a.e. }a\in[0,1].
\end{align*}
Consequently,
\begin{align*}
  &\int_0^1|\dot\mu_a|^2_{\vartheta(a)}da\leq \int_0^1 A^2(a)da\leq\liminf_{h\to0}\int_0^1|\dot\mu_a^h|^2_{\vartheta(a)}da\\
  &\leq \liminf_{h\to0}e^{2Lh}\sum_{i=1}^{n+1}\int_{a_{i-1}}^{a_i}|\dot\mu_a^h|^2_{\vartheta(a_{i-1})}da
  \leq \liminf_{h\to0}e^{2Lh}\sum_{i=1}^{n+1}(a_{i}-a_{i-1})^{-1}W_{\vartheta(a_{i-1})}^2(\mu_{a_{i-1}},\mu_{a_i})\\
  &\leq \mathcal A_{s,t}(\mu),
\end{align*}
which proves the claim.

\end{proof}

To conclude this section we define a dynamic `dual distance' inspired by the dual formulation of the Kantorovich distance. We introduce the function space
$HLS_\vartheta$ defined by
\begin{align*}
 HLS_{\vartheta}:=\bigg\{&\varphi\in{\Lip}_b([a_0,a_1]\times X)\bigg| 
 \ \partial_a\varphi_a\leq -\frac12\Gamma_{\vartheta(a)}(\varphi_a) \quad L^1\times m\text{ a.e. in }(a_0,a_1)\times X\bigg\}.
\end{align*}

In particular for all nonnegative $\phi\in L^1(X)$ and $\varphi\in HLS_{\vartheta}$
\begin{equation*}
\int\phi\varphi_{a_1}dm-\int\phi\varphi_{a_0}dm\leq-\frac12\int_{a_0}^{a_1}\int\phi\Gamma_{\vartheta(a)}(\varphi_a)dmda .
\end{equation*}

\begin{definition}
Let $s<t$ and let $\vartheta\colon[a_0,a_1]\to[s,t]$ denote the linear interpolation. Define for two 
probability measures $\mu_0,\mu_1$
\begin{equation*}
\tilde W^2_{\vartheta}(\mu_0,\mu_1):=2\sup_\varphi\left\{\int\varphi_{a_1}d\mu_1-\int\varphi_{a_0}d\mu_0\right\},
\end{equation*}
where the supremum runs over all maps $\varphi(a,x)=\varphi_a(x)\in HLS_{\vartheta}$. 
\end{definition}

Note that $\tilde W_{\vartheta}$ does not necessarily define a distance. It does not even have to be symmetric.
The next Lemma collects two essential properties of
$\tilde W_{\vartheta}$.
\begin{lemma}\label{duallsc}
The following holds true.
\begin{enumerate}
 \item $\tilde W_{\vartheta}$ is lower semicontinuous with respect to the weak-$^*$topology on $\mathcal P(X)\times\mathcal P(X)$.
 \item For every $\mu_0,\mu_1$
 \begin{equation}\label{dualcomp}
W_s^2(\mu_0,\mu_1)\leq e^{2L|s-t|}(a_1-a_0)\tilde W_{\vartheta}^2(\mu_0,\mu_1).
\end{equation}
\end{enumerate}
\end{lemma}

\begin{proof}
To show the first assertion, let $\mu_0,\mu_1\in\mathcal P(X)$ and choose $\varphi\in HLS_\vartheta$ almost optimal, i.e.
\begin{align*}
 \frac12\tilde W_\vartheta(\mu_0,\mu_1)\leq\int\varphi_{a_1}d\mu_{1}-\int\varphi_{a_0}d\mu_0-\varepsilon,
\end{align*}
where $\varepsilon>0$. Let $\mu_0^n\to\mu_0$, $\mu_1^n\to\mu$ be two sequences converging in duality with continuous bounded functions on $X$.
then, since $\varphi_{a_1}$ and $\varphi_{a_0}$ belong to $\mathcal C_b(X)$,
\begin{align*}
 \frac12\tilde W_\vartheta(\mu_0,\mu_1)&\leq\int\varphi_{a_1}d\mu_{a_1}-\int\varphi_{a_0}-\varepsilon\\
 &=\lim_{n\to\infty}\left\{\int\varphi_{a_1}d\mu_{1}^n-\int\varphi_{a_0}d\mu_0^n\right\}-\varepsilon\\
 &\leq \frac12\liminf_{n\to\infty}\tilde W_\vartheta(\mu_0^n,\mu_1^n)-\varepsilon.
\end{align*}
This proves,
since $\varepsilon>0$ was arbitrary, that $\tilde W_{\vartheta}$ is lower semicontinuous with respect to the 
weak-$^*$topology on $\mathcal P(X)\times \mathcal P(X)$.
The second statement follows from the Kantorovich duality.
Indeed, let $\varphi\in{\Lip}_b(X)$. As already mentioned above the Hopf-Lax semigroup $\varphi_b:=Q^s_b(\varphi)$ solves 
\begin{equation}\label{HopLaxstatic}
\frac{d}{db}\varphi_b\leq-\frac12\Gamma_s(\varphi_b)\leq -\frac12 e^{-2L|s-t|}\Gamma_{(1-b)s+bt}(\varphi_b)  \quad L^1\times m\text{ a.e. in} (0,1)\times X.
\end{equation}
Set $\tilde\varphi_a:=e^{-2L|s-t|}(a_1-a_0)^{-1}\varphi_{\gamma(a)}$, where $\gamma\colon[a_0,a_1]\to[0,1]$ with $\gamma(a)=\frac{a-a_0}{a_1-a_0}$. Then $\tilde\varphi$
solves
\begin{equation*}
\frac{d}{da}\tilde\varphi_a\leq-\frac12\Gamma_{\vartheta(a)}(\tilde\varphi_a) \text{ in } (a_0,a_1)\times X,
\end{equation*}
and
\begin{equation*}
 e^{-2L|s-t|}(a_1-a_0)^{-1}\left(\int\varphi_{1}d\mu_1-\int\varphi_{0}d\mu_0\right)=\int\tilde\varphi_{a_1}d\mu_1-\int\tilde\varphi_{a_0}d\mu_0.
\end{equation*}
Hence
\begin{equation*}
 e^{-2L|s-t|}(a_1-a_0)^{-1}\left(\int\varphi_{1}d\mu_1-\int\varphi_{0}d\mu_0\right)\leq \frac12\tilde W_{\vartheta}^2(\mu_0,\mu_1).
\end{equation*}

Taking the supremum among all $\varphi$ the Kantorovich duality for the metric $W_s$ implies
\begin{equation*}
W_s^2(\mu_0,\mu_1)\leq e^{2L|s-t|}(a_1-a_0)\tilde W_{\vartheta}^2(\mu_0,\mu_1).
\end{equation*}

\end{proof}
 
\begin{proposition}\label{dualsmallerw}
Let $\vartheta\colon[0,1]\to [s,t]$ be the linear interpolation. Then we have
 $\tilde W_{\vartheta}\leq  W_{s,t}$.
\end{proposition}

\begin{proof}
Fix $\varphi\in HJS_\vartheta$ and $(\mu)_{a\in[0,1]}$ 2-absolutely continuous curve. We subdivide $[0,1]$ into $l$ intervals $[(k-1)/l,k/l]$ of length
$\frac1l$. On each interval $[(k-1)/l,k/l]$ we approximate $(\mu_a)_{| [(k-1)/l,k/l]}$ by regular curves $(\rho_a^{n,k})_{a\in[(k-1)/l,k/l]}$.
Obviously, for each $k,n$ the map $[(k-1)/l,k/l]\ni a\mapsto \int\varphi_ad\rho_a^{k,n}$ is absolutely continuous;
 \begin{align*}
 \int\varphi_{a+h}d\rho_{a+h}-\int\varphi_ad\rho_a\leq \Lip(\varphi_{a+h})W(\rho_{a+h},\rho_a)+||\varphi_{a+h}-\varphi_a||_\infty.
 \end{align*}
 Let $u_a^{k,n}$ be the density of the regular curve $\rho_a^{k,n}$.
 Hence for fixed $k,n$
 \begin{align*}
  \frac{d}{da}\int\varphi_a u_a^{k,n}dm\leq\int\varphi_a\dot u_a^{k,n} dm-\frac12\int u_a^{k,n}\Gamma_{\vartheta(a)}(\varphi_a)dm
 \end{align*}
From Lemma \ref{oscillation} we deduce
 \begin{align*}
 \int \dot u_a^{k,n}\varphi_adm\leq\frac12|\dot\rho_a^{k,n}|_{\vartheta(k-1/l)}^2+\frac12\int(\mathrm{lip}_{\vartheta(k-1/l)}\varphi_a)^2d\rho_a^{k,n}.
 \end{align*}
 Adding these two inequalities, integrating over $[(k-1)/l,k/l]$ and noting that 
 $$e^{-L\frac{|t-s|}{l}}(\mathrm{lip}_{\vartheta(k-1/l)}(\varphi_a))^2\leq\Gamma_{\vartheta(a)}(\varphi_a) \qquad  m\text{ a.e.,}$$
 we obtain
 \begin{align*}
  &\int\varphi_{k/l} u_{k/l}^{k,n}dm-\int\varphi_{k-1/l} u_{k-1/l}^{k,n}dm\\
  &\leq\frac12\int_{k-1/l}^{k/l}|\dot\rho_a^{k,n}|^2_{\vartheta(k-1/l)}da+\frac12(1-e^{-L\frac{|t-s|}{l}})\int_{k-1/l}^{k/l}\int(\mathrm{lip}_{\vartheta(k-1/l)}\varphi_a)^2d\rho_a^{k,n}da\\
  &\leq \frac12\int_{k-1/l}^{k/l}|\dot\rho_a^{k,n}|^2_{\vartheta(k-1/l)}da+\frac{C_1}{2l}(1-e^{-L\frac{|t-s|}{l}})
 \end{align*}
Taking the limit $n\to\infty$ (and taking the scaling into account) gives
 \begin{align*}
  \int\varphi_{k/l}d\mu_{k/l}-\int\varphi_{k-1/l} d\mu_{k-1/l}\leq\frac12lW^2_{\vartheta(k-1/l)}(\mu_{k-1/l},\mu_{k/l})+\frac{C_1}{2l}(1-e^{-L\frac{|t-s|}{l}}).
 \end{align*}
 Summing over each partition and noting that the left hand side is a telescoping sum yields
 \begin{align*}
  \int\varphi_{1}d\mu_{1}-\int\varphi_{0} d\mu_{0}\leq\frac12\sum_{k=1}^l lW^2_{\vartheta(k-1/l)}(\mu_{k-1/l},\mu_{k/l})+\frac{C_1}{2}(1-e^{-L\frac{|t-s|}{l}}).
 \end{align*}
 Letting $l\to\infty$ we obtain the desired estimate.
\end{proof}

\begin{corollary}\label{cordual}
Let $s<t$ and $[0,1]\ni a\mapsto\vartheta(a)=(1-a)s+at$. Then for every $\mu_0,\mu_1\in\mathcal P(X)$ we have
 \begin{align*}
  W_{s,t}(\mu_0,\mu_1)=\tilde W_{\vartheta}(\mu_o,\mu_1).
 \end{align*}

\end{corollary}
\begin{proof}
We already know from Proposition \ref{dualsmallerw} that $W_{s,t}(\mu_0,\mu_1)\geq W_{\vartheta^*}(\mu_o,\mu_1)$. Hence it remains to prove the other 
inequality.

 For this let $(\varphi_a)\in HLS_\vartheta$, and $(\mu_a)$ an absolutely continuous curve connecting $\mu_0$ and $\mu_1$.
 
 Consider the Partition $0=a_0<a_1<\dots a_n=1$ with $a_i-a_{i-1}\leq h$ for some $h>0$.
 Set 
 \begin{align*}
  [a_{i-1},a_i]\ni a\mapsto \vartheta_i(a)=\frac{a_i-a}{a_i-a_{i-1}}\vartheta(a_{i-1})+\frac{a-a_{i-1}}{a_i-a_{i-1}}\vartheta(a_{i})
 \end{align*}
 and 
  $\tilde \varphi_a^i=\varphi_a\rvert_{[a_{i-1},a_i]}$. Notice that $(\varphi_a^i)_a$ is in $HLS_{\vartheta_i}$. Hence
  \begin{align*}
   \tilde W_{\vartheta_i}^2(\mu_{a_{i-1}},\mu_{a_i})\leq 2\left\{\int\varphi_{a_i}d\mu_{a_i}-\int\varphi_{a_{i-1}}d\mu_{a_{i-1}}\right\}.
  \end{align*}
  Then summing over the partitions and taking the scalings into account we end up with
  \begin{align*}
   \sum_{i=1}^n(a_i-a_{i-1})^{-1}W_{\vartheta(a_{i-1})}^2(\mu_{a_{i-1}},\mu_{a_i})&\leq e^{2Lh|s-t|}\sum_{i=1}^n\tilde W_{\vartheta_i}^2(\mu_{a_{i-1}},\mu_{a_i})\\
   &\leq 2e^{2Lh|s-t|}\sum_{i=1}^n\left\{\int\varphi_{a_i}d\mu_{a_i}-\int\varphi_{a_{i-1}}d\mu_{a_{i-1}}\right\}\\
   &=2e^{2Lh|s-t|}\left\{\int\varphi_{1}d\mu_{1}-\int\varphi_{0}d\mu_{0}\right\},
  \end{align*}
where we made use of Lemma \ref{duallsc}(ii) in the first inequality. Taking the supremum over all $(\varphi_a)\in HLS_{\vartheta}$ we deduce
\begin{align}\label{wsmallerdual}
   \sum_{i=1}^n(a_i-a_{i-1})^{-1}W_{\vartheta(a_{i-1})}^2(\mu_{a_{i-1}},\mu_{a_i})
   \leq e^{2Lh|s-t|}\tilde W^2_{\vartheta}(\mu_0,\mu_1),
  \end{align}
We conclude
\begin{align*}
 W^2_{s,t}(\mu_0,\mu_1)\leq \tilde W^2_{\vartheta}(\mu_0,\mu_1),
\end{align*}
from taking the supremum in \eqref{wsmallerdual} over the partition $0=a_0<a_1<\dots<a_n=1$ with $a_i-a_{i-1}<h$ and subsequently letting $h\searrow0$.

\end{proof}

\bigskip

\subsection{Action Estimates}
Let us recall the following estimate about the oscillation of $a\mapsto\int\varphi d\rho^a$ from \cite[Lemma 4.12]{agsbe}.
For fixed  $t>0$, let $(\rho^a)_a$ be a 2-absolutely continuous curve in $\Pz$ with $\rho^a=u^am_t$ and $u\in\mathcal{C}^1((0,1),L^1(X,m_t))$. 
Then for any Lipschitz function $\varphi$ we have
\begin{equation}\label{oscillation}
\left|\int \dot u^a\varphi dm_{t}\right|\leq \frac12 |\dot\rho^a|_{t}^2+\frac12\int\Gamma_{t}(\varphi)d\rho^a.
\end{equation}
Actually,  we have inequality \eqref{oscillation} for each $\varphi\in\Dom(\mathcal E)$ 
since we assume that each $(X,d_t,m_t)$ is a static RCD$(K,\infty)$ which implies 
 that Lipschitz functions are dense in the domain of the quadratic form $\mathcal{E}$ with respect
to the norm $\sqrt{||\varphi||^2+\mathcal E(\varphi)}$ (Proposition 4.10 in \cite{agsmet}).

Moreover we will use the following result about difference quotients and concatenations of functions in $\mathcal F_{(s,t)}$.
\begin{lemma}\label{diffquot}
Let $0<s<T$.
\begin{enumerate}
 \item Let $u\in\mathcal F_{(s,t)}$. Then for almost every $a\in(s,t)$
 \begin{equation*}
  \frac1h(u_{a+h}-u_a)\to \partial_au_a \text{ weakly}^* \text{ in }\mathcal F^*,
 \end{equation*}
i.e. for every $v\in\mathcal F$ and for almost every $a\in(s,t)$
\begin{align*}
 \int\frac1h(u_{a+h}-u_a)vdm_\diamond\to \langle\partial_au_a,v\rangle.
\end{align*}
\item For $u\in\mathcal F_{(s,t)}$ and $\vartheta\in\mathcal C^1([0,1])$ the linear interpolation from $s$ to $t$,
we have that $(u\circ\vartheta)\in\mathcal F_{(0,1)}$ with distributional derivative
\begin{align*}
 \partial_a(u\circ\vartheta)(a)=(t-s)\partial_au_{\vartheta(a)}.
\end{align*}
\end{enumerate}
\end{lemma}
 
\begin{proof}
 From Corollary 5.6. in \cite{lierl2015} it follows for $u\in\mathcal F_{(s,t)}$ and $v\in\mathcal F$
 \begin{equation*}
  \int u_{a+h} vdm_\diamond-\int u_avdm_\diamond=\int_a^{a+h}\langle\partial_bu_b,v\rangle db.
 \end{equation*}
 Since $b\mapsto\langle\partial_bu_b,v\rangle$ is in $L^1(s,t)$ we apply the Lebesgue differentiation theorem
 and obtain that for almost every $a\in(s,t)$
 \begin{align*}
 \lim_{h\to0} \frac1h\int u_{a+h} vdm_\diamond-\int u_avdm_\diamond=\lim_{h\to0}\frac1h\int_a^{a+h}\langle\partial_bu_b,v\rangle db
 =\langle\partial_au_a,v\rangle.
 \end{align*}
 This proves the first assertion. To show the second recall that we can approximate each $u\in\mathcal F_{(s,t)}$ by smooth functions 
 $(u^n)\subset \mathcal C^\infty([s,t]\to\mathcal F)$ by virtue of \cite[Lemma 5.3]{lierl2015}. So for each $n\in\mathbb N$ and for each smooth
 compactly supported test function
 $\psi\colon (0,1)\to\mathcal F$ we have that
 \begin{align*}
  \int_0^1\int(u^n\circ\vartheta)(a)\partial_a\psi_adm_\diamond da=-\int_0^1\int\dot\vartheta(a)\partial_au^n_{\vartheta(a)}\psi_adm_\diamond da.
 \end{align*}
 Note that the term on the left-hand side converges to $\int_0^1\int(u\circ\vartheta)(a)\partial_a\psi_adm_\diamond da$ as $n\to\infty$ since
 \begin{align*}
  \left|\int_0^1\int(u^n\circ\vartheta-u\circ\vartheta)\partial_a\psi_adm_\diamond da\right|
  \leq(t-s)^{-1}\int_s^t||u^n_a-u_a||_{\mathcal F}||\partial_a\psi_{\vartheta^{-1}(a)}||_{\mathcal F}da,
 \end{align*}
where we applied integration by substitution. Similarly for the right-hand side
\begin{align*}
 \left|\int_0^1\dot\vartheta(a)\langle\partial_au^n_{\vartheta(a)}-\partial_au_{\vartheta(a)},\psi_a\rangle dm_\diamond da\right|
 \leq\int_s^t||\partial_au^n_{a}-\partial_au_{a}||_{\mathcal F^*}||\psi_{\vartheta^{-1}(a)}||_{\mathcal F}da,
\end{align*}
and consequently as $n\to\infty$
 \begin{align*}
  \int_0^1\int(u\circ\vartheta)(a)\partial_a\psi_adm_\diamond da=-\int_0^1(t-s)\langle\partial_au_{\vartheta(a)},\psi_a\rangle da,
 \end{align*}
which is the assertion.
\end{proof}

For the following lemmas let $(\rho_a)_{a\in[0,1]}$ be a regular curve and let $\vartheta\colon[0,1]\to [0,\infty)$ 
\begin{align*}
 \vartheta(a):=(1-a)s+at, \text{ where }s<t.
\end{align*}
Set $\rho_{a,\vartheta}:=\hat P_{t,\vartheta(a)}(\rho_a)=u_{a,\vartheta}m_{\vartheta(a)}$.

\begin{lemma}\label{hfofreg}
 The curve $(u_{a,\vartheta})_{a\in[0,1]}$ belongs to $\Lip([0,1],\mathcal F^*)$ with $u_{a,\vartheta}\in L^2([0,1]\to\mathcal F)$ and distributional derivative
 $\partial_a u_{a,\vartheta}\in L^\infty([0,1]\to\mathcal F^*)$ satisfying
 \begin{align*}
  \partial_au_{a,\vartheta}=-(t-s)\Delta_{\vartheta(a)}u_{a,\vartheta}+\partial_af_{\vartheta(a)}u_{a,\vartheta}-P_{t,\vartheta(a)}^*(\dot u_a).
 \end{align*}
\end{lemma}
\begin{proof}
 First we show that $(u_{a,\vartheta})$ is in $L^2([0,1]\to\mathcal F)$. For this recall that, since $(\rho_a)$ is regular, 
 $u_a\leq R$ and $\mathcal E_t(\sqrt{u_a})\leq E$ for all $a\in[0,1]$ and hence by Lemma \ref{trivial-lp}
 we get
 \begin{align*}
  \int_0^1||u_{a,\vartheta}||_{L^2(m_{\vartheta(a)})}^2da\leq e^{L(t-s)}\int_0^1||u_a||^2_{L^2(m_t)}da\leq Re^{L(t-s)}\int_0^1||u_a||_{L^1(m_t)}da=Re^{L(t-s)},
 \end{align*}
and by Theorem \ref{energy-est}
\begin{align*}
 \int_0^1\mathcal E_{\vartheta(a)}(u_{a,\vartheta})da&\leq e^{3L(t-s)}\int[\mathcal E_t(u_a)+||u_a||^2_{L^2(m_t)}]da\\
 &\leq e^{3L(t-s)}\sqrt{R}[\int_0^12\mathcal E_t(\sqrt{u_a})da+R]\leq e^{3L(t-s)}\sqrt{R}(2E+R).
\end{align*}
This shows that $(u_{a,\vartheta})$ is in $L^2([0,1]\to \mathcal F)$. 

Next we show that $(u_{a,\vartheta})$ is contained in $\Lip([0,1],\mathcal F^*)$. For this let $\psi\in \mathcal F$. Then, for almost every $a_0,a_1\in(0,1)$, we obtain with Lemma \ref{diffquot}, since $P_{t,\vartheta(a)}^*u_{a_0}\in\mathcal F_{(0,1)}$,
\begin{align*}
&\int \psi u_{a_1,\vartheta}dm_\diamond-\int\psi u_{a_0,\vartheta}dm_\diamond\\
=&\int\psi( P_{t,\vartheta(a_1)}^*u_{a_0}-P_{t,\vartheta(a_0)}^*u_{a_0})dm_{\diamond}+
\int\psi P_{t,\vartheta(a_1)}^*(u_{a_1}-u_{a_0})dm_\diamond\\
=&(t-s)\int_{a_0}^{a_1}\mathcal E_{\vartheta(a)}^\diamond(P_{t,\vartheta(a)}^*u_{a_0},\psi)da+
(t-s)\int_{a_0}^{a_1}\int\dot f_{\vartheta(a)}P_{t,\vartheta(a)}^*u_{a_0}\psi dm_\diamond da\\
&+\int P_{t,\vartheta(a_1)}(\psi e^{f_{\vartheta(a_1)}})(u_{a_1}-u_{a_0})dm_{t}\\
\leq& (t-s)\int_{a_0}^{a_1}\mathcal E_{\vartheta(a)}(P_{t,\vartheta(a)}^*u_{a_0})^{1/2}\mathcal E_{\vartheta(a)}(\psi e^{f_{\vartheta(a)}})^{1/2}da\\
&+(t-s)\int_{a_0}^{a_1}||\dot f_{\vartheta(a)}||_\infty||P_{t,\vartheta(a)}^*u_{a_0}||_{L^2(m_{\vartheta(a)})}||\psi e^{f_{\vartheta(a)}}||_{L^2(m_\diamond)} da\\
&+ ||e^{-f_{t}}||_\infty\mathcal E_\diamond(P_{t,\vartheta(a_1)}(\psi e^{f_{\vartheta(a_1)}}))^{1/2}\sup_a||\dot u_a||_{\mathcal F^*}(a_1-a_0)\\
\leq& (t-s)\mathcal E_{\vartheta(a)}(\psi)^{1/2}\int_{a_0}^{a_1}\Lip(f_{\vartheta(a)})\mathcal E_{\vartheta(a)}(P_{t,\vartheta(a)}^*u_{a_0})^{1/2}da\\
&+(t-s)\int_{a_0}^{a_1}||\dot f_{\vartheta(a)}||_\infty||P_{t,\vartheta(a)}^*u_{a_0}||_{L^2(m_{\vartheta(a)})}||\psi e^{f_{\vartheta(a)}}||_{L^2(m_\diamond)} da\\
&+||e^{-f_{t}}||_\infty\mathcal E_\diamond(P_{t,\vartheta(a_1)}(\psi e^{f_{\vartheta(a_1)}}))^{1/2}\sup_a||\dot u_a||_{\mathcal F^*}(a_1-a_0) .
\end{align*}
Due to our assumptions on $f$ we have that
\begin{align*}
\Lip(f_{\vartheta(a)})\leq C,\ ||\dot f_{\vartheta(a)}||_\infty\leq L,\ ||f_{t}||_\infty\leq C,
\end{align*}
while the energy estimate Theorem \ref{energy-est} and Corollary \ref{trivial-lp} yields
\begin{align*}
\mathcal E_{\vartheta(a)}(P_{t,\vartheta(a)}^*u_{a_0})&\leq e^{3L(t-s)}[\mathcal E_t(u_{a_0})+||u_{a_0}||^2_{L^2(m_t)}],\\
||P^*_{t,\vartheta(a)}u_{a_0}||_{L^2(m_{\vartheta(a)})}&\leq e^{L(t-s)/2}||u_{a_0}||_{L^2(m_t)}.
 \end{align*}
 Note that the last two expressions are bounded since $u$ is a regular curve. Moreover from \eqref{G-unif}, the gradient estimate \eqref{eq:Bakry} and Corollary \ref{trivial-lp} we find
 \begin{align*}
 \mathcal E_\diamond(P_{t,\vartheta(a_1)}(\psi e^{f_{\vartheta(a_1)}}))\leq Ce^{L(t-s)}\Lip(e^{f_{\vartheta(a_1)}})^2\mathcal E_{\vartheta(a_1)}(\psi)
 \end{align*}
 Applying \eqref{G-unif} once more we find that there exists a constant $\lambda$ such that
\begin{align}\label{conteq}
\int \psi u_{a_1,\vartheta}dm_\diamond-\int\psi u_{a_0,\vartheta}dm_\diamond
\leq ({a_1}-{a_0)}\lambda||\psi||_{\mathcal F},
\end{align}
and thus
\begin{align*}
||u_{a_1}-u_{a_0}||_{\mathcal F^*}\leq \lambda.
\end{align*}
Note also that \eqref{conteq} holds for every $a_0,a_1$ by approximating with Lebesgue points.
This implies the existence of $\partial_a u_{a,\vartheta}\in L^\infty([0,1],\mathcal F^*)$ such that
\begin{align*}
\int \psi u_{a_1,\vartheta}dm_\diamond-\int\psi u_{a_0,\vartheta}dm_\diamond=\int_{a_0}^{a_1}\langle \partial_au_{a,\vartheta},\psi\rangle_{\mathcal F^*,\mathcal F}da.
\end{align*}

Fix $\psi\in\Lip_b(X)$.
By a similar calculation as above it ultimately follows that
\begin{align*}
&\lim_{h\to0}\frac1h(\int \psi u_{a+h,\vartheta}dm_\diamond-\int\psi u_{a,\vartheta}dm_\diamond)\\
&=(t-s)\mathcal E_{\vartheta(a)}^\diamond(P_{t,\vartheta(a)}^*u_{a},\psi)+
(t-s)\int\dot f_{\vartheta(a)}P_{t,\vartheta(a)}^*u_{a}\psi dm_\diamond\\
&+\lim_{h\to0}\int P_{t,\vartheta(a+h)}(\psi e^{f_{\vartheta(a+h)}})\frac{(u_{a+h}-u_{a})}{h}dm_{t}
\end{align*}
almost everywhere. To determine the last integral recall that $u\in\mathcal C^1([0,1],L^1(X))$. Then since $\psi\in\Lip_b(X)$
\begin{align*}
\lim_{h\to0}\int P_{t,\vartheta(a+h)}(\psi e^{f_{\vartheta(a+h)}})\frac{(u_{a+h}-u_{a})}{h}dm_{t}=\int P_{t,\vartheta(a)}(\psi e^{f_{\vartheta(a)}})\dot u_adm_{t}\\
=\int(\psi e^{f_{\vartheta(a)}}) P_{t,\vartheta(a)}^*\dot u_adm_{\vartheta({a})}=\langle  P_{t,\vartheta(a)}^*\dot u_a,\psi\rangle_{\mathcal F^*,\mathcal F}.
 \end{align*} 
 From the Lipschitz continuity of $(u_{a,\vartheta})$ we deduce that for almost every $a\in[0,1]$
 \begin{align*}
 \langle \partial_au_{a,\vartheta},\psi\rangle_{\mathcal F^*,\mathcal F}=\langle-(t-s)\Delta_{\vartheta(a)}u_{a,\vartheta}+\partial_af_{\vartheta(a)}u_{a,\vartheta}-P_{t,\vartheta(a)}^*(\dot u_a),\psi\rangle_{\mathcal F^*,\mathcal F}.
 \end{align*}
 We conclude the proof by approximating $\psi\in\mathcal F$ with bounded Lipschitz functions.
\end{proof}

\begin{lemma}\label{actionder}
For any map $\varphi\in HLS_\vartheta$
  the map $a\mapsto\int\varphi_a d\rho_{a,\vartheta}$ is absolutely continuous and 
\begin{equation*}
\begin{aligned}
\int\varphi_1d\rho_{1,\vartheta}-\int\varphi_0d\rho_{0,\vartheta}\leq\int_0^1\Big[&-\frac12\int\Gamma_{\vartheta(a)}(\varphi_a)d\rho_{a,\vartheta}+\int P_{t,\vartheta(a)}(\varphi_a)\,\partial_a u_a\,dm_t\\
&+(t-s)\int\Gamma_{\vartheta(a)}(\varphi_a,u_{a,\vartheta})dm_{\vartheta(a)}\Big]da.
\end{aligned}
\end{equation*}
\end{lemma}

\begin{proof}
Let us begin by showing that $a\mapsto \rho_{a,\vartheta}$ is 2-absolutely continuous. Indeed, let $a_0<a_1$, we have with the equivalence of 
the gradient estimate \eqref{eq:Bakry} and the Wasserstein contraction \eqref{II} 
\begin{align*}
 W_{\vartheta(a_0)}(\rho_{{a_0},\vartheta},\rho_{{a_1},\vartheta})&\leq W_{\vartheta(a_0)}(\hat P_{t,\vartheta(a_0)}\rho_{{a_0}},\hat P_{t,\vartheta(a_0)}\rho_{{a_1}})+W_{\vartheta(a_0)}(\hat P_{t,\vartheta(a_0)}\rho_{a_1},\hat P_{t,\vartheta(a_1)}\rho_{a_1})\\
 &\leq W_{t}(\rho_{{a_0}},\rho_{{a_1}})+W_{\vartheta(a_0)}(\hat P_{t,\vartheta(a_0)}\rho_{a_1},\hat P_{t,\vartheta(a_1)}\rho_{a_1}).
\end{align*}
By virtue of Lemma \ref{p-cont}(iv) we have that $\tilde\rho_a=\hat P_{t,\vartheta(a)}\rho_{a_1}=\tilde u_a m_{\vartheta(a)}$  is in $AC^2([0,1],\mathcal P(X))$.
This proves that $a\mapsto \rho_{a,\vartheta}$ is 2-absolutely continuous.

To conclude that $a\mapsto\int\varphi_ad\rho_{a,\vartheta}$ is absolutely continuous we write 
\begin{align*}
&\int\varphi_{a_1}d\rho_{a_1,\vartheta}-\int\varphi_{a_0}d\rho_{a_0,\vartheta}\\
&=\int(\varphi_{a_1}-\varphi_{a_0})d\rho_{a_1,\vartheta}+\int\varphi_{a_0}d\rho_{a_1,\vartheta}-\int\varphi_{a_0}d\rho_{a_0,\vartheta}\\
&\leq||\varphi_{a_1}-\varphi_{a_0}||_\infty+\Lip(\varphi_{a_0}) W(\rho_{a_1,\vartheta},\rho_{a_0,\vartheta}).
\end{align*}
To compute its derivative we consider difference quotients.
Since $\varphi\in \Lip([0,1],L^\infty(X))$ is in $HLS_\vartheta$ and
$u_{a+h,\vartheta}\to u_{a,\vartheta}$ in $L^1(X)$  we have
\begin{equation}\label{exp1}
\lim_{h\to0}h^{-1}\int(\varphi_{a+h}-\varphi_a)d\rho_{a+h,\vartheta}\leq-\frac12\int|\nabla_{\vartheta(a)}\varphi_a|^2d\rho_{a,\vartheta}.
\end{equation}
Now we need to determine
\begin{align*}
 \lim_{h\to0}\frac1h(\int\varphi_ae^{-f_{\vartheta(a)}}(u_{a+h,\vartheta}-u_{a,\vartheta})dm_\diamond+\int\varphi_au_{a+h,\vartheta}d(m_{\vartheta(a+h)}-m_{\vartheta(a)})).
 \end{align*}
 The expression on the right hand side clearly converges to 
 \begin{align}\label{exp3}
-\dot\vartheta(a)\int\varphi_a\dot f_{\vartheta(a)}u_{a,\vartheta}dm_{\vartheta(a)},
\end{align}
while from Lemma \ref{hfofreg} we deduce
 \begin{align*}
  \lim_{h\to0}\int e^{-f_{\vartheta(a)}}\varphi_a\frac1h( u_{a+h,\vartheta}- u_{a,\vartheta})dm_{\diamond}
 =&\langle\partial_a u_{a,\vartheta},e^{-f_{\vartheta(a)}}\varphi_a\rangle_{\mathcal F,\mathcal F^*},
 \end{align*}
and after inserting
\begin{align}
\langle\partial_a u_a,e^{-f_{\vartheta(a)}}\varphi_a\rangle_{\mathcal F,\mathcal F^*}
 =&(t-s)\Big(\int\dot f_{\vartheta(a)} u_{a,\vartheta}\varphi_ae^{-f_{\vartheta(a)}}dm_{\diamond}
 +\mathcal E_{\vartheta(a)}^\diamond( u_{a,\vartheta},\varphi_a e^{-f_{\vartheta(a)}})\Big)\\
 =&(t-s)\Big(\int\dot f_{\vartheta(a)}u_{a,\vartheta}\varphi_adm_{\vartheta(a)}
 +\int\Gamma_{\vartheta(a)}(u_{a,\vartheta},\varphi_a)dm_{\vartheta(a)} \Big).\label{exp4}
 \end{align}

Then from the absolute continuity of $a\mapsto\int\varphi_ad\rho_{a,\vartheta}$ together with \eqref{exp1}, \eqref{exp3} and \eqref{exp4}, 
we obtain
\begin{align*}
&\int\varphi_1d\rho_{1,\vartheta}-\int\varphi_0d\rho_{0,\vartheta}=\int_0^1\partial_a\int\varphi_ad\rho_{a,\vartheta}da\\
 \leq&\int_0^1\Big[-\frac12\int|\nabla_{\vartheta(a)}\varphi_a|^2d\rho_{a,\vartheta}+\int P_{t,\vartheta(a)}\varphi_a\dot u_adm_t
 -(t-s)\int\varphi_a\dot f_{\vartheta(a)}u_{a,\vartheta}dm_{\vartheta(a)}\\
 &+(t-s)\int\dot f_{\vartheta(a)}u_{a,\vartheta}\varphi_adm_{\vartheta(a)}
 +(t-s)\int\Gamma_{\vartheta(a)}(u_{a,\vartheta},\varphi_a)dm_{\vartheta(a)}\Big]da\\
 \leq&\int_0^1\Big[-\frac12\int|\nabla_{\vartheta(a)}\varphi_a|^2d\rho_{a,\vartheta}+\int P_{t,\vartheta(a)}\varphi_a\dot u_adm_t
 +(t-s)\int\Gamma_{\vartheta(a)}(u_{a,\vartheta},\varphi_a)dm_{\vartheta(a)}\Big]da.
\end{align*}
\end{proof}

We regularize the entropy functional by truncating the singularities of the logarithm. Define $e_\varepsilon\colon[0,\infty)$ by setting $e'_\varepsilon(r)=\log(\varepsilon+r)+1$
and $e_\varepsilon(0)=0$. Then $e_\varepsilon$ is still a convex function and $e_\varepsilon'\in\Lip_b([0,R])$. For any $t$ and $\rho=um_t\in\mathcal P(X)$
we define
\begin{align*}
 S^\varepsilon_t(\rho)=\int e_\varepsilon(u)dm_t.
\end{align*}
Note that for any $\rho\in\Dom(S)$ we clearly have $S^\varepsilon(\rho)\to S(\rho)$as $\varepsilon\to0$.

As in \cite{agsbe} we introduce
\begin{align*}
p_\varepsilon(r):=e_\varepsilon'(r^2)-\log\varepsilon.
\end{align*}

\begin{lemma}\label{derent}
With the same notation as in Lemma \ref{actionder} we find for any $\varepsilon>0$
\begin{equation*}
\begin{aligned}
S^\varepsilon_{t}(\rho_{1,\vartheta})-S^\varepsilon_s(\rho_{0,\vartheta})\geq&\int_0^1\int\dot u_a P_{t,\vartheta(a)}(e_\varepsilon'(u_{a,\vartheta}))dm_t+4(t-s)\int e_\varepsilon''(u_{a,\vartheta})\Gamma_{\vartheta(a)}(\sqrt{u_{a,\vartheta}})d \rho_{a,\vartheta}\\
&+(t-s)\int\dot f_{\vartheta(a)}(u_{a,\vartheta}e_\varepsilon'(u_{a,\vartheta})-e_\varepsilon'(u_{a,\vartheta}))d m_{\vartheta(a)}da.
\end{aligned}
\end{equation*}
\end{lemma}

\begin{proof}
From the convexity of $e_\varepsilon $ we get for every $a_0,a_1\in[0,1]$ by virtue of Lemma \ref{hfofreg}
\begin{align*}
 &S^\varepsilon_{\vartheta(a_1)}(\rho_{a_1,\vartheta})-S^\varepsilon_{\vartheta(a_0)}(\rho_{a_0,\vartheta})\\
 =&\int e_\varepsilon(u_{a_1,\vartheta})-e_\varepsilon(u_{a_0,\vartheta})e^{-f_{\vartheta(a_0)}}dm_\diamond+\int e_\varepsilon(u_{a_1,\vartheta})(e^{-f_{\vartheta(a_1)}}-e^{-f_{\vartheta(a_0)}})dm_\diamond\\
\geq &\int e_\varepsilon'(u_{a_0,\vartheta})(u_{a_1,\vartheta}-u_{a_0,\vartheta})e^{-f_{\vartheta(a_0)}}dm_\diamond+\int e_\varepsilon(u_{a_1,\vartheta})(e^{-f_{\vartheta(a_1)}}-e^{-f_{\vartheta(a_0)}})dm_\diamond\\
=&\int_{a_0}^{a_1}(\langle\partial_au_{a,\vartheta},e^{-f_{\vartheta(a_0)}}e_\varepsilon'(u_{a_0,\vartheta})\rangle -\int e_\varepsilon(u_{a_1,\vartheta})\dot\vartheta(a)\dot f_{\vartheta(a)}e^{-f_{\vartheta(a)}}dm_\diamond) da\\
=&\int_{a_0}^{a_1}(\langle-\dot\vartheta(a)\Delta_{\vartheta(a)}u_{a,\vartheta}+\dot\vartheta(a)\dot f_{\vartheta(a)}u_{a,\vartheta}+P_{t,\vartheta(a)}^*(\dot u_a),e^{-f_{\vartheta(a_0)}}e_\varepsilon'(u_{a_0,\vartheta})\rangle\\
&-\int e_\varepsilon(u_{a_1,\vartheta})\dot\vartheta(a)\dot f_{\vartheta(a)}e^{-f_{\vartheta(a)}}dm_\diamond) da\\
=&\int_{a_0}^{a_1}(-\dot\vartheta(a)\langle\Delta_{\vartheta(a)}u_{a,\vartheta},e^{-f_{\vartheta(a_0)}}e_\varepsilon'(u_{a_0,\vartheta})\rangle+\int\dot\vartheta(a)\dot f_{\vartheta(a)}u_{a,\vartheta}e^{-f_{\vartheta(a_0)}}e_\varepsilon'(u_{a_0,\vartheta})dm_\diamond\\
&+\int P_{t,\vartheta(a)}^*(\dot u_a)e^{-f_{\vartheta(a_0)}}e_\varepsilon'(u_{a_0,\vartheta})dm_\diamond
-\int e_\varepsilon(u_{a_1,\vartheta})\dot\vartheta(a)\dot f_{\vartheta(a)}e^{-f_{\vartheta(a)}}dm_\diamond) da.
 \end{align*}
 
 Now fix $h>0$ and choose a partition of $[0,1]$ consisting of Lebesgue points $\{a_i\}_{i=0}^n$ such that $0\leq a_{i+1}-a_i\leq h$. Then
 \begin{align*}
  &S^\varepsilon_t(\rho_{1,\vartheta})-S^\varepsilon_s(\rho_{0,\vartheta})=\sum_{i=1}^n(S^\varepsilon_{\vartheta(a_i)}(\rho_{a_i,\vartheta})-S^\varepsilon_{\vartheta(a_{i-1})}(\rho_{a_{i-1},\vartheta}))\\
  \geq&\sum_{i=1}^n\int_{a_{i-1}}^{a_i}(-\dot\vartheta(a)\langle\Delta_{\vartheta(a)}u_{a,\vartheta},e^{-f_{\vartheta(a_{i-1})}}e_\varepsilon'(u_{a_{i-1},\vartheta})\rangle+\int\dot\vartheta(a)\dot f_{\vartheta(a)}u_{a,\vartheta}e^{-f_{\vartheta(a_{i-1})}}e_\varepsilon'(u_{a_{i-1},\vartheta})dm_\diamond\\
&+\int P_{t,\vartheta(a)}^*(\dot u_a)e^{-f_{\vartheta(a_{i-1})}}e_\varepsilon'(u_{a_{i-1},\vartheta})dm_\diamond
-\int e_\varepsilon(u_{a_i,\vartheta})\dot\vartheta(a)\dot f_{\vartheta(a)}e^{-f_{\vartheta(a)}}dm_\diamond) da\\
=&\int_{0}^{1}(-\dot\vartheta(a)\langle\Delta_{\vartheta(a)}u_{a,\vartheta},\varsigma_a^h\rangle+\int\dot\vartheta(a)\dot f_{\vartheta(a)}u_{a,\vartheta}\varsigma_a^hdm_\diamond\\
&+\int P_{t,\vartheta(a)}^*(\dot u_a)\varsigma_a^hdm_\diamond
-\int \omega_a^h\dot\vartheta(a)\dot f_{\vartheta(a)}e^{-f_{\vartheta(a)}}dm_\diamond) da,
 \end{align*}
where
\begin{align*}
 \varsigma_a^h=e^{-f_{\vartheta(a_{i-1})}}e_\varepsilon'(u_{a_{i-1},\vartheta}), &\text{ for }a\in(a_{i-1},a_i]\\
 \omega_a^h=e_\varepsilon(u_{a_i,\vartheta}), &\text{ for }a\in(a_{i-1},a_i].
\end{align*}
Letting $h\to0$ we obtain
\begin{align*}
  \varsigma_a^h\to e^{-f_{\vartheta(a)}}e_\varepsilon'(u_{a,\vartheta}), &\text{ in }L^1(X)\text{ for a.e. }a\in(0,1)\\
 \omega_a^h\to e_\varepsilon(u_{a,\vartheta}), &\text{ in }L^1(X)\text{ for a.e. }a\in(0,1),
\end{align*}
and thus from dominated convergence
  \begin{align*}
  &S^\varepsilon_t(\rho_{1,\vartheta})-S^\varepsilon_s(\rho_{0,\vartheta})\\
\geq&\limsup_{h\to0}[\int_{0}^{1}(-\dot\vartheta(a)\langle\Delta_{\vartheta(a)}u_{a,\vartheta},\varsigma_a^h\rangle+\int\dot\vartheta(a)\dot f_{\vartheta(a)}u_{a,\vartheta}\varsigma_a^hdm_\diamond\\
&+\int P_{t,\vartheta(a)}^*(\dot u_a)\varsigma_a^hdm_\diamond
-\int \omega_a^h\dot\vartheta(a)\dot f_{\vartheta(a)}e^{-f_{\vartheta(a)}}dm_\diamond) da]\\
\geq&\limsup_{h\to0}[\int_{0}^{1}(-\dot\vartheta(a)\langle\Delta_{\vartheta(a)}u_{a,\vartheta},\varsigma_a^h\rangle da]\\
&+\int_0^1(\int\dot\vartheta(a)\dot f_{\vartheta(a)}u_{a,\vartheta}e^{-f_{\vartheta(a)}}e_\varepsilon'(u_{a,\vartheta})dm_\diamond\\
&+\int P_{t,\vartheta(a)}^*(\dot u_a)e^{-f_{\vartheta(a)}}e_\varepsilon'(u_{a,\vartheta})dm_\diamond
-\int e_\varepsilon(u_{a,\vartheta})\dot\vartheta(a)\dot f_{\vartheta(a)}e^{-f_{\vartheta(a)}}dm_\diamond) da.
 \end{align*}
 To see that $\langle\Delta_{\vartheta(a)}u_{a,\vartheta},\varsigma_a^h\rangle\to 
 \langle\Delta_{\vartheta(a)}u_{a,\vartheta},e^{-f_{\vartheta(a)}}e_\varepsilon'(u_{a,\vartheta})\rangle$, recall that from Theorem \ref{energy-est} it suffices to show that
 \begin{align*}
  \varsigma_a^h\to e^{-f_{\vartheta(a)}}e_\varepsilon'(u_{a,\vartheta})\text{ in }L^2(X).
 \end{align*}
 This is a consequence of the boundedness of $u_{a,\vartheta}$ and $f_{\vartheta(a)}$.
 Then again by dominated convergence we have
  \begin{align*}
  &S^\varepsilon_t(\rho_{1,\vartheta})-S^\varepsilon_s(\rho_{0,\vartheta})\\
\geq&\int_{0}^{1}[\dot\vartheta(a)\mathcal E_{\vartheta(a)}^\diamond(u_{a,\vartheta},e^{-f_{\vartheta(a)}}e_\varepsilon'(u_{a,\vartheta}))
+\int\dot\vartheta(a)\dot f_{\vartheta(a)}u_{a,\vartheta}e^{-f_{\vartheta(a)}}e_\varepsilon'(u_{a,\vartheta})dm_\diamond\\
&+\int P_{t,\vartheta(a)}^*(\dot u_a)e^{-f_{\vartheta(a)}}e_\varepsilon'(u_{a,\vartheta})dm_\diamond
-\int e_\varepsilon(u_{a,\vartheta})\dot\vartheta(a)\dot f_{\vartheta(a)}e^{-f_{\vartheta(a)}}dm_\diamond] da\\
=&\int_{0}^{1}[\dot\vartheta(a)\mathcal E_{\vartheta(a)}(u_{a,\vartheta},e_\varepsilon'(u_{a,\vartheta}))
+\int\dot\vartheta(a)\dot f_{\vartheta(a)}u_{a,\vartheta}e_\varepsilon'(u_{a,\vartheta})dm_{\vartheta(a)}\\
&+\int P_{t,\vartheta(a)}^*(\dot u_a)e_\varepsilon'(u_{a,\vartheta})dm_{\vartheta(a)}
-\int e_\varepsilon(u_{a,\vartheta})\dot\vartheta(a)\dot f_{\vartheta(a)}dm_{\vartheta(a)}] da.
 \end{align*}
\end{proof}

\subsection{The Dynamic EVI$^-$-Property}

\begin{proposition}\label{theoremregularcurves1}
Let $\rho^a=u^am_t$ be a regular curve. Then setting $\rho_{a,\vartheta}=\hat P_{t,\vartheta(a)}\rho^a$, it holds
 \begin{equation}
\begin{aligned}\label{actest}
\frac12\tilde W_{\vartheta}^2(\rho_{1,\vartheta},\rho_{0,\vartheta})-(t-s)(S_t(\rho_{1,\vartheta})-S_s(\rho_{0,\vartheta}))\\
\leq \frac12 \int_0^1|\dot\rho_a|_t^2da-(t-s)^2\int_0^1\int\dot f_{\vartheta(a)}d\rho_{a,\vartheta}da.
\end{aligned}
\end{equation}
\end{proposition}
\begin{proof}
Applying Lemma \ref{actionder} and Lemma \ref{derent}, we find 
\begin{equation}
\begin{aligned}
&\int\varphi_1d\rho_{1,\vartheta}-\int\varphi_0d\rho_{0,\vartheta}-(t-s)(S_t^\varepsilon(\rho_{1,\vartheta})-S_s^\varepsilon(\rho_{0,\vartheta}))\\
&\leq \int_0^1\Big[\int\dot u_a P_{t,\vartheta(a)}(\varphi_a-(t-s)e_\varepsilon'(u_{a,\vartheta}))dm_t\\
&-\frac12\int\Gamma_{\vartheta(a)}(\varphi_a)d\rho_{a,\vartheta}+(t-s)\int\Gamma_{\vartheta(a)}(\varphi_a,u_{a,\vartheta})dm_{\vartheta(a)}-4(t-s)^2\int e_\varepsilon''(u_{a,\vartheta})\Gamma_{\vartheta(a)}(\sqrt{u_{a,\vartheta}})d\rho_{a,\vartheta}\\
&-(t-s)^2\int(e_\varepsilon(u_{a,\vartheta})-e_\varepsilon'(u_{a,\vartheta})u_{a,\vartheta})\dot f_{\vartheta(a)}d m_{\vartheta(a)}\Big]da.
\end{aligned}
\end{equation}
Then since 
\begin{align*}
4re_\varepsilon''(r)\geq 4r^2(e_\varepsilon''(r))^2=r(p_\varepsilon'(\sqrt{r}))^2,
\end{align*}
we can estimate
\begin{align*}
-4u_{a,\vartheta}e_\varepsilon''(u_{a,\vartheta})\Gamma_{\vartheta(a)}(\sqrt{u_{a,\vartheta}})\leq -u_{a,\vartheta}(p_\varepsilon'(\sqrt{u_{a,\vartheta}}))^2\Gamma_{\vartheta(a)}(\sqrt{u_{a,\vartheta}})
=-u_{a,\vartheta}\Gamma_{\vartheta(a)}(p_\varepsilon(\sqrt{u_{a,\vartheta}})),
\end{align*}
and while, with $q_\varepsilon(r):=\sqrt{r}(2-\sqrt{r}p_\varepsilon'(\sqrt{r}))$,
\begin{align*}
\Gamma_{\vartheta(a)}(u_{a,\vartheta},\varphi_a)=2\sqrt{u_{a,\vartheta}}\Gamma_{\vartheta(a)}(\sqrt{u_{a,\vartheta}},\varphi_a)=u_{a,\vartheta}\Gamma_{\vartheta(a)}(p_\varepsilon(\sqrt{u_{a,\vartheta}}),\varphi_a)+q_\varepsilon(u_{a,\vartheta})\Gamma_{\vartheta(a)}(\sqrt{u_{a,\vartheta}},\varphi_a)
\end{align*}
we find
\begin{equation}
\begin{aligned}
&\int\varphi_1d\rho_{1,\vartheta}-\int\varphi_0d\rho_{0,\vartheta}-(t-s)(S_t^\varepsilon(\rho_{1,\vartheta})-S_s^\varepsilon(\rho_{0,\vartheta}))\\
&\leq \int_0^1\Big[\int\dot u_a P_{t,\vartheta(a)}(\varphi_a-(t-s)e_\varepsilon'(u_{a,\vartheta}))dm_t\\
&-\frac12\int\Gamma_{\vartheta(a)}(\varphi_a)d\rho_{a,\vartheta}+(t-s)\int\Gamma_{\vartheta(a)}(\varphi_a,p_\varepsilon(\sqrt{u_{a,\vartheta}}))d\rho_{a,\vartheta}-(t-s)^2\int\Gamma_{\vartheta(a)}(p_\varepsilon(\sqrt{u_{a,\vartheta}}))d\rho_{a,\vartheta}\\
&+(t-s)\int q_\varepsilon(u_{a,\vartheta})\Gamma_{\vartheta(a)}(\sqrt{u_{a,\vartheta}},\varphi_a)dm_{\vartheta(a)}-(t-s)^2\int(e_\varepsilon(u_{a,\vartheta})-e_\varepsilon'(u_{a,\vartheta})u_{a,\vartheta})\dot f_{\vartheta(a)}d m_{\vartheta(a)}\Big]da.
\end{aligned}
\end{equation}
Hence, by means of \eqref{oscillation}, the gradient estimate \eqref{eq:Bakry}, and Young inequality $2xy\leq \delta x^2+y^2/\delta$ this yields 
\begin{equation*}
\begin{aligned}
&\int\varphi_1d\rho_{1,\vartheta}-\int\varphi_0d\rho_{0,\vartheta}-(t-s)(S^\varepsilon_t(\rho_{1,\vartheta})-S^\varepsilon_s(\rho_{0,\vartheta}))\\
&\leq \int_0^1\Big[\frac12|\dot \rho_a|^2_t+\frac12\int\Gamma_t(P_{t,\vartheta(a)}(\varphi_a-(t-s)e_\varepsilon'(u_{a,\vartheta}))d\rho_a\\
&-\frac12\int P_{t,\vartheta(a)}\Gamma_{\vartheta(a)}(\varphi_a-(t-s)p_\varepsilon(\sqrt{u_{a,\vartheta}}))d\rho_{a}
+(t-s)\int q_\varepsilon(u_{a,\vartheta})\Gamma_{\vartheta(a)}(\sqrt{u_{a,\vartheta}},\varphi_a)dm_{\vartheta(a)}\\
&-(t-s)^2\int(e_\varepsilon(u_{a,\vartheta})-e_\varepsilon'(u_{a,\vartheta})u_{a,\vartheta})\dot f_{\vartheta(a)}d m_{\vartheta(a)}\Big]da\\
&\leq \int_0^1\Big[\frac12|\dot \rho_a|^2_t+
+(t-s)\int |q_\varepsilon(u_{a,\vartheta})||\Gamma_{\vartheta(a)}(\sqrt{u_{a,\vartheta}},\varphi_a)|dm_{\vartheta(a)}\\
&-(t-s)^2\int(e_\varepsilon(u_{a,\vartheta})-e_\varepsilon'(u_{a,\vartheta})u_{a,\vartheta})\dot f_{\vartheta(a)}d m_{\vartheta(a)}\Big]da\\
&\leq \int_0^1\Big[\frac12|\dot \rho_a|^2_t
+\frac{(t-s)}{2\delta}\int (q_\varepsilon(u_{a,\vartheta}))^2\Gamma_{\vartheta(a)}(\varphi_a)dm_{\vartheta(a)}+\frac{(t-s)\delta}{2}\int\Gamma_{\vartheta(a)}(\sqrt{u_{a,\vartheta}})dm_{\vartheta(a)}\\
&-(t-s)^2\int(e_\varepsilon(u_{a,\vartheta})-e_\varepsilon'(u_{a,\vartheta})u_{a,\vartheta})\dot f_{\vartheta(a)}d m_{\vartheta(a)}\Big]da.
\end{aligned}
\end{equation*}
We first pass to the limit $\varepsilon\to0$,
\begin{align*}
\lim_{\varepsilon\to0}q_\varepsilon^2(r)=0, \quad q_\varepsilon^2(r)=4r(1-\frac{r}{\varepsilon+r})^2
\leq 4r,
\end{align*}
\begin{align*}
\lim_{\varepsilon\to0}(e_\varepsilon(r)-re_\varepsilon'(r))&=-r,\\
|e_\varepsilon(r)-re_\varepsilon'(r)|
&\leq 2(\varepsilon+r)|\log(\varepsilon+r)|+r+\varepsilon\log\varepsilon
\leq 2\sqrt{\varepsilon+r}+r+\varepsilon\log\varepsilon,
\end{align*}
and then, $\delta\to0$, 
\begin{align*}
&\int\varphi_1d\rho_{1,\vartheta}-\int\varphi_0d\rho_{0,\vartheta}-(t-s)(S_t(\rho_{1,\vartheta})-S_s(\rho_{0,\vartheta}))\\
&\leq \int_0^1\Big[\frac12|\dot \rho_a|^2_t
+(t-s)^2\int \dot f_{\vartheta(a)}d\rho_{a,\vartheta}\Big]da.
\end{align*}
Taking the supremum over $\varphi$ we obtain the desired estimate \eqref{actest}.
\end{proof}

\begin{theorem}\label{theorem:evi-}
Assume that the gradient estimate holds true for the time-dependent metric measure space $(X,d_t,m_t)_{t\in(0,T)}$. 
Then for every $\mu\in\Dom(S)$ and every $\tau\in(0,T]$
the dual heat flow $\mu_t:=\hat P_{\tau,t}\mu$ emanating in $\mu$ we have 
\begin{equation}\label{evi-int}
\begin{aligned}
S_s(\mu_s)-S_t(\sigma)
\leq \frac1{2(t-s)}(W_t^2(\mu_t,\sigma)- W_{s,t}^2(\mu_s,\sigma))-(t-s)\int_0^1\int\dot f_{\vartheta(a)}d\rho_{a,\vartheta}da
\end{aligned}
\end{equation}
for all $s\in(0,\tau)$ and all  $\sigma,\mu\in\Dom(S)$. Here $(\rho_a)_{a\in[0,1]}$ denotes the $W_t$-geodesic connecting $\rho_0=\mu_t$, $\rho_1=\sigma$ and 
$\rho_{a,\vartheta}=\hat P_{t,\vartheta(a)}(\rho_a)$.

In particular $\mu_t$
is a dynamic upward EVI$^-$-gradient flow, i.e. for 
every $t\in(0,\tau)$ and every $\sigma\in\Dom(S)$ we have
\begin{equation*}
\frac12 \partial_s^- W_{s,t}^2(\mu_s,\sigma)_{|s=t-}\geq S_t(\mu_t)-S_t(\sigma).
\end{equation*}
\end{theorem}

\begin{proof}
Let $(\rho_a)_{a\in[0,1]}$ be a $W_t$-geodesic connecting $\mu_t$ and $\sigma$, which exists and is unique.
We approximate the geodesic $(\rho_a)_{a\in[0,1]}$ by regular curves $(\rho_a^n)_{a\in[0,1]}$.
Proposition \ref{theoremregularcurves1} states that for each $(\rho_a^n)_{a\in[0,1]}$
\begin{equation}
\begin{aligned}
\frac12\tilde W_{\vartheta}^2(\rho_{1,\vartheta}^n,\rho_{0,\vartheta}^n)-(t-s)(S_t(\rho_{1,\vartheta}^n)-S_s(\rho_{0,\vartheta}^n))\\
\leq \frac12 \int_0^1|\dot\rho_a^n|_t^2da-(t-s)^2\int_0^1\int\dot f_{\vartheta(a)}d\rho_{a,\vartheta}^nda.
\end{aligned}
\end{equation}
Since for every $a\in[0,1]$ $\rho_a^n$ converges to $\rho_a$ in duality with bounded continuous functions, 
$\rho_{a,\vartheta}^n$ converges to $\rho_{a,\vartheta}$ in duality with bounded continuous functions as well. 
By virtue of Lemma \ref{duallsc} we obtain
\begin{align*}
\liminf_{n\to\infty}\tilde W^2_{\vartheta}(\rho_{1,\vartheta}^n,\rho_{0,\vartheta}^n)&\geq \tilde W^2_{\vartheta}(\rho_{1,\vartheta},\rho_{0,\vartheta}).
\end{align*}
Note that $(\rho_a^n)$ also converges to $\rho_a$ in duality with $L^\infty$ functions, since Lemma \ref{regularcurves} provides
$\sup_n S_t(\rho_a^n)<\infty$.
The same argument applies then to 
$\rho_{a\vartheta}^n$.
Hence
\begin{align*}
\lim_{n\to\infty}\int\dot f_{\vartheta(a)}d\rho_{a,\vartheta}^n&=\int\dot f_{\vartheta(a)}d\rho_{a,\vartheta}.\\
\end{align*}
Then we end up with
\begin{equation}
\begin{aligned}\label{actest2}
\frac12\tilde W^2_{\vartheta}(\mu_s,\sigma)-(t-s)(S_t(\sigma)-S_s(\mu_s))\\
\leq \frac12 W_t^2(\mu_t,\sigma)-(t-s)^2\int_0^1\int\dot f_{\vartheta(a)}d\rho_{a,\vartheta}da.
\end{aligned}
\end{equation}

Applying Corollary \ref{cordual} we obtain
\begin{equation*}
\begin{aligned}
(t-s)(S_s(\mu_s)-S_t(\sigma))
\leq \frac12W_t^2(\mu_t,\sigma)-\frac12 W_{s,t}^2(\mu_s,\sigma)-(t-s)^2\int_0^1\int\dot f_{\vartheta(a)}d\rho_{a,\vartheta}da.
\end{aligned}
\end{equation*}

Dividing by $t-s$ and letting $s\nearrow t$ we find
\begin{equation*}
\begin{aligned}
S_t(\mu_t)-S_t(\sigma)
&\leq \liminf_{s\nearrow t}\frac1{2(t-s)}\left(W_t^2(\mu_t,\sigma)- W_{s,t}^2(\mu_s,\sigma)\right)\\
&=\frac12 \partial_s^- W_{s,t}^2(\mu_s,\sigma)_{|s=t-}.
\end{aligned}
\end{equation*}
\end{proof}

\subsection{Summarizing}

The precise integrated version \eqref{evi-int} of the EVI$^-$-property indeed also implies a relaxed version of the EVI$^+$-property which then in turn allows 
to prove uniqueness of dynamic EVI-flows for the entropy.

\begin{corollary}\label{evi+heat}
The gradient estimate {\bf(III)} implies the {\bf EVI}$^+(-2L,\infty)$-property. More precisely, for every $\mu\in\Dom(S)$ and every $\tau\le T$ 
the dual heat flow $\mu_t:=\hat P_{t,\tau}\mu$ emanating in $\mu$ satisfies
\begin{equation*}
\frac12 \partial_s^- W_{s,t}^2(\mu_s,\sigma)_{|s=t}\geq S_t(\mu_t)-S_t(\sigma)- L\, W_t^2(\mu_t,\sigma)
\end{equation*}
for all $t<\tau$ and all $\sigma\in\Pz(X)$.
\end{corollary}

\begin{proof}
Given $\mu_t:=\hat P_{t,\tau}\mu$  for $t\tau$, consider \eqref{evi-int} for fixed $s<\tau$ and with $s\searrow t$. Then
\begin{eqnarray*}
S_s(\mu_s)-S_s(\sigma)&=&
\lim_{s\searrow t}S_s(\mu_s)-S_t(\sigma)\\
&\le& \lim_{s\searrow t}\frac1{2(t-s)}\Big[ W_t^2(\mu_t,\sigma)-W_{s,t}^2(\mu_s,\sigma)\Big]\\
&\le& \Big(\lim_{s\searrow t}\frac1{2(t-s)}\Big[ W_{t,s}^2(\mu_t,\sigma)-W_{s}^2(\mu_s,\sigma)\Big]\\
&&\qquad\qquad\ +\frac L2\Big[ W_t^2(\mu_t,\sigma)+W_{s}^2(\mu_s,\sigma)\Big]\Big)\\
&=&\frac12 \partial_t^- W_{t,s}^2(\mu_t,\sigma)_{t=s+}+ L\, W_s^2(\mu_s,\sigma)
\end{eqnarray*}
where the last estimate follows from \eqref{dst-dt}.
\end{proof}

\begin{corollary}
 Assume that  {\bf(III)}  holds true and that $(\mu_t)_{t\in(\sigma,\tau)}$ is a dynamic upward EVI$^-$- or EVI$^+$gradient flow for $S$ emanating in some 
 $\mu\in\Pz$. Then
\begin{equation*}
\mu_t=\hat P_{t,\tau}\mu
\end{equation*}
for all $t\in(\sigma,\tau)$. 
That is, the dual heat flow is the unique dynamic backward EVI$^-$-flow for the Boltzmann entropy.
\end{corollary}

\begin{proof}
Corollary \ref{evi-uni}
together with Corollary \ref{evi+heat}
and Theorem \ref{theorem:evi-}.
\end{proof}

\begin{theorem}\label{von III nach I}
The gradient estimate  {\bf(III$_N$)} 
implies the dynamic $N$-convexity of the Boltzmann entropy {\bf(I$_N$)}. 
\end{theorem}

\begin{proof}
According to Theorem \ref{from III to II}
and Theorem \ref{theorem:evi-}
the gradient estimate  {\bf(III$_N$)} implies both
\begin{itemize}
\item the transport estimate {\bf(II$_N$)} and
\item the {\bf EVI}$^-(0,\infty)$-property
\end{itemize}
According to Theorem \ref{dyn-Nconv} and Remark \ref{remark:dyn-Nconv},
both properties together imply dynamic $N$-convexity.
\end{proof}

\newpage

\newpage

\section{Appendix}\label{appendix}

\subsection{Time-dependent Geodesic Spaces}

For this chapter, 
 our basic setting will be a
space $X$ equipped with a 1-parameter family of complete geodesic metrics
$(d_t)_{ t\in I}$ where $I\subset\R$ is a bounded open interval, say for convenience $I=(0,T)$.
(More generally, one might allow $d_t$ to be pseudo metrics  where the existence of connecting geodesics is only requested for pairs $x,y\in X$ with 
$d_t(x,y)<\infty$.) 
We always request 
 that there exists a  constant
 $L\in\R$  (`log-Lipschitz bound') such that
\begin{equation}
\left| \log\frac{d_t(x,y)}{d_s(x,y)}\right|\le L\cdot |t-s|
\end{equation}
for all $s,t$ and all $x,y$ (`log Lipschitz continuity in $t$');

\medskip

Let us first introduce a natural `distance' on $I\times X$.
\begin{definition}\label{ddist}
Given $s,t\in I$ and $x,y\in X$ we put
\begin{equation}
d_{s,t}(x,y):=\inf\left\{\int_0^1|\dot\gamma^a|_{s+a(t-s)}^2da\right\}^{1/2}
\end{equation}
where the infimum runs over all absolutely continuous curves $(\gamma^a)_{a\in[0,1]}$ in $X$ connecting $x$ and $y$.

\end{definition}

\begin{proposition}\label{d-diff-2}
(i)
The infimum in the above formula is attained. Each minimizer $(\gamma^a)_{a\in[0,1]}$ is a curve of constant speed, i.e.
$|\dot\gamma^a|_{s+a(t-s)}=d_{s,t}(x,y)$
for all $a\in [0,1]$.

(ii) A point $z\in X$ lies on some minimizing curve $\gamma$ with $z=\gamma^a$ if and only if 
$$d_{s,t}(x,y)=d_{s,r}(x,z)+d_{r,t}(z,y)$$
with $r=s+a(t-s)$.

(iii) For all $s,t\in I$ and $x,y\in X$ 
\begin{equation*}
 \frac{1-e^{-L|t-s|}}{L|t-s|} 
\le
\frac{d_{s,t}(x,y)}{d_s(x,y)}\le \frac{e^{L|t-s|}-1}{L|t-s|}.
\end{equation*}
Thus in particular,
\begin{equation}\label{dst-dt}
 \Big|
\partial_t d_{s,t}(x,y)\big|_{t=s} \Big|\le \frac L2 d_s(x,y) .
\end{equation}

(iv)
For all $s<t\in I$ and $x,y\in X$ 
\begin{equation}
d_{s,t}(x,y)=\lim_{\delta\to0} \ \inf_{(t_i,x_i)_i}\left\{\sum_{i=1}^k\frac{t-s}{t_i-t_{i-1}}\, d^2_{t_i}\big( x_i,x_{i-1}\big)\right\}^{1/2}
\end{equation}
where the infimum runs over all $k\in\N$. all partitions $(t_i)_{i=0,\ldots,k}$ of $[s,t]$ with $t_0=s, t_k=t$  and $|t_i-t_{i-1}|\le\delta$ as well as over 
all $x_i\in X$ with $x_0=x, x_k=y$.
\end{proposition}

\begin{proof}
(i)  For each absolutely continuous curve $(\gamma^a)_{a\in[0,1]}$
\begin{equation*}
\left(\int_0^1|\dot\gamma^a|_{s+a(t-s)}^2da\right)^{1/2}
\ge
\int_0^1|\dot\gamma^a|_{s+a(t-s)}da
\end{equation*}
with equality if and only if the curve has constant speed.

(ii) Restricting the minimizing curve for $d_{s,t}$ to parameter intervals $[0,a]$ and $[a,1]$  provides upper estimates for $d_{s,r}(x,z)$ and $d_{r,t}(z,y)$,  resp., and thus yields the  ``$\ge$''-inequality. Conversely, given any pair of minimizers for $d_{s,r}(x,z)$ and $d_{r,t}(z,y)$ by concatenation a curve connecting $x$ and $y$ can be constructed with action bounded by the scaled action of the two ingredients. This proves the  ``$\le$''-inequality.

(iii) The log-Lipschitz continuity of the distance implies that for each absolutely continuous curve
\begin{equation*}
 e^{-La|t-s|}
\int_0^1|\dot\gamma^a|_{s}da
\le
\int_0^1|\dot\gamma^a|_{s+a(t-s)}da
\le e^{La|t-s|}
\int_0^1|\dot\gamma^a|_{s}da.
\end{equation*}

(iv) see section \ref{secdynamickant} for the argument in the case of $W_{s,t}$.
\end{proof}

\subsection{EVI Formulation of Gradient Flows}

 For the subsequent discussion,  a 
 lower semi-bounded function  $V: I\times X\to (-\infty,\infty]$ will be given 
with 
 $V_s(x)\le C_0\cdot V_t(x)+C_1$ for all $s,t\in I$ and $x\in X$ (thus, in particular,
$ \Dom(V)=\{x\in X: \ V_t(x)<\infty\}$ is independent of $x$)
 and  such that 
 for each $t\in I$ the function
$x\mapsto V_t(x)$ is  $\kappa$-convex along each $d_t$-geodesic (for some $\kappa\in\R$).
We also assume that minimizing $d_t$-geodesics between pairs of points in $\Dom(V)$ are unique.

In previous chapters, the following results  will be applied 
\begin{itemize}
\item
to the Boltzmann entropy $S_t$ on the time-dependent geodesic space $(\Pz,W_t)_{t\in I}$ as well as 
\item
to the Dirichlet energy $\E_t$ on  the time-dependent geodesic space $L^2(X,m_t)_{t\in I}$
\end{itemize}
in the place of the function $V_t$ on the time-dependent geodesic space $(X,d_t)_{t\in I}$.

\begin{definition}
Given a left-open interval $J\subset I$,
an absolutely continuous curve $(x_t)_{t\in J}$ will be called \emph{dynamic backward
 EVI$^-$-gradient flow} for $V$ 
 if for all $t\in J$ and all   $z\in \Dom(V_t)$ 
  \begin{align}\label{evi-dyn}
     {\frac12\partial_s^- d^2_{s,t}(x_s,z)}\Big|_{s=t-}
       ~\ge V_t(x_t)-V_t(z)
  \end{align}
  where $d_{s,t}$ is  defined  in Definition \ref{ddist}.
  
  A curve $(x_t)_{t\in J}$ with a right-open interval $J\subset I$ will be called \emph{dynamic backward
 EVI$^+$-gradient flow} for $V$ if instead
  \begin{align*}
     {\frac12\partial_s^- d^2_{s,t}(x_s,z)}\Big|_{s=t+}
       ~\ge V_t(x_t)-V_t(z)
  \end{align*}
  for all $t\in J$.
  
  It is called  \emph{dynamic backward 
 EVI-gradient flow} if it is both, a  \emph{dynamic backward
 EVI$^+$-gradient flow} and a  \emph{dynamic backward
 EVI$^-$-gradient flow}.
 
 We say that the backward gradient flow 
   $(x_t)_{t\in J}$   \emph{emanates} in $x'\in X$ if $\lim_{t\nearrow \sup J}x_t=x'$.
  \end{definition}

 Being a dynamic backward EVI$^\pm$-gradient flow for $V$ obviously implies that $x_t\in\Dom(V_t)$ for all $t<\tau$.
  
  \begin{remark*}  Note that these definitions are slightly different from a previous one presented in \cite{sturm2015}.
  If $d_s$ depends smoothly on $s$ then
  $$ \partial_s^- d^2_{s,t}(x_s,z)\big|_{s=t-}=
  \partial_s^- d^2_{t}(x_s,z)\big|_{s=t-}+
  \partial_s^- d^2_{s,t}(x_t,z)\big|_{s=t-}$$
  and always 
  $\partial_s^- d^2_{s,t}(x_t,z)\big|_{s=t-}\ge {\mathfrak b}_t^0(\gamma)$ for any $d_t$-geodesic  $\gamma$  connecting $x_t$ and $z$.
  \end{remark*}
  
  Often, we ask for an improved notion of dynamic backward EVI-gradient flows, involving parameters $N\in (0,\infty]$
  (regarded as an upper bound for the `dimension') and/or $K\in\R$ (regarded as a lower bound for the `curvature'). The choices $N=\infty$ and $K=0$ will yield the previous concept.
 
\begin{definition}
We say that  an absolutely continuous curve $(x_t)_{t\in(\sigma,\tau)}$ is a  \emph{dynamic backward
 EVI$(K,N)$-gradient flow} for $V$ if for   all $z\in \Dom(V_t)$ and all    $t\in (\sigma,\tau)$ 
  \begin{align}\label{evi-dyn-N}
     {\frac12\partial_s^- d^2_{s,t}(x_s,z)}\Big|_{s=t}-\frac K2\cdot d^2_{t}(x_t,z)
       ~\ge V_t(x_t)-V_t(z)+\frac1N\int_0^1 \Big(\partial_a V_t(\gamma^a)\Big)^2(1-a)da
  \end{align}
  where $\gamma$ denotes the $d_{t}$-geodesic connecting $x_t$ and $z$.

Analogously, we define   \emph{dynamic backward
 EVI$^\pm(K,N)$-gradient flows} for $V$. 
 
 In the case, $K=0$, dynamic backward
 EVI$(K,N)$-gradient flows will be simply called \emph{dynamic backward
 EVI$_{N}$-gradient flows}.
  \end{definition}
  
  The concept of `backward' gradient flows is tailor-made for our later application to the dual heat flow. This flow is running backward in time and on its way it tries to minimize the Boltzmann entropy. Regarded in positive time direction, it follows the `upward gradient' of the entropy.
  
  On the other hand, in calculus of variations mostly the `downward' gradient flow will be considered where a curve tries to follow the negative gradient of a given functional.

\begin{definition}
We say that  an absolutely continuous curve $(x_t)_{t\in(\sigma,\tau)}$ is a  \emph{dynamic forward
 EVI$(K,N)$-gradient flow} for $V$ if for   all $z\in \Dom(V_t)$ and all    $t\in (\sigma,\tau)$ 
  \begin{align}\label{evi-dyn-N-forward}
     -{\frac12\partial_s^+ d^2_{s,t}(x_s,z)}\Big|_{s=t}-\frac K2\cdot d^2_{t}(x_t,z)
       ~\ge V_t(x_t)-V_t(z)+\frac1N\int_0^1 \Big(\partial_a V_t(\gamma^a)\Big)^2(1-a)da
  \end{align}
  where $\gamma$ denotes the $d_{t}$-geodesic connecting $x_t$ and $z$. 
  
  We say that a forward gradient flow emanates in a given point $x'\in X$ if $\lim_{t\searrow \sigma}x_t=x'$.
  \end{definition}
  
  We will formulate all our results for `backward' gradient flows and leave it to the reader to carry them over to the case of `forward' gradient flows.
  
  \begin{lemma}\label{evi-mono-lem}
  For each dynamic backward EVI$^\pm(K,\infty)$-gradient flow $(x_t)_{t\in(\sigma,\tau)}$ for $V$
  $$\int_\sigma^\tau V_t(x_t)dt<\infty.$$
  \end{lemma}
 
 \begin{proof}
 Choose $z\in\Dom(V)$, apply the EVI$(K,\infty)$-property  at time $t$, and then integrate w.r.t.\ time $t$
  \begin{eqnarray*}
  \int_\sigma^\tau V_t(x_t)dt&\le&
  \int_\sigma^\tau\Big[ V_t(z) +\frac12\partial_s d^2_{s,t}(x_s,z)\big|_{s=t}-\frac K2 d_t^2(x_t,z)\Big]dt\\
  &\le&(C_0\, V_\tau(z)+C_1)(\tau-\sigma)+\frac12 \int_\sigma^\tau\Big[\partial_t d^2_{t}(x_t,z)+(L-K)\, d_t^2(x_t,z)\Big]dt\\ 
  &=&(C_0\, V_\tau(z)+C_1)(\tau-\sigma)+
\frac12 d^2_{\tau}(x_\tau,z)-
\frac12 d^2_{\sigma}(x_\sigma,z)+
  \frac{L-K}2 \int_\sigma^\tau d_t^2(x_t,z)dt.
  \end{eqnarray*}
  Obviously, the right hand side is finite which thus proves the claim.
  \end{proof}

\subsection{Contraction Estimates}

\begin{theorem}\label{evi-mono} Given two curves $(x_t)_{t\in(\sigma,\tau)}$  and  $(y_t)_{t\in(\sigma,\tau)}$, one of which is an is a {dynamic backward
 EVI$^-(K,N)$-gradient flow} for $V$ and the other  is a {dynamic backward
 EVI$^+(K,N)$-gradient flow} for $V$, then for all $\sigma< s<t< \tau$
 \begin{equation}
 d_s^2(x_s,y_s)\le e^{-2K(t-s)}\cdot d_t^2(x_t,y_t)-\frac2N \int_s^te^{-2K(r-s)}\cdot\Big|V_r(x_r)-V_r(y_r)\Big|^2dr.
 \end{equation}
\end{theorem}

\begin{proof}
Assume that the curve $(x_t)_{t\in(\sigma,\tau]}$ is a {dynamic backward
 EVI$^-$-gradient flow} for $V$ and  $(y_t)_{t\in(\sigma,\tau]}$ is a {dynamic backward
 EVI$^+$-gradient flow} for $V$. It implies that $r\mapsto d_r(x_r,y_r)$ is absolutely continuous since
 $$|d_t(x_t,y_t)-d_s(x_s,y_s)|\le d_{s}(x_s,x_t)+ d_{s}(y_s,y_t)+L(t-s)d_t(x_t,y_t).$$
 Thus by the very definition of EVI flows
 \begin{eqnarray*}
d^2_t(x_t,y_t)- d^2_s(x_s,y_s)&=&
 \limsup_{\delta\searrow0}\Big[\frac1\delta\int_{t-\delta}^td_r^2(x_r,y_r)\,dr- \frac1\delta\int_{s}^{s+\delta}d_r^2(x_r,y_r)\,dr\Big]\\
 &=&
 \limsup_{\delta\searrow0}\frac1\delta\int_{s+\delta}^t\Big[d_r^2(x_r,y_r)-d^2_{r-\delta}(x_{r-\delta},y_{r-\delta})\Big]\,dr\\
&\ge& \liminf_{\delta\searrow0}
\frac1\delta \int_{s+\delta}^t\big[d^2_r(x_r,y_r)-d^2_{r,r-\delta}(x_r,y_{r-\delta})\big]\,dr\\
&&
+\liminf_{\delta\searrow0}
\frac1\delta\int_{s+\delta}^t \big[d^2_{r,r-\delta}(x_r,y_{r-\delta})-d^2_{r-\delta}(x_{r-\delta},y_{r-\delta})\big]\,dr\\
&=& \liminf_{\delta\searrow0}
\frac1\delta \int_{s+\delta}^t\big[d^2_r(x_r,y_r)-d^2_{r,r-\delta}(x_r,y_{r-\delta})\big]\,dr\\
&&
+\liminf_{\delta\searrow0}
\frac1\delta\int_{s}^{t-\delta} \big[d^2_{r+\delta,r}(x_{r+\delta},y_r)-d^2_{r}(x_{r},y_{r})\big]\,dr\\
&\stackrel{(\ast)}\ge&
\int_s^t \liminf_{\delta\searrow0}
\frac1\delta \big[d^2_r(x_r,y_r)-d^2_{r,r-\delta}(x_r,y_{r-\delta})\big]\,dr\\
&&
+\int_s^t \liminf_{\delta\searrow0}
\frac1\delta \big[d^2_{r+\delta,r}(x_{r+\delta},y_r)-d^2_{r}(x_{r},y_{r})\big]\,dr\\
&\ge&
2 \int_s^t\Big[\frac K2 d_r^2(x_r,y_r)+V_r(y_r)-V_r(x_r)+\frac1N\int_0^1\Big(\partial_a V_r(\gamma_r^a)\Big)^2a\, da\Big] dr\\
 &&+
2 \int_s^t \Big[\frac K2 d_r^2(x_r,y_r)+V_r(x_r)-V_r(y_r)
 +\frac1N\int_0^1\Big(\partial_a V_r(\gamma_r^a)\Big)^2(1-a)\, da\Big] dr\\
 &=&
 2K \int_s^td_r^2(x_r,y_r)\,dr+
 \frac2N\int_s^t\int_0^1\Big(\partial_a V_r(\gamma_r^a)\Big)^2 da\, dr\\
 &\ge&2K \int_s^td_r^2(x_r,y_r)\,dr+
 \frac2N \int_s^t\Big|V_r(x_r)-V_r(y_r)\Big|^2dr.
 \end{eqnarray*}
 Dividing by $t-s$ and passing to the limit $t-s\searrow0$ yields
  \begin{eqnarray*}
\partial_t d^2_t(x_t,y_t)\ge 2K d_t^2(x_t,y_t)+
 \frac2N \Big|V_t(x_t)-V_t(y_t)\Big|^2
\end{eqnarray*}
for a.e.\ $t$.
The claim now follows via `variation of constants'.

It remains to justify the interchange of $\liminf_{\delta\searrow0}$ and $\int\ldots dr$ in $(\ast)$ which requires quite some effort.
Recall from Proposition \ref{d-diff-2} that $|\frac{d^2_{s,t}(x,y)}{d^2_s(x,y)}-1|\le 2L\cdot |t-s|$ for all $x,y,s,t$ with $|t-s|\le \frac1L$. 
Thus we can estimate 
\begin{eqnarray*}
\lefteqn{
-\frac1\delta\Big[d^2_r(x_r,y_r)-d^2_{r,r-\delta}(x_r,y_{r-\delta})\Big]}\\
&\le& -\frac1\delta\Big[ d^2_r(x_r,y_r)-d^2_{r-\delta}(x_r,y_{r-\delta})\Big]+o_1
\\
&=&-\frac1\delta\int_{r-\delta}^r\partial_sd_s^2(x_r,y_s)\,ds+o_1
\\
&\le&
-\frac1\delta\int_{r-\delta}^r\partial_td_{s,t}^2(x_r,y_t)\Big|_{t=s}\,ds+o_1+o_2
\\
&\le&\frac2\delta\int_{r-\delta}^r\Big[V_s(x_r)-V_s(y_s)\Big]ds+o_1+o_2+o_3
\\
&\le& 2C_0\cdot V_r(x_r)+2C_1+C+o_1+o_2+o_3
\end{eqnarray*}
where for the last inequality we used the growth estimate of $s\mapsto V_s(x)$ and the lower boundedness of $V$ and where we put
with
$o_1(r,\delta)=2L\,d_r^2(x_r,y_{r-\delta})$,
$o_2(r,\delta)=2L\, \frac1\delta\int_{r-\delta}^rd_r^2(x_r,y_{\sigma})\,d\sigma$,
$o_3(r)=K\, d_r^2(x_r,y_r)$.
Continuity of $r\mapsto d_r$ and  of $r\mapsto x_r$ as well as of  $r\mapsto y_r$ imply that for any fixed $z\in X$ the function
$r\mapsto d_r^2(x_r,z)$ is bounded  as well as $r\mapsto d_r^2(y_{r-\delta},z)$ for $r\in (s,t)$, uniformly in $\delta\in (0,1)$. Thus $o_1(r,\delta)+o_2(r,\delta)+o_3(r,\delta)\le C'$ which finally justifies the interchange of  limit and integral.

Similarly, we can estimate
\begin{eqnarray*}
\lefteqn{
-\frac1\delta\Big[d^2_{r+\delta,r}(x_{r+\delta},y_r)-d^2_{r}(x_r,y_{r})\Big]}\\
&\le&-\frac1\delta\int^{r+\delta}_r\partial_sd_s^2(x_s,y_r)\,ds+o_1'
\\
&\le& 2C_0\cdot V_r(y_r)+2C_0+C+o_1'+o_2'+o_3'.
\end{eqnarray*}
In both cases, the final expression is integrable w.r.t.\ $r\in[s,t]$  according to Lemma \ref{evi-mono-lem} 
since by assumption $V_t(x_t)<\infty$ as well as $V_t(y_t)<\infty$.
\end{proof}

\begin{corollary} \label{evi-uni}
Assume that $(x_t)_{t\in(\sigma,\tau)}$    is a {dynamic backward
 EVI$(K,N)$-gradient flow} for $V$ 
 and  that $(y_t)_{t\in(\sigma,\tau)}$   is a dynamic backward
 EVI$^-(K,N)$- or
 EVI$^+(K,N)$-gradient flow for $V$ emanating in the same point $x_\tau=y_\tau$. Then 
 $$x_t=y_t$$
 for all $t\le \tau$.
\end{corollary}

  \begin{corollary}
  Assume that for given $\tau$, a dynamic upward EVI$(K,\infty)$-gradient flow terminating in $x'$ exists for each $x'$ in a dense subset  $D\subset X$. Then this flow can be extended to a flow  terminating in any $x'\in X$ and satisfying
   \begin{equation}
 d_s(x_s,y_s)\le e^{-K(t-s)}\cdot d_t(x_t,y_t)
 \end{equation}
 for any $s<t\le \tau$.
  \end{corollary}

\subsection{Dynamic Convexity}
Let us recall the notion of dynamic convexity as introduced in \cite{sturm2015}.

\begin{definition}
We say that
the  function $V: I\times X\to (-\infty,\infty]$  is \emph{strongly dynamically $(K,N)$-convex} if for a.e.\ $t\in I$ and for every $d_t$-geodesic $(\gamma^a)_{a\in[0,1]}$ with $\gamma^0,\gamma^1\in\Dom(V_t)$ 
  \begin{equation}\label{N-dyn-conv2}
\partial^+_a V_t(\gamma_t^{1-})- \partial^-_a V_t(\gamma_t^{0+})\ge -\frac 12\partial_t^- d_{t-}^2(\gamma^0,\gamma^1)+
\frac K2 d^2_t(\gamma^0,\gamma^1)+
\frac1N \left| V_t(\gamma^0)-V_t(\gamma^1)\right|^2.
\end{equation}
\end{definition}

\begin{theorem}\label{dyn-Nconv}
Assume that for each $t\in I$ and each $x'\in \Dom(V_t)$ there exists a dynamic backward  EVI$(K,N)$-gradient flow $(x_s)_{s\in(\sigma,t]}$ for $V$ emanating  in $x'$ and such that $\lim_{s\nearrow t}V_s(x_s)=V_t(x_t)$. Then $V$ is strongly dynamically $(K,N)$-convex.
\end{theorem}

\begin{remark}\label{remark:dyn-Nconv}
To be more precise, we request the inequality \eqref{evi-dyn} at the  point $t$ and the inequality \eqref{evi-dyn-N} at all times before $t$.
\end{remark}

\begin{proof}
Fix $t\in I$ and a $d_t$-geodesic  $(\gamma^a)_{a\in[0,1]}$ with $\gamma^0,\gamma^1\in\Dom(V_t)$. The a priori assumption of $\kappa$-convexity implies $\gamma^a\in\Dom(V_t)$ for all $a\in [0,1]$.  For each $a$, let $(\gamma^a_s)_{s\le t}$ denote the EVI$_N$-gradient flow  for $V$ emanating in $\gamma^a=\gamma^a_t$.
Then for all $a\in(0,\frac12)$
\begin{eqnarray*}
V_t(\gamma^a)-V_t(\gamma^0)&\le&\frac12\partial_s^-d_{s,t}^2(\gamma_s^a,\gamma^0)\Big|_{s=t-}\\
&\le&\frac12\partial_s^-d_{s}^2(\gamma_s^a,\gamma^0)\Big|_{s=t-}+a^2 L \, d_t^2(\gamma^0,\gamma^1)
\end{eqnarray*}
(due to the log-Lipschitz continuity of $s\mapsto d_s$) and
\begin{eqnarray*}
V_t(\gamma^{1-a})-V_t(\gamma^1)&\le&\frac12\partial_s^-d_{s,t}^2(\gamma_s^{1-a},\gamma^1)\Big|_{s=t-}\\
&\le&\frac12\partial_s^-d_{s}^2(\gamma_s^{1-a},\gamma^1)\Big|_{s=t-}+a^2 L \, d_t^2(\gamma^0,\gamma^1).
\end{eqnarray*}
Moreover, the previous Theorem \ref{evi-mono} implies
\begin{eqnarray*}
0&\le&
\liminf_{s\nearrow t}\frac1{t-s}\Big[\frac12d_{t}^2(\gamma^{a},\gamma^{1-a})-
\frac12d_{s}^2(\gamma_s^{a},\gamma^{1-a}_s)
-K\, d^2_t(\gamma^a,\gamma^{1-a})
-\frac1N\int_s^t
\Big| V_r(\gamma^a_r)-V_r(\gamma^{1-a}_r)\Big|^2dr \\
 &=& 
\frac12\partial_s^-d_{s}^2(\gamma_s^{a},\gamma^{1-a}_s)\Big|_{s=t-}
-K\, d^2_t(\gamma^a,\gamma^{1-a})
-\frac1N
\Big| V_t(\gamma^a)-V_t(\gamma^{1-a})\Big|^2 .
\end{eqnarray*}
(Here we used the requested continuity $V_r(\gamma^a_r)\to V_t(\gamma^a)$ for $r\nearrow t$.)

Adding up these inequalities (the last one multiplied by $\frac1{1-2a}$ and the previous ones by $\frac1a$) yields
\begin{eqnarray*}
\lefteqn{\frac1a\Big[V_t(\gamma^a)-V_t(\gamma^0)+V_t(\gamma^{1-a})-V_t(\gamma^1)\Big]}\\
&\le&\liminf_{s\nearrow t}\frac1{2(t-s)}\Big(\big[\frac1ad_t^2(\gamma^0,\gamma^a)+\frac1{1-2a}d_t^2(\gamma^a,\gamma^{1-a})+\frac1ad_t^2(\gamma^{1-a},\gamma^1)\big]\\
&&\qquad\qquad-
\big[\frac1ad_s^2(\gamma^0,\gamma^a_s)+\frac1{1-2a}d_s^2(\gamma^a_s,\gamma^{1-a}_s)+\frac1ad_s^2(\gamma^{1-a}_s,\gamma^1)\big]\Big)\\
&&\quad+2aL\, d^2_t(\gamma^0,\gamma^1)
-\frac{K}{1-2a} d^2_t(\gamma^a,\gamma^{1-a})-\frac1{N(1-2a)}\Big| V_t(\gamma^a)-V_t(\gamma^{1-a})\Big|^2\\
&\le&\liminf_{s\nearrow t}\frac1{2(t-s)}\Big(d_t^2(\gamma^0,\gamma^1)-d_s^2(\gamma^0,\gamma^1)\Big)\\
&&\quad-\big[(1-2a)K-2aL\big]\cdot d^2_t(\gamma^0,\gamma^1)-\frac1{N(1-2a)}\Big| V_t(\gamma^a)-V_t(\gamma^{1-a})\Big|^2.
\end{eqnarray*}
In the limit $a\to0$ this yields the claim.
\end{proof}

\newpage

\bibliography{srf-KoSt-6}
\bibliographystyle{plain}

\end{document}